\documentclass{amsart}     

\usepackage{amssymb,latexsym,amsmath}
\usepackage{fancyhdr, lastpage} 
\usepackage{leftidx}
\usepackage{color}
\usepackage{amsaddr}

\usepackage{mathptmx}
\usepackage{times}

\begin{document} 


\newtheorem{theorem}{Theorem}[section]
\newtheorem{lemma}[theorem]{Lemma}
\newtheorem{corollary}[theorem]{Corollary}
\newtheorem{proposition}[theorem]{Proposition}
\theoremstyle{definition}
\newtheorem{definition}[theorem]{Definition}
\theoremstyle{remark}
\newtheorem{remark}[theorem]{Remark}

\newcommand\prf{\begin{proof}}
\newcommand\prfe{\end{proof}}
\newcommand\be{\begin{equation}}
\newcommand\ee{\end{equation}}
\newcommand\bea{\begin{eqnarray}}
\newcommand\eea{\end{eqnarray}}
\newcommand\beas{\begin{eqnarray*}}
\newcommand\eeas{\end{eqnarray*}}

\newcommand{\R}{\mathbb R}
\newcommand{\N}{\mathbb N}
\newcommand{\T}{\mathbb T}
\newcommand{\K}{\mathbb S}
\newcommand{\leftexp}[2]{{\vphantom{#2}}^{#1}{#2}}

\newcommand{\todo}[1]{\vspace{5 mm}\par \noindent
\marginpar{\textsc{ \hspace{.2 in}   \textcolor{red}{ To Fix}}} \framebox{\begin{minipage}[c]{0.95
\textwidth} \tt #1 \end{minipage}}\vspace{5 mm}\par}

\newcommand\g{\partial}
\newcommand\E{\mathcal{ E}}
\newcommand\D{\mathcal{D}}
\renewcommand\t{\bar{\partial}} 
\renewcommand\div{\mbox{div}}
\newcommand\curl{\mbox{curl}}

\newcommand\qn{\open{q}}
\newcommand\bp{{\bar \partial}}
\newcommand\lessim{\lesssim}





\title{Global stability and decay for the classical Stefan problem}

\author{Mahir Had\v zi\'c and Steve Shkoller}
\address{{\large Accepted for publication in Communications on Pure and Applied Mathematics}}

\address{Department of Mathematics,
King's College London,
Strand, London, WC2R 2LS, UK}
\email{mahir.hadzic@kcl.ac.uk}
\address{Mathematical Institute, University of Oxford, Andrew Wiles Building, Radcliffe Observatory Quarter, Woodstock Road, Oxford, OX2 6GG, UK}
\email{shkoller@maths.ox.ac.uk}





\begin{abstract}
The classical one-phase Stefan problem describes the temperature distribution
in a homogeneous medium undergoing a phase transition, such as ice melting to water.
This is accomplished by solving the heat equation on a time-dependent domain
whose boundary is transported by the normal derivative of the temperature along
the evolving and a priori unknown free-boundary.    We establish a global-in-time stability result for nearly
spherical geometries and small temperatures, using a novel hybrid methodology,
which combines energy estimates, decay estimates, and Hopf-type inequalities.
\end{abstract}
\maketitle   



 \tableofcontents



\section{Introduction}
\subsection{The problem formulation}
We consider the problem of global existence and  asymptotic stability of classical solutions to the {\em classical} Stefan problem
describing the evolving free-boundary between the liquid and solid phases.
The temperature of the liquid $p(t,x)$ 
and the {\it a priori unknown}  moving phase boundary $\Gamma(t)$ must satisfy the
following system of equations:
\begin{subequations}
\label{eq:stefan}
\begin{alignat}{2}
p_t-\Delta p&=0&&\ \text{ in } \ \Omega(t)\,;\label{eq:heat}\\
\mathcal{V} (\Gamma(t))&=-\g_np  && \ \text{ on } \ \Gamma(t)\,;\label{eq:neumann}\\
p&=0&& \ \text{ on } \ \Gamma(t)\,;\label{eq:dirichlet}\\
p(0,\cdot)=p_0\,, & \ \Omega (0)=\Omega _0&&\,.\label{eq:initial}
\end{alignat}
\end{subequations}
For each instant of time $t \in [0,T]$, $\Omega(t)$ is a time-dependent open subset of $\R^d$ with $d\ge 2$, and $\Gamma(t):=\partial \Omega(t)$ denotes
the moving, time-dependent  free-boundary.

The heat equation (\ref{eq:heat})  models thermal diffusion
 in the bulk $\Omega(t)$ with thermal diffusivity set to $1$. The boundary transport equation  (\ref{eq:neumann})
states that each point on the moving boundary is transported with normal velocity equal to
 $-\g_np=-\nabla p\cdot n$,  the normal derivative of $p$ on $\Gamma(t)$.  Here
 $n$ denotes the outward pointing unit normal to $\Gamma(t)$,  and 
$ \mathcal{V} (\Gamma(t))$ denotes the speed or the normal velocity of the hypersurface $\Gamma(t)$.
The homogeneous Dirichlet boundary condition~(\ref{eq:dirichlet}) is termed the {\em classical Stefan condition}
and problem~(\ref{eq:stefan}) is called the {\em classical Stefan problem}. It implies that the
freezing of the liquid occurs at a constant temperature $p=0$. 
Finally, we  must specify the initial temperature distribution
$p_0:\Omega_0\to\R$, as well as the initial geometry  $ \Omega _0$.   Because the liquid phase $\Omega(t)$ is characterized by
the set $\{ x \in \mathbb{R}  ^d \ : \ p(x,t) >0 \}$, we shall consider initial data $p_0 >0$ in $\Omega_0$.
Problem~(\ref{eq:stefan}) belongs to the category of free boundary problems which are of {\it parabolic-hyperbolic} type.   Thanks to
 (\ref{eq:heat}), the parabolic Hopf lemma implies that 
$\g_np(t)<0$ on $\Gamma(t)$ for $t>0$,  so we impose the  {\it non-degeneracy condition} or so-called
{\em Taylor sign condition}\footnote{This type of stability condition dates back to the early work of Lord Rayleigh  \cite{Ray} and Taylor \cite{Tay} in fluid mechanics,
and appears as a necessary well-posedness condition on the initial data in many  free-boundary problems wherein
the effects of surface tension are ignored;
examples include the Hele-Shaw cell, the water waves equations \cite{Wu97}, and the full Euler equations in both 
incompressible~\cite{CoSh07} and compressible form~\cite{CoSh12, CoHoSh12}.   
}
\be\label{eq:tay}
-\g_np_0\ge \lambda > 0 \quad\text{on}\,\,\,\,\Gamma(0)
\ee
on our initial temperature distribution.  
Under the above assumptions, we proved in Had\v{z}i\'{c} \& Shkoller \cite{HaSh} that (\ref{eq:stefan}) is indeed well-posed.
\subsection{The reference domain $\Omega$ and the dimension}  
For our reference domain, we choose the unit ball in $ \mathbb{R}^2  $ given by
\[
\Omega=B(0,1) := \{x\in\R^2:\,|x|<1\} ,
\] 
with boundary  $\Gamma  = \K^1 :=\{x\in\R^2:\,|x|=1\}$.   We shall consider initial domains $\Omega_0$ whose boundary $\Gamma_0$
is a graph over the reference boundary $\Gamma$. In order to simplify our presentation, 
we consider evolving domains $\Omega(t)$ in $ \mathbb{R}  ^2$,
but as we shall explain in Section~\ref{se:general}, our methodology works equally well in any dimension $d \ge 2$.

Our choice of the reference domain $\Omega$ follows from two considerations.  First,  we need employ only
 {\it one} global coordinate system near the boundary $\Gamma$, rather than a collection of local
coordinate charts that a  more general domain would necessitate,  and the use of one coordinate system greatly simplifies
the presentation of  our energy identities, that provide very natural estimates for the second-fundamental
form of the evolving free-boundary $\Gamma(t)$.   Second,
 we shall need quantitative Hopf-type inequalities in order to bound the term defined in~(\ref{eq:tay})
from below, and such estimates are available in a particularly satisfying form in the case of the nearly spherical domains, thanks
to the explicit construction of comparison functions in \textsc{Oddson} \cite{Od}.   

\subsection{Notation}
%
%
For any $s\geq0$ and given functions $f:\Omega\to\R$, $\varphi:\Gamma\to\R$ we set
$$
\|f\|_s:=\|f\|_{H^s(\Omega)} \text{ and }
|\varphi|_s:=\|\varphi\|_{H^s(\Gamma)}.
$$
$H^s(\Omega)'$ shall denote the dual space of $H^s(\Omega)$, while on the boundary, $H^s(\Gamma )' = H^{-s}( \Gamma )$.
If $i=1,2$ then 
$f,_i:=\g_{x^i}f$ is the partial derivative of $f$ with respect to $x^i$. Similarly,
$f,_{ij}:=\g_{x^i}\g_{x^j}f$, etc.   For time-differentiation,
$f_t:=\g_tf$.   Furthermore, for a function $f(t,x)$,  we shall often write $f(t)$ for $f(t, \cdot )$, and
$f(0)$ to mean $f(0,x)$.
We use $\t:= \tau \cdot \nabla $ to denote the tangential derivative, so that
\[
\t f:=\g_{\theta}f,\quad\t^kf:=\g_{\theta}^kf \,,
\]
where $ \theta \in [0,2\pi)$ denotes the angular component in polar coordinates.
The Greek letter $\alpha$ will often be reserved for multi-indices $\alpha=(\alpha_1,\alpha_2)$,
with $\g^{\alpha}:=\g_{x_1}^{\alpha_1}\g_{x_2}^{\alpha_2}$ and
$|\alpha|=\alpha_1+\alpha_2$.
The identity map on $\Omega$ is denoted by $e(x)=x$, while the identity matrix is denoted by $\text{Id}$.
 We use $C$ to denote a universal (or generic) constant that may change from inequality to inequality.
We write $X\lesssim Y$ to denote 
$X\leq C\, Y$. 
We use the notation $P(\cdot)$ to denote generic real polynomial function of its argument(s) with positive coefficients. 
The Einstein summation convention is employed, indicating   summation over repeated indices.
The $L^2$-inner product on $\Omega$ is denoted by $(\cdot,\cdot)_{L^2}$. 
\subsection{Fixing the domain}   
We transform the Stefan problem~(\ref{eq:stefan}), set on the moving domain $\Omega(t)$,  to an
equivalent problem on the fixed domain $\Omega$.  For many problems in fluid dynamics, the  Lagrangian flow map of the fluid
velocity provides a natural family of diffeomorphisms which can be used to fix the domain, but for the classical Stefan problem, we
use instead (in the parlance of fluid dynamics) the so-called Arbitrary Lagrangian-Eulerian (ALE) family of diffeomorphisms; these ALE maps 
interpolate between the Lagrangian and Eulerian representations of the equations.   For this problem, we choose a simple type of ALE map,
consisting of {\it harmonic coordinates}, also known as the harmonic gauge.

\subsubsection{The diffeormorphism $\Psi(t)$}
We  represent our moving domain $\Omega(t)$  as the image of a time-dependent family of diffeomorphisms
$\Psi(t) :  \Omega \to \Omega(t)$.   In order to define these diffeomorphisms, we let
$h(t,\cdot):\Gamma\to\R$ denote  the {\it signed height function} whose graph (over $\Gamma$) is the set $\Gamma (t)$. For $\xi\in \Gamma=\K^1$, we define
the map 
$$\Psi(t,\xi)=(1+h(t,\xi))\xi=R(t,\xi)\xi$$ 
which is a diffeomorphism of $\Gamma$ onto $\Gamma(t)$ as long as $h(t)$ remains a graph.
The outward-pointing  unit normal  vector $n(t,\cdot )$ to the moving surface $\Gamma(t)$ is defined by
\[
(n \circ \Psi)(t, \xi) =(R\xi)_{\theta}^{\perp}/|(R\xi)_{\theta}^{\perp}| \,.
\]
We shall henceforth drop the explicit composition with the diffeomorphism $\Psi$, and simply write
$$ n(t, \xi) =(R\xi)_{\theta}^{\perp}/|(R\xi)_{\theta}^{\perp}|  $$
for the unit normal to the moving boundary at the point $\Psi(t, \xi) \in \Gamma(t)$.

Introducing  the unit normal and tangent vectors to the reference surface $\Gamma$ as
\be\label{eq:normaltangent}
N:=\xi, \ \ \tau:=\xi_{\theta} \quad \text{ or equivalently } \quad N(\theta) = ( \cos \theta, \sin \theta), \quad \tau(\theta) = (-\sin \theta , \cos \theta ) \,,
\ee
we write the unit normal to $\Gamma(t)$ as 
\begin{equation}\label{ssss77}
n(t,\xi)=\frac{RN - h_{\theta}\tau}{\sqrt{R^2+R_{\theta}^2}}\,.
\end{equation} 
The evolution of $h(t)$ is then given by 
\be\label{eq:velformula}
h_t = v \cdot N(\theta) - \frac{h_\theta}{R} v\cdot \tau(\theta)  \,.
\ee

Assuming that the signed height function $h(t, \cdot)$ is sufficiently regular and remains a graph, we can define a diffeomorphism $\Psi:\Omega\to\Omega(t)$
as the elliptic extension of the  boundary diffeomorphism $\xi \mapsto (1+ h( \xi , t)) \xi $, by solving the following Dirichlet problem
\begin{align} 
\Delta\Psi&=0 \ \text{ in } \ \Omega,\nonumber\\
\Psi(t,\xi)&=R(t,\xi)\xi \ \,\,\,\, \xi\in\Gamma.\label{eq:gauge}
\end{align} 
Since the identity map $e: \Omega \to \Omega$ is harmonic in $\Omega$ and $\Psi -e = h \xi$ on $\Gamma$, 
standard elliptic regularity theory for solutions to (\ref{eq:gauge}) shows that
\begin{equation}\label{DM}
\|\Psi-e\|_{H^s(\Omega)}\leq C\|h\|_{H^{s-0.5}(\Gamma)}\,, \ s>0.5,
\end{equation} 
so that for $h(t)$ sufficiently small and $s$ large enough,  the Sobolev embedding theorem shows that $\nabla \Psi(t)$  is close to the identity matrix $ \operatorname{Id} $, and by the inverse function theorem, each
$\Psi(t)$ is a
diffeomorphism.  

\subsubsection{The temperature and velocity variables on the fixed domain $\Omega$}\label{se:ALE}  
First we introduce the velocity variable $u = - \nabla p $ in $\Omega(t)$.  Next, we introduce the following new variables set on
the fixed domain $\Omega$:
\begin{alignat*}{2} 
          q&= p\circ \Psi\ && \text{(\textit{temperature})},\\
          v&=u\circ\Psi\ && \text{(\textit{velocity})},\\
          w&=\Psi_t\ && \text{(\textit{extension of boundary velocity vector})},\\
	A&= [D\Psi]^{-1}\ &&\text{(\textit{inverse of the deformation tensor})},\\
	J&= \det D\Psi\ &\quad&\text{(\textit{Jacobian determinant})},\\
	a&= JA\ && \text{(\textit{cofactor matrix of the deformation tensor})}. 
\end{alignat*}
The relation $u=-\nabla p$ is then written as
$v^i+A^k_iq,_k=0$ for $ i=1,2$.
By the chain rule, 
$$q_t = p_t \circ \Psi + (\nabla p\circ \Psi) \cdot \Psi_t = p_t \circ \Psi - v \cdot w \,,$$
and
$$
\Delta p\circ\Psi=\Delta_{\Psi} q:=A^j_i(A^k_iq,_k),_j \,.
$$
Letting $\tilde n = J^{-1}(R^2+R_{\theta}^2)^ {\frac{1}{2}}  n$,  we see that
\be\label{eq:ntilde}
\tilde{n}_i(t,x)=A^k_i(t,x) N_k(x) \,,
\ee
and equation (\ref{eq:velformula}) can be written  as $h_t = v \cdot \tilde n/ R_J$, where $R_J=RJ^{-1}$.
Note that $R_J=RJ^{-1}=(1+h)J^{-1}$ is very close to $1$.

\subsubsection{The classical Stefan problem set on the fixed domain $\Omega$}
The classical Stefan problem on  the fixed domain $\Omega$ 
is written as
\begin{subequations}
\label{eq:ALE}
\begin{alignat}{2}
q_t-A^j_i(A^k_iq,_k),_j&=-v\cdot \Psi_t&& \ \text{ in } \ (0,T] \times \Omega \,,\label{eq:ALEheat}\\
v^i+A^k_iq,_k&=0&&\ \text{ in } \ [0,T] \times \Omega\,,  \label{eq:ALEv} \\
q&=0&& \ \text{ on } \  [0,T] \times \Gamma \,,\label{eq:ALEdirichlet}\\
h_t &= v \cdot N - (1+h) ^{-1} h_\theta v\cdot \tau&& \ \text{ on } \  (0,T] \times \Gamma \,,  \label{eq:ALEneumann2}\\
\Delta \Psi &=0&& \ \text{ on } \ [0,T] \times \Omega\,,\label{eq:Psi_elliptic}\\
 \Psi &=(1+h) N&& \ \text{ on } \ [0,T] \times \Gamma \,,\label{eq:Psi_laplace}\\
q&=q_0>0&&\ \text{ on } \ \{t=0\} \times \Omega   \,,\label{eq:ALEinitial} \\
h&=h_0&&\ \text{ on } \ \{t=0\} \times \Gamma  \,,\label{eq:ALEinitialh}
\end{alignat}
\end{subequations}
where the initial boundary $\partial \Omega_0$ is given as a graph over $\Omega$ with the initial height function $h_0$, i.e.
$\partial\Omega_0=\{x\in\R^2, \ x=(1+h_0(\xi))\xi, \ \xi\in\K^1\}$.   
Note that $\Phi=\Psi(0):\Omega\to\Omega_0$  is a near identity transformation,  mapping the reference domain $\Omega$ onto the initial domain $\Omega_0$.
The initial temperature function $q_0$ equals $p_0 \circ \Phi$.
Problem~(\ref{eq:ALE}) is a  reformulation of the problem~(\ref{eq:stefan}). 

Henceforth, without loss of generality, we shall assume that the initial domain $\Omega_0$ {\em is} the unit ball $B_1(0)$ or in other
words $h_0=0$.  In this case, we set $\Phi=e$, where $e: \Omega \to \Omega$ is the identity map, and $\Psi(t)|_{t=0} = e$.
In Section~\ref{se:general}, we will explain the minor modification required when $h_0 \neq 0$,  as well as the case that the dimension $d=3$.

Observe that the boundary condition~(\ref{eq:ALEneumann2}) implies that
\be\label{eq:ALEneumann}
\Psi_t\cdot {n}(t)=v\cdot {n}(t) \ \text{ on } \ [0,T] \times \Gamma  \ \text{ so that } \ \Psi(t)(\Gamma) = \Gamma(t) \,.
\ee
\subsubsection{The energy and dissipation functions}\label{se:norms}
Near $\Gamma=\partial \Omega $,  it is convenient to use tangential derivatives 
$\t:=\g_{\theta}$ with $\theta$ denoting the polar angle, while near the origin, Cartesian partial derivatives $\partial_{x_i}$ are
natural. 
For this reason,  we introduce a  non-negative $C^{\infty}$  cut-off function $\mu:\bar{\Omega}\to\R_+$ with the
property
\[
\mu(x)\equiv0\,\,\,\,\text{ if }\,\,|x|\leq1/2;\qquad\mu(x)\equiv1\,\,\,\,\text{ if }\,\,3/4\leq|x|\leq1.
\]

\begin{definition}[Higher-order norms]
The following high-order energy and dissipation functionals are fundamental to our analysis:
\begin{align} 
\E(t)&=\E(q,h)(t) \nonumber \\
& := \frac{1}{2}\sum_{a+2b\leq5}\|\mu^{1/2}\t^a\partial_t^bv\|_{L^2_x}^2
+\frac{1}{2}\sum_{b=0}^2|(-\g_Nq)^{1/2}R\, J^{-1} \, \t^{6-2b}\g_t^bh|_{L^2_x}^2
+\frac{1}{2}\sum_{a+2b\leq6}\|\mu^{1/2}\big(\t^a\partial_t^bq+\t^a\partial_t^b\Psi\cdot v\big)\|_{L^2_x}^2 \nonumber \\
& + \sum_{|\vec{\alpha}|+2b\leq5}\|(1-\mu)^{1/2}\g_{\vec{\alpha}}\partial_t^bv\|_{L^2_x}^2
+\frac{1}{2}\sum_{|\vec{\alpha}|+2b\leq6}\|(1-\mu)^{1/2}\big(\g_{\vec{\alpha}}\partial_t^bq+\g_{\vec{\alpha}}\partial_t^b\Psi\cdot v\big)\|_{L^2_x}^2
\label{eq:ALEenergy}
\end{align} 
and
\begin{align} 
\D(t)& =\D(q,h)(t) \nonumber \\
&:=
\sum_{a+2b\leq6}\|\mu^{1/2}\t^a\partial_t^bv\|_{L^2_x}^2
+\sum_{b=0}^2|(-\g_Nq)^{1/2}R\, J^{-1} \, \t^{5-2b}\g_t^bh_t|_{L^2_x}^2
+\sum_{a+2b\leq5}\|\mu^{1/2}\big(\t^a\partial_t^bq_t+\t^a\partial_t^b\Psi_t\cdot v\big)\|_{L^2_x}^2 \nonumber \\ 
&+\sum_{|\vec{\alpha}|+2b\leq6}\|(1-\mu)^{1/2}\g_{\vec{\alpha}}\partial_t^bv\|_{L^2_x}^2
+\sum_{|\vec{\alpha}|+2b\leq5}\|(1-\mu)^{1/2}\big(\g_{\vec{\alpha}}\partial_t^bq_t+
\g_{\vec{\alpha}}\partial_t^b\Psi_t\cdot v\big)\|_{L^2_x}^2\,.
\label{eq:ALEdissipation}
\end{align} 
\end{definition} 
 
Note that the boundary norms of the height function are weighted by
$\sqrt{-\g_N q}$.
We thus introduce the time-dependent function 
\[
\chi(t):=\inf_{x\in\Gamma}(-\g_Nq)(t,x)>0,
\]
which will be used to track the weighted behavior of $h$.
We will show that $\E$ is indeed equivalent to 
$$\sum_{l=0}^3\|\g_{t}^lq\|_{H^{6-2l}(\Omega)}^2
+\chi(t)\sum_{l=0}^3|\g_t^lh|_{H^{6-2l}(\Gamma)}^2 \,,
$$
and that $\D$ is equivalent to
$$ \|q\|_{H^{6.5}(\Omega)}^2+
\sum_{l=0}^{2}\|\g_t^lq_t\|_{H^{5-2l}(\Omega)}^2
+\chi(t)\sum_{l=0}^2|\g_{t}^{l+1}h|_{H^{5-2l}(\Gamma)}^2.
$$

The elliptic operator in the parabolic equation (\ref{eq:ALE}a) for $q$ has coefficients that depend on
$A = [D \Psi ] ^{-1} $, which in turn depend on $h$; hence, the regularity of $q$ is limited (and, in fact, determined)
by the regularity of $h$ on the boundary $\Gamma$.     Since the regularity of $h$ is given by norms which are
 weighted by the
factor $\chi(t)$, a naive application of  elliptic estimates would thus lead to the crude bound
\begin{equation}\label{crude1}
\|\g_t^lq\|_{6.5-2l}^2\lesssim\frac{\D}{\chi(t)}\,,
\end{equation} 
which could a priori grow in time.
However, by using the fact that lower-order norms of $q$ have exponential decay (in time), 
estimate (\ref{crude1})  can be improved to yield
\be\label{eq:improved}
\|q\|_{6.5}^2 + \|\g_t^lq_t\|_{5-2l}^2\lesssim e^{-\gamma t}\E+\D,\quad l=0,\dots,2
\ee
for some {\em positive} constant $\gamma>0$. 
This is one of the essential ingredients of our analysis,
as (\ref{eq:improved}) will be used to 
control error terms arising from higher-order energy estimates in Section~\ref{se:long}.
\par
In order to capture the exponential  decay of the temperature $q$, we 
introduce the {\it lower-order} decay norms:
\be\label{eq:ebeta}
E_{\beta}(t):=e^{\beta t}\Big(\sum_{b=0}^2\|\g_t^bq(t)\|_{H^{4-2b}(\Omega)}^2
+\sum_{b=0}^1\|\g_t^bv\|_{H^{3-2b}(\Omega)}^2\Big)
,\quad D(t):=\sum_{b=0}^2\|\g_t^bq(t)\|_{H^{5-2b}(\Omega)}^2,
\ee
with the constant $\beta$ denoting  a positive real number given by
\be\label{eq:beta}
\beta:=2\lambda_1-\eta,
\ee
where $\lambda_1$ is the smallest eigenvalue of the Dirichlet-Laplacian 
on $\Omega=B_1(0)$ 
and $\eta$  is a small positive constant related to the size of the initial data, which will be made precise below.
Note that the smallness of $E_{\beta}$ in particular implies an exponential 
decay (in time) estimate for the
$H^4$-norm of the temperature $q(t)$.

\subsubsection{Taylor sign condition or non-degeneracy condition on $q_0$}
With respect to $q_0 = p_0 \circ \Phi$, condition (\ref{eq:tay}) becomes $\inf_{x\in\Gamma}[-\g_Nq_0(x)]  \ge \lambda > 0$ on $\Gamma$.
For initial temperature distributions that are not necessarily strictly positive in $\Omega$, this condition was shown to be necessary
for local well-posedness for (\ref{eq:stefan}) (see \cite{HaSh, Me, PrSaSi}).  On the other hand, if we require strict
 positivity of our initial temperature function\footnote{Condition (\ref{eq:positive}) is natural, since it determines the phase: $\Omega(t) = \{ q(t) >0\}$.}, 
\be\label{eq:positive}
q_0>0\quad\text{ in }\Omega\,,
\ee
then the parabolic Hopf lemma (see, for example, \cite{Fr64})  guarantees that $-\g_N q(t,x)>0$ for $0<t<T$ on some a priori (possibly small)
time interval, which, in turn, shows that $\E$ and $\D$ are norms for $t>0$, but uniformity may be lost as $t\to 0$.  To ensure a uniform lower-bound
for $-\g_N q(t)$ as $t\to 0$,  we impose the Taylor sign condition with the following lower-bound\footnote{When $h_0\neq 0$, the unit normal to
the initial surface $\Gamma_0$ is given by 
$N=\frac{(1+h_0) \xi  - \partial_\theta h_0 \tau}{\sqrt{(1+h_0)^2+\partial_\theta h_0^2}}$ where $ \xi = ( \cos \theta, \sin\theta)$ and $\tau = (-\sin \theta, \cos \theta)$.}:
\be\label{eq:wlog}
-\g_Nq_0 \ge C \int_\Omega q_0 \, \varphi_1 dx \,,
\ee
Here, $\varphi_1$ is the positive first eigenfunction of the Dirichlet-Laplacian, and $C>0$ denotes a {\em universal} constant. The uniform
lower-bound in (\ref{eq:wlog}) thus ensures that our solutions are continuous in time; moreover, (\ref{eq:wlog}) allows us to  establish a 
time-dependent {\it optimal lower-bound} for the quantity
$
\chi(t)=\inf_{x\in\Gamma}(-\g_Nq)(t,x) >0
$
for all time $t\geq0$, which will be crucial for our analysis.

\subsubsection{Compatibility conditions}

The definition of our higher-order energy function $\E$, restricted to time $t=0$,
requires an explanation of the time-derivates of $q$ and $h$ evaluated  at $t=0$.   Specifically,
the values $q_t|_{t=0}$, 
$q_{tt}|_{t=0}$, $h_t|_{t=0}$ and $h_{tt}|_{t=0}$ are defined via
space-derivatives using equations~(\ref{eq:ALEheat}) and~(\ref{eq:ALEneumann2}).
To ensure that the solution is continuously differentiable in time  at $t=0$ we must impose  
compatibility conditions on the initial data (such conditions are, of course, only necessary for regular initial data).
By restricting the equation~(\ref{eq:ALEheat}) to the boundary at time $t=0$ and using the fact that $q_t(0) =0$ on $\Gamma$ and that $A^k_i|_{t=0}=\delta^k_i$, where
$\delta^k_i$ denotes the Kronecker delta which equals $1$ if $k=i$ and $0$ otherwise,
we obtain the first-order compatibility condition
\be\label{eq:comp1}
\Delta q_0=(\g_Nq_0)^2 \quad\text{on}\,\,\,\Gamma.
\ee
Upon differentiating~(\ref{eq:ALEheat}) with respect to time,  and then restricting  to $\Gamma$ at $t=0$ and using~(\ref{eq:comp1}),
we arrive at the second-order compatibility condition
\be\label{eq:comp2}
\Delta ^2 q_0 = \Delta | \g_Nq_0|^2 + 2 \g_N ( \Delta q_0 -  |\g_N q_0|^2) \, \partial_N q_0-2| \partial_{NN} q_0|^2
\quad\text{on}\,\,\,\Gamma,
\ee
where we have used that $h_t (t,\theta) = v \cdot  [N(\theta) -\tau(\theta) h_\theta (1 + h) ^{-1} ]$.

We note that our functional framework only requires specification of two higher-order compatibility conditions (the condition $q_0=0$ on $\Gamma$ being the zeroth-order
condition).   

\subsubsection{Main result}
Our main result is a global-in-time stability theorem for solutions of
the classical Stefan problem  for surfaces which are nearly spherical
and for temperature fields close to zero. 
The notion
of ``near" is measured by our energy norms as well as the dimensionless quantity
\be\label{eq:k}
K:=\frac{\| q_0\|_4}{\|q_0\|_0}.
\ee
as expressed in the following 
\begin{theorem}\label{th:main}
Let $(q_0,h_0)$ satisfy the 
Taylor sign condition~(\ref{eq:wlog}), the strict positivity 
assumption~(\ref{eq:positive}), and the compatibility conditions~(\ref{eq:comp1}),~(\ref{eq:comp2}). 
Let $K$ be defined as in~(\ref{eq:k}).
Then there exists an $\epsilon_0>0$ and a monotonically 
increasing function $F:(1, \infty )\to\R_+$, such that if
\be\label{eq:smallassum}
\E(q_0,h_0)<\frac{\epsilon_0^2}{F(K)},
\ee
then there exist unique solutions $(q,h)$ to  problem~(\ref{eq:ALE})
satisfying
\[
\sup_{0\leq t\leq\infty}\E(q(t),h(t))<C\epsilon_0^2,
\]
for some universal constant $C>0$.
Moreover, the temperature $q(t) \to 0$ as $t \to \infty $ with bound
\[
\|q\|_{H^4(\Omega)}^2\leq C e^{-\beta t},
\]
where $\beta=2\lambda_1-O(\epsilon_0)$ and $\lambda_1$ is the smallest eigenvalue of the 
Dirichlet-Laplacian on the unit disk. 
The moving boundary $\Gamma(t)$ settles asymptotically to some nearby steady surface 
$\bar{\Gamma}$ and we have the uniform-in-time estimate
\[
\sup_{0\leq t<\infty}|h-h_0|_{4.5} \lessim \sqrt{ \epsilon _0}
\]
\end{theorem}
\begin{remark}\label{re:fk}
The increasing function $F(K)$ given in (\ref{eq:smallassum}) has an explicit form.  For generic constants $\bar{C},C>1$
chosen in Sections~\ref{se:long} and~\ref{se:main} below, 
\be\label{eq:F}
F(K):=\max\{8K^{2C\bar{C}K^2},\bar{C}^{10}(\ln K)^{10}K^{20\bar{C}\lambda_1}\}.
\ee
\end{remark}
\begin{remark} The use of 
the constant $K$  in our smallness assumption~(\ref{eq:smallassum}) allows us to determine a time $T = T_K$
when the dynamics of the Stefan problem become strongly dominated by the  projection of  $q$ onto the first eigenfunction $\varphi_1$
of the Dirichlet-Laplacian.
Explicit knowledge of the $K$-dependence in the smallness assumption~(\ref{eq:smallassum}) permits the use of energy estimates 
to show that solutions exist in 
our energy space on the time-interval $[0,T_K]$.  For $t\ge T_K$, 
certain error terms (that cannot be controlled by our energy and dissipation functions for large $t$) become
 {\em sign-definite} with a good sign.
\end{remark}
\subsection{A brief history of prior results on the Stefan problem}
There is a large amount of literature on the classical one-phase Stefan problem. For an  
overview we refer the reader to \textsc{Friedman}~\cite{Fr82}, \textsc{Meirmanov}~\cite{Me} and \textsc{Visintin}~\cite{Vi}. 
First,  {\em weak} solutions were defined by \textsc{Kamenomostskaya}~\cite{Ka}, \textsc{Friedman}~\cite{Fr68}, and \textsc{Ladyzhenskaya, Solonnikov
\& Ural'ceva}~\cite{LaSoUr}.  For the one-phase problem studied herein, a variational formulation
was introduced by \textsc{Friedman \& Kinderlehrer}~\cite{FrKi75}, wherein additional regularity results for the
free surface were obtained.
\textsc{Cafarelli}~\cite{Ca77} showed that in some space-time neighborhood of points $x_0$ on the free-boundary  that have   Lebesgue density,
the boundary is  $C^1$ in both space and time, and  second derivatives of temperature
are continuous up to the boundary.  Under some regularity assumptions on the temperature, Lipschitz regularity of the free boundary was shown by \textsc{Cafarelli}~\cite{Ca78}.
In related work, \textsc{Kinderlehrer \& Nirenberg}~\cite{KiNi77, KiNi78} showed that the free boundary is analytic in space and of second Gevrey class in time,
under the a priori assumption that the free boundary is $C^1$  with certain assumptions on the temperature function.
In~\cite{CaFr79}, \textsc{Caffarelli \& Friedman} showed the continuity of the temperature in $d$ dimensions.
As for the two-phase classical Stefan problem, the continuity of the temperature in $d$ dimensions for  weak solutions
was shown by \textsc{Caffarelli \& Evans}~\cite{CaEv}.

Since the Stefan problem satisfies a maximum principle,
its analysis is ideally suited to  another type of weak solution called the {\it viscosity solution}. Regularity of
viscosity solutions for the two-phase Stefan problem was established by \textsc{Athanasopoulos, Caffarelli \& Salsa} in a 
series of seminal papers~\cite{AtCaSa1, AtCaSa2}. 
Existence of viscosity solutions for the one-phase problem was established by \textsc{Kim}~\cite{Ki}, and for the
two-phase problem by \textsc{Kim \& Po\v zar}~\cite{KiPo}.  A local-in-time
regularity result was established by \textsc{Choi \& Kim}~\cite{ChKi}, where it was shown that initially Lipschitz free-boundaries become
$C^1$ over a possibly smaller spatial region. 
For  an exhaustive 
overview and introduction to the regularity theory of viscosity solutions we refer the reader 
to \textsc{Caffarelli \& Salsa}~\cite{CaSa}. 
In~\cite{Ko98}, \textsc{Koch} showed by the use of von Mises variables and 
harmonic analysis, that an priori $C^1$ free-boundary in the two-phase problem becomes smooth.

Local existence of {\em classical} solutions for the classical Stefan problem was established by 
\textsc{Meirmanov} (see~\cite{Me} and references therein) and \textsc{Hanzawa}~\cite{Ha}. 
Meirmanov regularized the problem by adding artificial viscosity to~(\ref{eq:neumann}) and fixed
the moving domain by switching to the so-called von Mises variables, obtaining solutions with less
Sobolev-regularity than the  initial data.  Similarly, Hanzawa
used  Nash-Moser iteration  to construct a local-in-time solution, but again, 
 with derivative loss.
A local-in-time existence result for the one-phase multi-dimensional Stefan problem was proved  by  \textsc{Frolova \& Solonnikov} \cite{FrSo}, 
using $L^p$-type Sobolev spaces.   For the two-phase
Stefan problem,  a local-in-time existence result for classical solutions  was established by \textsc{Pr\"{u}ss, Saal, \& Simonett} \cite{PrSaSi} in the framework
of $L^p$-maximal regularity theory.

In a related work, local existence for the two-dimensional two-phase Muskat problem (with varying viscosity and density) 
was proved by \textsc{C\'ordoba, C\'ordoba \& Gancedo}~\cite{CoCoGa} and in three dimensions in~\cite{CoCoGa10}.
Their methods rely on a boundary-integral formulation for the Muskat problem, together with the Taylor sign condition.
In a subsequent work~\cite{CoCoGaSt}, various global existence results were established. An overview can be found in~\cite{CaCoGa}.

As to the Stefan problem with surface tension (also known as the Stefan problem with Gibbs-Thomson correction),
global weak solutions (without uniqueness) were given  by \textsc{Almgren \& Wang}, \textsc{Luckhaus}, and \textsc{R\"oger} \cite{AlWa,Lu,Roe}.
In  \textsc{Friedman \& Reitich}~\cite{FriRe} the authors considered the
Stefan problem with small surface tension, i.e.  with $\sigma\ll1$, whereby~(\ref{eq:dirichlet})
is replaced by $v=\sigma\kappa$, $\kappa$ denoting mean curvature of the boundary.
Local existence of classical solutions was studied by \textsc{Radkevich}  \cite{Ra}; \textsc{Escher,  Pr\"uss, \& Simonett}
 \cite{EsPrSi} 
proved a local existence and uniqueness result for classical solutions 
under a smallness assumption on the initial height function
close to  the reference flat boundary. Global existence close to flat hyper-surfaces
was proved by \textsc{Had\v{z}i\'{c} \& Guo} in \cite{HaGu1},  and close to stationary spheres for the
two-phase problem in \textsc{Had\v{z}i\'{c}} \cite{Ha1} and \textsc{Pr\"{u}ss, Simonett, \& Zacher} \cite{PrSiZa}.

In order to understand the asymptotic behavior of the classical Stefan problem on  {\em external} domains, \textsc{Quir\'os \& V\'azquez}~\cite{QuVa}
proved that on a complement of a given bounded domain $G$, with {\em non-zero} boundary conditions on the fixed boundary $\g G$, the solution
to the classical Stefan problem converges,  in a suitable sense,  to the corresponding solution of the Hele-Shaw problem and sharp global-in-time expansion rates 
for the expanding liquid blob are obtained. Moreover,  the blob asymptotically has the geometry of a ball.
Note that the non-zero boundary conditions act as an effective {\em forcing} which is absent from our problem and the 
techniques of~\cite{QuVa} do not  directly apply.  Since the corresponding Hele-Shaw problem (in the absence of surface tension and forcing) is {\em not}
a dynamic problem, possessing only time-independent solutions,  we are not able to use the  Hele-Shaw solution as a  comparison problem for our problem.

A global stability result for the two-phase classical Stefan problem in a smooth
functional framework was also established by {\sc Meirmanov}~\cite{Me} for a specific (and somewhat restrictive) perturbation
of a flat interface, wherein the initial geometry is a strip with imposed Dirichlet temperature conditions on the fixed top and
bottom boundaries, allowing for only one equilibrium solution.
A global existence result for smooth solutions was given by \textsc{Daskalopoulos \& Lee}~\cite{DaLe} under the log-concavity assumption
on the initial temperature function, which in light of the level-set reformulation of the Stefan problem, requires convexity of the
initial domain (a property that is preserved by the dynamics).

In~\cite{HaSh}, we established the local-in-time existence, uniqueness, and regularity for the classical Stefan problem  in $L^2$ Sobolev spaces, without derivative loss, using
 the functional framework given by 
(\ref{eq:ALEenergy}) and~(\ref{eq:ALEdissipation}). This framework is natural,  and  relies on the geometric control of the free-boundary,  analogous to that
 used in the analysis of the free-boundary incompressible Euler equations in
 \textsc{Coutand \& Shkoller}~\cite{CoSh07,CoSh10}; the second-fundamental form  is controlled by a
 a natural coercive quadratic form,  generated from the inner-product of the tangential derivative 
 of the cofactor matrix $a$, and the tangential derivative of the velocity of the moving boundary, and
yields control of the norm
$
\int_{\Gamma}(-\g_Nq(t))|\t^k \, h|^2\,dx' 
$ for any $k\ge 3$. 
The Hopf lemma ensures positivity of $-\g_N q(t)$ and the Taylor sign condition
 on $q_0$ ensures a uniform lower-bound as $t \to 0$;
on the other hand, $-\g_Nq(t) \to 0$ as $t \to \infty $, and so an optimal lower-bound for $(-\g_Nq(t))$ for large $t$ is essential to establish a global existence and
stability theory.

We remark that  global stability of solutions in  the {\em presence} of surface tension (see, for example, \cite{HaGu1,Ha1,PrSiZa}) does not require the use of function framework with a  decaying weight, such as $-\g_Nq(t)$. 
In this regard, the surface tension problem is
simpler for two important reasons:  first, the surface tension contributes a positive-definite energy-contribution that is uniform-in-time, and provides 
better regularity of the free-bounary (by one spatial derivative), and second, the space of equilibria is finite-dimensional and thus it is easier to understand the 
degrees-of-freedom that regulate the asymptotic state of the system, given the initial conditions. 

\subsection{Methodology and outline of the paper}
Our present work builds on our new energy method for the Stefan problem that we developed 
  in~\cite{HaSh}.
We obtain {\it global and uniform} control of the geometry of the free-boundary by controlling the weighted boundary-norm $ \sup_{t\in[0,T]} \| \sqrt{\chi(t)} h\|_6$ for
all $t \ge 0$.   We are thus able to track the location of the moving free-boundary and measure its deviation from the initial state;
this  geometric control is strongly coupled to, and dependent upon,  the
exponential-in-time decay of the temperature function to zero.

There exist infinitely many steady states for the classical Stefan problem:   for any sufficiently smooth hypersurface $\bar{\Gamma}\subset\R^d$,
the pair $(\bar{p},\bar{\Gamma})\equiv(0,\bar{\Gamma})$ forms an equilibrium solution of the Stefan 
problem~(\ref{eq:stefan}).
This abundance of possible attractors for the long-time behavior of the solution $\Gamma(t)$
creates a conceptual difficulty in approaching the question of ``asymptotic'' convergence.

We address  the temporal asymptotics by requiring our initial surface to be a small perturbation of the reference sphere. We use the  energy
spaces introduced in~\cite{HaSh}; moreover, we do not expect to observe any
decay for the height of the moving surface in this norm. Rather,
 given the expectation that the solution {\em does} converge to 
some nearby shape (so that $h$ remains small), we expect the temperature $q(t)$ to converge to zero exponentially fast,
since it is a solution of the nonlinear heat equation~(\ref{eq:ALEheat}).
Returning to the definition of the energy space $ \mathcal{E} $ given in (\ref{eq:ALEenergy}), we 
immediately encounter a potential problem for global-in-time estimates; specifically, 
the {\it coefficient }  $-\g_Nq(t)$ in the energy expression
$\int_{\Gamma}\bigl(-\g_Nq(t) \bigr)|\t^6 h|^2\,d\theta$
is also expected to decay as $t \to \infty $ and it is {\it a priori} unclear how to {\it uniformly-in-time} control the regularity of
the boundary height function $h$.
To understand the relationship between the decay of $q(t)$ and the smallness of $\E$,  we 
will analyze the dynamics in three different and coupled regimes.
\subsubsection*{High-order energy estimates} 
We do not expect the height function $h(t)$ to decay to $0$ as $t \to \infty $;  rather,  we expect $h(t)$ to
remain close to the initial height  function $h_0$. Assuming, without loss of generality,  that  $h_0=0$, 
to guarantee the smallness of $h-h_0=h$ we will prove that
\begin{align} 
\sup_{0\leq s\leq t}\E(s)+\int_0^t\D(s)\,ds & \leq\E(0)+ \sup_{0\leq s\leq t} P( E_\beta) \ \E(s)+\delta\int_0^t\D(s)\,ds \nonumber \\
& \leq\E(0)+O({\epsilon_0})\sup_{0\leq s\leq t}\E(s)+\delta\int_0^t\D(s)\,ds,
\label{eq:enest}
\end{align} 
where $P$ is some polynomial function of the low-norm $E_{\beta}$.
The above estimate yields an {\it a priori} bound on $\E$ if $\epsilon$, $\delta$ and $ \mathcal{E} (0)$  are sufficiently small.  

However, to  close the higher-order energy estimates and thus obtain (\ref{eq:enest}), we must contend with a very problematic
integral (or error term)  given by
\[
\mathcal{N}:=-\int_0^T\int_{\Gamma}\g_Nq_t \, |\t^6h|^2\,d\theta \, dt \,.
\]
Driven by  intuition from the linear heat equation, we expect
$\g_Nq_t$ to decay {\em exactly} as fast as $-\g_Nq$. 
Comparing $\mathcal{N}$ to the energy contribution $\int_{\Gamma}(-\g_Nq)|\t^6h|^2$ above, 
we note that $\mathcal{N}$  {\em cannot}  be controlled by $\E$,  as it is the same order  as
 $\E$.  Hence, 
to bound $\mathcal{N}$, we prove that after a sufficiently long time has elapsed, the quantity
$\g_Nq_t$ turns strictly positive and hence $\mathcal{N}$ can be bounded from above by zero.
In Lemma~\ref{lm:positive} we will quantify the meaning of ``sufficiently long" time $t=T_K$ from the previous sentence,
expressing it as a function of the ratio $K=\|q_0\|_4/\|q_0\|_0$. 

More precisely, we break the total time interval into a (possibly long) transient interval $[0,T_K]$ and $[T_K, \infty )$.
On the transient time-interval  $[0,T_K]$ we {\em do} treat $\mathcal{N}$ as an error term, and
by choosing $\mathcal{E} (0)$ sufficiently small, a straightforward  application of a Gronwall-type inequality verifies
that the interval of existence is greater than $T_K$,  as explained in our proof of the main theorem (given Section~\ref{se:mainproof}).    
The bound for $ \mathcal{N} $ grows exponentially with time, 
and as such,
{\em cannot} be used to establish global-in-time estimates.    Instead, a significantly more refined analysis is employed on the
time-interval
$[T_K, \infty )$, wherein we prove in Lemma~\ref{lm:positive} the negativity of $\mathcal{N}$ for $t=T_K$
and then use a maximum principle-type argument to guarantee the negativity for  all $t\geq T_K$.

\subsubsection*{Exponential decay-in-time of the temperature function $q$}
The last inequality  in (\ref{eq:enest}) holds  only if  $E_\beta$ itself remains small;  in fact
we will prove  that as $t \to \infty $,  $\|q(t)\|_4^2$ has the nearly
 optimal decay rate
\be\label{eq:beta0}
e^{-(2\lambda_1- C\epsilon_0)t},
\ee
where $\lambda_1$ denotes the smallest eigenvalue of the Dirichlet-Laplacian on the unit disk.  
Moreover, the parabolic estimate we prove, will be roughly of the form
\be\label{eq:betaest}
\g_tE_{\beta}+D\leq C(\epsilon_0 +\|q_0\|_4\frac{e^{-\beta t/2}}{\chi(t)^{1/2}})D,
\ee
where the norms $E_{\beta}$ and $D$ have been defined in~(\ref{eq:ebeta}). A nice 
consequence of our analysis is that the potentially growing term,
$\frac{e^{-\beta t/2}}{\chi(t)^{1/2}}$,  in fact remains small and decays in time.
Next, we explain why this is true.
\subsubsection*{Lower bound for the velocity of the free boundary}
We may think of the presence of the denominator $\frac{1}{\chi(t)^{1/2}}$
in the estimate~(\ref{eq:betaest}) as a possible obstruction to
controlling the regularity of $h$ and thus potentially 
preventing uniform ellipticity bounds for the parabolic 
operator~(\ref{eq:ALEheat}).
To deal with this issue,  we need a quantitative lower bound on 
the decay rate of $\chi(t)$. 
Moreover, this lower bound has to favorably compare to the 
size of $e^{-\beta t}$. With some extra work, such a Hopf-type inequality is
 implied by a result of Oddson~\cite{Od},  which leads to the lower-bound
\be\label{eq:odbound}
\chi(t)\gtrsim c_1 e^{-(\lambda_1+ c \epsilon _0) t},
\ee
where  $c>0$ denotes a generic constant, and as before $c_1=\int_{\Omega}q_0\varphi_1$ is the first 
coefficient in the eigenfunction expansion of the initial datum $q_0$ with respect to 
the $L^2$-orthonormal eigenbasis of the Dirichlet-Laplacian on the unit disk.
Finally, combining~(\ref{eq:beta0}) and~(\ref{eq:odbound}),
we will show in Lemma~\ref{lm:useful},
that for small initial data, 
\be\label{eq:fund}
\|q_0\|_4\frac{e^{-\beta t/2}}{\chi(t)^{1/2}} \lesssim \sqrt{\epsilon_0} e^{-\gamma^* t}.
\ee
for some positive constant $\gamma^*$. 

The result of Oddson~\cite{Od} relies on a good choice of a barrier function that, combined with a maximum principle, allows for  very precise information on the decay rate.
That choice is, however,  only one possible choice of a comparison function, and it is possible that there are different ones since~\cite{Od} gives nearly {\em sharp} decay rate only
in a nearly {\em radial} regime. If nearly radial, it is possible that in a viscosity or weak solution framework, one can use comparison principle arguments to deduce that ``no-thin tentacles" form (cf.~\cite{JeLeSh} which is in spirit close to~\cite{QuVa}, but again relies on presence of the forcing term) and the moving boundary remains in an annulus of width
$O(\epsilon)$. To that end, but in absence of forcing, the ideas from~\cite{AtCaSa2, ChKi, QuVa} may be very valuable - they would require a construction of an adaptive family of comparison functions that yield precise decay rates as time evolves.
%
In  forthcoming work, we plan to address the Stefan problem on arbitrary domains 
diffeomorphic to the unit ball, as well as the case of the two-phase Stefan problem. In both instances and not unrelated to the above discussion, we shall need a better, new choice of barrier functions related to the existence of 
so-called half-eigenvalues for the extremal Pucci operators in order to get the sharp decay rates. In particular, our approach is insensitive to the convexity properties of the initial domain, but it requires sufficient regularity.

Another advantage of  the techniques developed 
in this paper is that it provides a general and robust framework for addressing 
the global stability questions for related free boundary problems in fluid mechanics in absence of surface tension.

\subsubsection*{Plan of the paper}
In Section~\ref{se:basic},  we introduce the bootstrap assumptions and 
obtain various a priori estimates, that allow us to control low norms of the
boundary function $h$ as well as the decaying low-norm $E_{\beta}$, and also
establish the equivalence between the
energies and the norms as mentioned earlier in the introduction (Section~\ref{se:equiv}).
In Section~\ref{se:long}, we state energy identities and then
perform the energy estimates. 
Finally in Section~\ref{se:main}, we prove the main theorem.  In Section 5, we discuss the
modifications required for the analysis in three space dimensions, and for initial height functions $h_0 \neq 0$.
Appendix~\ref{se:derivation} is devoted to the proof of the energy identities stated in Section~\ref{se:long}.
The very short Appendix B provides a simple proof for the upper bound of $\g_Nq_t$.
\section{Bootstrap assumptions and {\it a priori} bounds}\label{se:basic}
Let us assume that the solution $(q,h)$ 
to the Stefan problem~(\ref{eq:ALE}) exists on some time interval $[0,T]$, $T>0$, which is guaranteed by
\cite{HaSh}.
With the positive constant $ \epsilon _0 < \epsilon \ll 1$  to be specified later, we make the following
 bootstrap assumptions:
\begin{subequations}
\label{bootstrapassumptions}
\begin{alignat}{2}
&(\text{smallness})\qquad\qquad  && \sup_{0\leq s\leq T}\E(s)+\int_0^t\D(s)\,ds\leq\epsilon^2,\quad
 \sup_{0\leq s\leq T}E_{\beta}(s)+\int_0^tD(s)\,ds\leq \tilde{C}E_{\beta}(0)\,,   \label{eq:assumption1} \\
 & (\text{lower-bound}) &&
  \chi(t)\gtrsim c_1e^{-(\lambda_1+\eta/2)t} \,,  \label{eq:assumption2}
\end{alignat} 
\end{subequations}
where we the 
definitions  of $\E$, $\D$, $E_{\beta}$, and $D$ are provided in 
(\ref{eq:ALEenergy}),~(\ref{eq:ALEdissipation}), and~(\ref{eq:ebeta}),  respectively.
With $ \beta $ given in (\ref{eq:beta}), 
$\beta=2\lambda_1-\eta$, 
the  bootstrap assumption (\ref{eq:assumption2}) can be written as 
$\chi(t)\gtrsim c_1e^{-(\beta/2+\eta) t}$.
Moreover, $\eta>0$ is a fixed small constant and it will be shown in the proof of the main theorem, Section~\ref{se:mainproof},
that $\eta$ must be chosen smaller than $1/\sqrt{C\ln K}$ for some universal constant $C$. 
Note that since $E_{\beta}(0)\leq\epsilon^2$, (\ref{eq:assumption1})  implies the {\it decay estimate} 
$\|q\|_4^2\leq\epsilon^2 e^{-\beta t}$.   Recall that the  constant $c_1$ in the estimate~(\ref{eq:assumption2}) 
is defined as $\int_\Omega q_0(x)\varphi_1(x) dx$.

We now briefly explain the logic of the proof of global existence that will be
carried out in Section~\ref{se:main}. 
If $\mathcal{T}$ is defined to be the maximal time at which the solution $(q,h)$ exists and satisfies the bootstrap
assumptions, the first objective is
to show that the bootstrap assumptions~(\ref{eq:assumption1}) and~(\ref{eq:assumption2}) yield
an improved {\it smallness} and {\it lower-bound} estimates at time $\mathcal{T}$. 
If $\mathcal{T}$ were finite, by the local-in-time well-posedness theory and continuity of our norms
we can extend the solution to an interval $\mathcal{T}+T^*$, while preserving the bootstrap assumptions~(\ref{eq:assumption1})
and~(\ref{eq:assumption2}), thus arriving at contradiction to the definition of $\mathcal{T}$. Hence $\mathcal{T}$ must be infinite.

It remains to show that for $\epsilon$ chosen small enough, the {\it smallness} and the {\it lower-bound} 
estimates can indeed be improved.
In Corollary~\ref{co:oddson} we will show that the assumption~(\ref{eq:assumption2}) is in fact improved,
and in Lemma~\ref{lm:heat} we
show that the assumption on $E_{\beta}+\int_0^TD$ in~(\ref{eq:assumption1}) is also improved. Finally,
in Section~\ref{se:mainproof}, we will prove that the smallness of $\E+\int_0^t\D$ assumed 
in~(\ref{eq:assumption1}) is also preserved. Thus the smallness regime 
introduced through~(\ref{eq:assumption1})--(\ref{eq:assumption2}) will be shown to remain preserved
by the dynamics of~(\ref{eq:ALE}) for $\epsilon>0$  chosen sufficiently small.
\subsection{Poincar\'e-type inequality} 
Because the first eigenfunction $\varphi_1$ of the Dirichlet-Laplacian is positive in $\Omega$, while the remaining eigenfunctions oscillate about zero,
it will be necessary to introduce a constant into our estimates which gives a measure of the initial temperature distribution in the first mode of the dynamics.
To this end, we will make use of the following
\begin{lemma}\label{lm:hardy} For $k\ge 3$,
 let $f\in H^k(\Omega)\cap H^1_0(\Omega)$, $f:\Omega\to\R^+$ be a strictly positive function on 
 the interior of $\Omega$.
 Let $\varphi_1$ be the first eigenvector of the Dirichlet-Laplacian on the
 unit ball $B_1(0)=\Omega$. Then there exists a universal constant $C^*$
 such that
 \[
 \|f\|_0^2\leq C^*\big(\int_{\Omega}f(x)\varphi_1(x)\,dx\big)\|f\|_3.
 \]
\end{lemma}
\begin{proof}
We have that 
$$
\int_\Omega f^2 dx \le   \max_{x\in \Omega } \frac{f(x)}{\varphi_1(x)} \,  \int_\Omega f \varphi_1 dx \,.
$$
Since $-\frac{ \partial \varphi_1}{\partial N}(x) \ge c >0$ for all $x \in \Gamma$, the higher-order Hardy inequality
(Lemma 1 in \cite{CoSh12}) together with the Sobolev embedding theorem shows that
$$
  \max_{x\in \Omega } \frac{f(x)}{\varphi_1(x)} \le C \left\| \frac{f}{\varphi_1}\right\|_2 \le C \|f\|_3
$$
which proves the lemma.
\end{proof}
\begin{corollary}\label{cor:hardy}
Let $q_0 \in H^4(\Omega) \cap H^1_0(\Omega) $ with $q_0 >0$ in $\Omega$.
We consider the eigenfunction expansion $q_0=\sum_{j=1}^{\infty}c_j\varphi_j$  of $q_0$ with respect 
to the $L^2$-orthonormal basis $\{\varphi_1,\varphi_2,\dots\}$ consisting of  the Dirichlet-Laplacian 
eigenfunctions
on the unit disk $B_1(0)=\Omega$.
Then, if $\frac{\| q_0\|_4}{\|q_0\|_0}\leq K$, it follows in particular that
\[
\frac{|c_j|}{c_1}<K,\quad j=1,2,\dots
\]
\end{corollary}
\begin{lemma}\label{lm:useful} If the
 bootstrap assumptions~(\ref{eq:assumption1}),~(\ref{eq:assumption2}) hold, then
 \be\label{eq:useful}
 \frac{E_{\beta}^{1/2}(t) e^{-\beta t / 2}}{\chi(t)^{1/2}}\leq\frac{\tilde{C}^{1/2}E_{\beta}(0)^{1/2}e^{-\beta t/2}}{\chi(t)^{1/2}}
 \lesssim\sqrt{\epsilon} e^{-\gamma t/2 }
 \ee
 where $\gamma=\frac{\beta}{2}-\eta>0$.
\end{lemma}
\begin{proof}
By~(\ref{eq:assumption2}), we have that
\beas
\frac{E_{\beta}(t)^{1/2} e^{-\beta t / 2}}{\chi(t)^{1/2}}&\leq& C\frac{e^{-\beta t/2}}{e^{-(\lambda_1/2+\eta/4)t}}\frac{E_{\beta}(0)^{1/2}}{c_1^{1/2}}\\
&\leq&Ce^{-\gamma t/2}\frac{\|q_0\|_4}{c_1^{1/2}}\leq CK\|q_0\|_4^{1/2}e^{-\gamma t/2}
\leq C\sqrt{\epsilon} e^{-\gamma t/2},
\eeas
where we have used the fact that
$c_1^{1/2}\gtrsim\frac{1}{K^{1/2}}\|q_0\|_0^{1/2}$ and $\|q_0\|_4\lesssim K\|q_0\|_0$.
We have also used the bound $K\|q_0\|_4^{1/2}\leq C\sqrt{\epsilon}$ (since
 $\epsilon_0<\epsilon$),  as well as the smallness assumption~(\ref{eq:smallassum})
so that $K\|q_0\|_4^{1/2}\lesssim K\epsilon_0/F(K)^{1/2}\leq C\epsilon$.
Note that $\gamma$ is explicitly given by $\gamma
=(\frac{\beta}{2}-\eta)>0$, and that
 $\eta \ll \lambda_1/2$.
\end{proof}

\subsection{A priori bounds on $h$}
%
\begin{lemma}[Suboptimal decay bound for $h_t$] \label{lm:useful0}
Under the bootstrap assumptions~(\ref{eq:assumption1}) and~(\ref{eq:assumption2}), the following decay bound holds:
\be\label{eq:useful0}
|h_t|_{2.5}\lesssim \epsilon e^{-\gamma t/2}.
\ee
\end{lemma} 
\begin{proof}
Differentiating equation (\ref{eq:velformula}), the Sobolev embedding theorem together with the fact that  $h\ge 0$  (by the maximum principle) show that
\begin{align*}
|h_t|_1 &\lessim |v|_{W^{1, \infty }} + | h|_2 |v|_2 + |h|_1 |v|_1 |h|_1 \\
& \lessim |v|_2 + | h|_2 |v|_2 + |h|_1 |v|_1 (|h_0|_1 + t \sup_{0\leq s\leq t}|h_t|_{1}) \,,
 \end{align*}
where we have used the fundamental theorem of calculus for the last inequality.  
Using the bootstrap assumption (\ref{eq:assumption1}), we see that $ |v(t)| \lessim e^{-\beta t}$, while
thanks to Lemma~\ref{lm:useful} and the fact that $ \sqrt{\E} \lessim \epsilon _0 < \epsilon $, 
$$
|h|_{2}|v|_2 \lesssim  \sqrt{\chi} |h|_{2} \ \frac{|v|_1}{\sqrt{\chi}} \lesssim\frac{\sqrt{E_\beta}}{\sqrt{\chi}} \sqrt{\E} e^{-\beta t/2}
\lesssim \epsilon e^{-\gamma t/2}.
$$
Hence,
\[
\sup_{0\leq s\leq t}|h_t|_{1}\lesssim\epsilon e^{-\beta t/2}+\epsilon e^{-\gamma t/2}\big(1+\sup_{0\leq s\leq t}|h_t|_{1}\big) \,,
\]
and with $ \epsilon >0$ sufficiently small, we see that
\begin{equation}\label{hass1001}
\sup_{0\leq s\leq t}|h_t|_{1}\lesssim\epsilon e^{-\gamma t/2}\lesssim\epsilon \,.
\end{equation} 

Taking more derivatives of  (\ref{eq:velformula}), the Sobolev embedding theorem shows that 
for $k=2,3$,
\begin{align}
|h_t|_k&\leq |v|_k+\Big|\frac{h_{\theta}}{1+h}\Big|_{L^{\infty}}|v|_k+\Big|\frac{h_{\theta}}{1+h}\Big|_k|v|_{L^{\infty}}
 \lesssim |v|_k+|h_{\theta}|_{1}|v|_k+\Big|\frac{h_{\theta}}{1+h}\Big|_k|v|_{1}\label{eq:sub0},
\end{align}
where we have again used the fact that $h\ge 0$.
Since 
$$
\big|\frac{h_{\theta}}{1+h}\big|_k
\lesssim|h|_{k+1}(1+P(|h|_{k-1})), \ \ k=2,3,
$$
for some polynomial function $P$, and since    $|h|_k\leq |h_0|_k +  t\sup_{0\leq s\leq t}|h_t|_k$,
we see that
\be\label{eq:sub1}
\big|\frac{h_{\theta}}{1+h}\big|_k\lesssim|h|_{k+1}\big(1+P(t)P(\sup_{0\leq s\leq t}|h_t|_{k-1})\big).
\ee
We now use~(\ref{eq:sub1}) and~(\ref{eq:sub0}) to infer that
\begin{equation}\label{hass1002}
|h_t|_k\lesssim |v|_k\big(1+\sup_{0\leq s\leq t}|h_t|_2\big)+|h|_{k+1}|v|_1\big(1+P(t)P(\sup_{0\leq s\leq t}|h_t|_2)\big),
\end{equation} 
where 
we have used $|h_{\theta}|_1\lesssim t\sup_{0\leq s\leq t}|h_{t}|_2$.
Interpolating between $k=2$ and $k=3$ yields
\begin{align}
|h_t|_{2.5}&\lesssim|v|_{2.5}\big(1+\sup_{0\leq s\leq t}|h_t|_2\big)+|h|_{2.5}|v|_1\big(1+P(t)P(\sup_{0\leq s\leq t}|h_t|_2)\big) \,.
\label{eq:sub2}
\end{align}
and as above, 
Lemma~\ref{lm:useful} provides us with the inequality
$|h|_{2.5}|v|_1
\lesssim \epsilon e^{-\gamma t/2}$,  which together with the bootstrap assumption  (\ref{eq:assumption1}) shows that
$$
\sup_{0\leq s\leq t}|h_t|_{2.5}
\lesssim \epsilon e^{-\beta t/2}\big(1+\sup_{0\leq s\leq t}|h_t|_{2.5}\big)
+\epsilon e^{-\gamma t/2}\big(1+P(t)P(\sup_{0\leq s\leq t}|h_t|_{2})\big)
$$
and therefore with $ \epsilon >0$ sufficiently small,
\be\label{eq:sub3}
\sup_{0\leq s\leq t}|h_t|_{2.5}\lesssim \epsilon e^{-\beta t/2}+\epsilon e^{-\gamma t/2}(1+P(\sup_{0\leq s\leq t}|h_t|_{2})),
\ee
where the polynomial $P(t)$ has been absorbed in some universal constant due to
the exponentially decaying factor $e^{-\gamma t/2}$.  On the other hand, the inequality (\ref{hass1002}) with $k=2$ together
with the estimate  (\ref{hass1001}) shows that  $|h_t|_2 \lessim \epsilon $ so that  with (\ref{eq:sub3}),
we conclude the proof.
\end{proof}
\begin{remark}
 Note that the estimate~(\ref{eq:useful0}) can be stated more precisely, by keeping track of constant
 $c_1$ on the right-hand side, in which case,
 \be\label{eq:useful0improved}
 |h_t|_{2.5}\lesssim \epsilon^{1/2}\sqrt{c_1} e^{-\gamma t/2}.
 \ee
 The proof follows from the last line of the proof of Lemma~\ref{lm:useful0}
since $E_{\beta}(0)^{1/2}\leq K^2c_1$, due to the bound $\|q\|_4\leq K\|q_0\|\leq K^2c_1$.
 Note that $\sqrt{\epsilon}$ on the right-hand side of~(\ref{eq:useful}) can be replaced by $\sqrt{c_1}$ for
 the same reason.
\end{remark}

\begin{lemma}[Smallness of the height function]\label{lm:basic} Let  $c_1=\int_{\Omega}q_0\varphi_1dx$ and suppose that
 the bootstrap assumptions~(\ref{eq:assumption1}),~(\ref{eq:assumption2}) hold.  For 
 $\epsilon > 0$  taken sufficiently small,
 \be\label{eq:basic}
\sup_{0\leq{s}\leq t} |h({s})|_{4.5}\lesssim\sqrt{\epsilon}\,,
 \ee
while for lower-order norms, 
\be\label{re:height1}
\sup_{0\leq s\leq t}|h(s)|_{2.5}\lesssim c_1 \text{ and }
\sup_{0\leq s\leq t}|h(s)|_4\lesssim\epsilon^{1/2}c_1^{1/4}. 
\ee 
\end{lemma}
\begin{proof}
Observe that
\[
|h|_{2.5}^2\leq 2\int_0^t|h|_{2.5}|h_s|_{2.5}\,ds
\leq\sup_{0\leq s\leq t}|h(s)|_{2.5}\int_0^t|h_s|_{2.5}\,ds
\lesssim\sup_{0\leq s\leq t}|h(s)|_{2.5}\int_0^t\epsilon^{1/2}\sqrt{c_1}e^{-\gamma t/2},
\]
where we have used~(\ref{eq:useful0improved}) in the last bound.
Taking the supremum over the time interval $[0,t]$ we deduce
\[
\sup_{0\leq s\leq t}|h(s)|_{2.5}\lesssim\epsilon^{1/2}\sqrt{c_1}.
\]
Using the well-known interpolation estimate (see, for example, \cite{Ad})
\be\label{eq:interp}
|f|_{k}\leq|f|_l^{\theta}|f|_m^{1-\theta},\quad \theta=\frac{m-k}{m-l},\,\,\,l\leq k\leq m,
\ee
with $k=3$, $l=2.5$, $m=4$, and the fact that $|\sqrt{\chi}\t^4h_t|_0^2$ is bounded by $\E$, we have that
\beas
|h_t|_{3}&\lesssim&|h_t|_4^{1/3}|h_t|_{2.5}^{2/3}
\lesssim\frac{\E^{1/6}}{\chi(t)^{1/6}}\epsilon^{1/3}c_1^{1/3}e^{-\gamma t/3}\\
&\lesssim&\epsilon^{2/3}c_1^{1/6}e^{-\gamma^{*}t},
\eeas
where $\gamma^{*}=-\frac{1}{3}\gamma+\frac{1}{6}(\frac{\beta}{2}+\frac{\eta}{2})=
-\frac{1}{6}\beta+\frac{5\eta}{12}>0$ (by definition, 
$\gamma=-\frac{\beta}{2}-\eta$).
As a consequence,
\[
|h|_{3}^2\lesssim\int_0^t|h|_{3}|h_t|_{3}
\lesssim\sup_{0\leq{s}\leq t}|h({s})|_{3}\int_0^t|h_t({s})|_{3}\,d{s}
\lesssim\epsilon \sup_{0\leq{s}\leq t}|h({s})|_{3}.
\]
Upon taking the supremum over the inetrval $[0,t]$, we finally have that
\be\label{eq:apriori3der}
\sup_{0\leq{s}\leq t}|h({s})|_{3}\lesssim\epsilon .
\ee

We can now improve the decay result of Lemma~\ref{lm:useful0}, first for the
quantity $|h_t|_2$. Simply using the bound~(\ref{eq:apriori3der}), exactly as in the proof of
Lemma~\ref{lm:useful0} , we infer the improved estimate
\be\label{eq:improbound}
|h_t|_{2}\lesssim\|v\|_{2.5}(1+|h|_3)\lesssim c_1e^{-\beta t/2}.
\ee
As an immediate consequence, we obtain the smallness bound for $\sup_{0\leq{s}\leq t}|h({s})|_{4}$:
\bea
\int_{\Gamma}|\t^4h|^2\,d\theta&=&\int_0^t\int_{\Gamma}\t^4h\t^4h_t\,d\theta\,d{s}
=\int_0^t\int_{\Gamma}\t^6h\t^2h_t\,d\theta\,d{s} \nonumber \\
&\leq&\int_0^t|\t^6h|_0|\t^2h_t|_0\,d{s}
\lesssim\int_0^t\big(\frac{\E^{1/2}}{\chi({s})^{1/2}}c_1e^{-\beta {s}/2}\big)\,d{s} \nonumber \\
&\lesssim&\int_0^t\epsilon \sqrt{c_1}e^{-\gamma s/2}\,d{s}
\lesssim \epsilon \sqrt{c_1} \,. \label{hass7}
\eea
Note that~(\ref{hass7}), in particular,  implies the second bound in~(\ref{re:height1}).
Next, we establish the a priori smallness of $\sup_{0\leq{s}\leq t}|h({s})|_{4.5}$.  Thanks to (\ref{hass7}), 
we  improve the decay bound
for $|h_t|_{2.5}$  in an analogous fashion  to the improved  decay  estimate~(\ref{eq:improbound}) 
for $|h_t|_2$.
We obtain 
$
|h_t|_{2.5}\lesssim c_1e^{-\beta t/2}.
$
The first bound in~(\ref{re:height1}) now follows from the fundamental theorem of calculus and the previous bound.
A straightforward interpolation argument for fractional Sobolev spaces on the unit 
circle $\Gamma$, shows
\be\label{eq:apriorib}
|h|_{4.5}^2\lesssim\int_0^t|h|_6|h_t|_3\,d{s}.
\ee
Using the interpolation estimate~(\ref{eq:interp}), with $l=2.5$, $k=3$, and $m=5$, we see that
\begin{equation}\label{ssss201}
|h_t|_{3}\leq C|h_t|_{2.5}^{4/5}|h_t|_5^{1/5}.
\end{equation} 
Using (\ref{ssss201})  with (\ref{eq:apriorib}), and using the above bound on $|h_t|_{2.5}$, yields 
\beas
|h|_{4.5}^2&\lesssim& \int_0^t|h|_6|h_t|_{2.5}^{4/5}|h_t|_5^{1/5}\,d{s}
\lesssim\int_0^t\frac{\E^{1/2}}{\chi({s})^{1/2}}c_1^{4/5}e^{-2\beta {s}/5}|h_t|_5^{1/5}\,d{s}\\
&\lesssim&\epsilon c_1^{3/10}\int_0^te^{-\bar{\gamma}{s}}|h_t|_5^{1/5}\,d{s},
\eeas
where we also used the bootstrap assumption~(\ref{eq:assumption2}). One checks that
$\bar{\gamma}=-\frac{2\beta}{5}+(\frac{\beta}{4}+\frac{\eta}{2})=\frac{3}{20}\beta-\frac{\eta}{2}>0$. We thus have
\[
|h|_{4.5}^2\lesssim\epsilon c_1^{3/10}\int_0^te^{-\bar{\gamma}{s}/2}
\times \Big(e^{-\bar{\gamma}{s}/2}|h_t|_5^{1/5}\Big)\,d{s}.
\]
 H\"older's inequality with $p=\frac{10}{9}$ and $q=10$ then shows that
\beas
|h|_{4.5}^2&\lesssim& \epsilon c_1^{3/10}
\Big(\int_0^t\big(e^{-\bar{\gamma}{s}/2}\big)^{10/9}\,d{s}\Big)^{9/10}
\Big(\int_0^t(e^{-5\bar{\gamma}{s}}|h_t|_5^2\,d{s}\Big)^{1/10}\\
&\lesssim&\epsilon c_1^{3/10}\Big(\int_0^te^{-5\bar{\gamma}{s}}|h_t|_5^2\,d{s}\Big)^{1/10}
\lesssim\epsilon^{6/5} c_1^{1/5},
\eeas
where the last inequality follows from the definition of $\bar{\gamma}$ above, the bootstrap assumptions ~(\ref{eq:assumption2}) 
and~(\ref{eq:assumption1}), 
and the estimate
\begin{align*}
\int_0^te^{-5\bar{\gamma}{s}}|h_t|_5^2\,d{s}
& \lesssim \int_0^t\frac{1}{c_1}e^{-5\bar{\gamma}{s}+(\beta/4+\eta/2)s}\inf_{\Gamma}(-\g_Nq({s}))|h_t|_5^2\,d{s}\\
& \lesssim \int_0^t\frac{1}{c_1}e^{-(\beta/2+3\eta){s}}\inf_{\Gamma}(-\g_Nq({s}))|h_t|_5^2\,d{s}\\
& \lesssim \frac{1}{c_1}\int_0^t\int_{\Gamma}(-\g_Nq({s}))|\t^5h_t|^2\,d\theta\,d{s}\leq\frac{\epsilon^2}{c_1}\,.
\end{align*}
\end{proof}

\subsection{Differentiation rules for $A$} Since 
 $A=[D\Psi]^{-1}$, it follows that
 \[
\g_tA^k_i=-A^k_rw^r,_sA^s_i;\qquad\t A^k_i=-A^k_r\t\Psi^r,_sA^s_i.
\]
In particular, a simple application of the above identities and the product rule imply that
for any given $a,b\in\N$, 
\begin{subequations}
\label{eq:comm}
\begin{align}
\t^a\g_t^bA^k_i& =-A^k_r\t^a\g_t^b\Psi^r,_sA^s_i+\{\t^a\g_t^b,\,A^k_i\}\,,\\
\{\t^m\g_t^n,\,A^k_i\} & :=\sum_{l+l'\geq1}a_{l,l'}\t^l\g_t^{l'}(A^k_rA^s_i)\t^{m-l}\g_t^{n-l'}\Psi^r,_s \,,
\end{align} 
\end{subequations}
where the term $\{\cdot,\,\cdot\}$  is the {\em commutator} error.
Here the constants $a_{l,l'}$ are some universal constants, depending only on $m$, $n$, $l$ and $l'$
(where $0\leq l\leq m$, $0\leq l'\leq n$).

\subsection{Estimates for $ \nabla \Psi - \operatorname{Id} $ and $ A- \operatorname{Id} $}

 Under assumption (\ref{eq:assumption1}),  the elliptic estimate (\ref{DM}) shows that on the time-interval $[0,T]$,
\begin{align}
\|\nabla \Psi -   \operatorname{Id} \|_{L^\infty(B_1)} &\le C \|\nabla \Psi - \operatorname{Id} \|_{1.5} \le C |h|_2
\label{psi_est1}
\end{align}
and for $0\le s\le 3$,
\[
\|D^2\Psi\|_s \le C |h|_{s+1.5} \,.
\]
Estimate (\ref{psi_est1}) implies that
\begin{align*}
\|A - \operatorname{Id} \|_{L^\infty(B_1)} &= \|(\operatorname{Id}  - \nabla \Psi)A\|_{L^\infty(B_1)} \le C \|A\|_{L^\infty(B_1)}|h|_2;
\end{align*}
thus under assumption (\ref{eq:assumption1}),
\begin{align}
\|A - \operatorname{Id} \|_{L^\infty(B_1)} \le C |h|_2 \label{A-1_est}
\end{align}
Note that (\ref{psi_est1}) and (\ref{A-1_est}) together imply that for $0\le s\le 3$,
\[
\|D A\|_s \le C |h|_{s+1.5}.
\]
Thus, with Lemma \ref{lm:basic}, we have proven the following
\begin{lemma}\label{lm:basic2} With
 the bootstrap assumptions~(\ref{eq:assumption1}),~(\ref{eq:assumption2}) and for
 $\epsilon > 0$  taken sufficiently small,
 $$
 \|\nabla \Psi -   \operatorname{Id} \|_{4}+ \|A - \operatorname{Id} \|_{4}  
 \lesssim \sqrt{ \epsilon }  \,.
 $$
 \end{lemma}

\subsection{High-order derivatives of $q$}\label{se:equiv}
Because our energy function $ \mathcal{E} (t)$ is formed using  only tangential derivatives  in space, the
 purpose of this section is show that radial derivatives of the temperature $q$ are also bounded, and thus the
 full Sobolev norms of the 
temperature $q$ are controlled by our energy function, as  was explained in the
introduction. 

We will make use of the heat equation and its time-differentiated variants:
\begin{subequations}
\label{heatss}
\begin{align} 
q_t - \Delta _\Psi q  & = f_0 \,, \\
q_{tt} - \Delta _\Psi q_t &= f_1\, \\
q_{ttt} - \Delta _\Psi q_{tt} &= f_2\,,
\end{align} 
\end{subequations}
where $ \Delta _\Psi = A^j_i \frac{ \partial}{\partial x_j} \left( A^k_i \frac{ \partial}{\partial x_k}\right)$ and where the forcing functions $f_0,f_1,f_2$ are given by
\begin{align*} 
f_0 & = - \Psi_t \cdot v \,, \\
f_1 & = -(\Psi_t \cdot v)_t + A^j_i( \partial_t A^k_i \, q,_k),_j +  \partial _t A^j_i( A^k_i \, q,_k),_j \,, \\
f_2 & = -(\Psi_t \cdot v)_{tt} + 2 A^j_i( \partial_t A^k_i \, q_t,_k),_j +  2\partial _t A^j_i( A^k_i \, q_t,_k),_j 
 +  2\partial _t A^j_i( \partial _t A^k_i \, q,_k),_j \\
 & \qquad\qquad   +  \partial^2 _t A^j_i( A^k_i \, q_t,_k),_j + A^j_i( \partial^2_t A^k_i \, q,_k),_j \,.
\end{align*} 

We will repeatedly make use of the following elliptic estimate:

\begin{lemma}[Elliptic regularity with Sobolev-class coefficients] \label{sobolevA}  
Let $q$ denote the unique $ H^1_0(\Omega) $ solution to
\begin{alignat*}{2}
-\Delta _\Psi q  & = F && \ \  \operatorname{ in } \ \ \Omega  \,,\\
q&=0 && \ \  \operatorname{ on }\ \  \partial \Omega \,.
\end{alignat*} 
Suppose that $\displaystyle{}k >1$, $F \in H^{k-1}(\Omega)$, and 
 $A \in H^k(\Omega)$ satisfying  $A^k_iA^j_i \xi _j \xi _k \ge \lambda |\xi |^2$ for all $ \xi \in \mathbb{R}  ^2$ for some $\lambda >0$.  Then
$$
\|q\|_{H^{k+1}(\Omega)} \le C \Big[\|F\|_{H^{k-1}(\Omega)} + \|A\|^p_{H^k(\Omega)} \|F\|_{L^2(\Omega)} \Big]
$$
for some power $p > 1$.
\end{lemma} 
\begin{proof} We provide the details in the course of the proof of Lemma \ref{lm:1}.
\end{proof}

\begin{lemma}[Bounding $\partial_t^lq$, $l=0,1,2,3$, by $\E(t)$]\label{lm:1}
With the bootstrap assumptions~(\ref{eq:assumption1}) and~(\ref{eq:assumption2}) holding, and  with $\epsilon>0$ sufficiently small,
there exists a constant $C^*$ such that 
\[
\|q_{ttt}\|_{0}^2 + \|q_{tt}\|_{2}^2 + \|q_{t}\|_{4}^2 + \|q\|_{6}^2
\leq C^*\E.
\]
\end{lemma}
\def\p{\partial}
\begin{proof}
{\em Step1. Estimating $|h_{ttt}|_{0.5}$}.
We denote by $ \mathcal{X}(t)$ the quantity $\|q_{ttt}\|_{0}^2 + \|q_{tt}\|_{2}^2 + \|q_{t}\|_{4}^2 + \|q\|_{6}^2$.
Twice time-differentiating (\ref{eq:ALEneumann2}), we find that
\be\label{eq:clever}
h_{ttt}=v_{tt}\cdot N-\big[\frac{h_{\theta}}{1+h}\big]_{tt}v\cdot\tau-2\big[\frac{h_{\theta}}{1+h}\big]_{t}v_t\cdot\tau
-\frac{h_{\theta}}{1+h}v_{tt}\cdot\tau.
\ee
By the normal trace theorem (see, for example, equation (6.1) in \cite{CoSh12}),
\[
|v_{tt}\cdot N|_{0.5}\lesssim\|\t v_{tt}\|_0^2+\|\div\, v_{tt}\|_0^2.
\]
Note that 
\be\label{eq:divest}
\div v_{tt}= (\div_{\Psi}\,v)_{tt}+((\div-\div_{\Psi})v)_{tt}
=(q_t+v\cdot\Psi_t)_{tt}+[(A^k_i-\delta^k_i)v^i,_k]_{tt}
=q_{ttt}+\Psi_{ttt}\cdot v + \mathcal{R},
\ee
where the remainder $\mathcal{R}$ reads
\[
\mathcal{R}= 2\Psi_{tt}\cdot v_t+\Psi_t\cdot v_{tt}+(A^k_i-\delta^k_i)_{tt}v^i,_k+2(A^k_i-\delta^k_i)_tv_{t,k}^i
+(A^k_i-\delta^k_i)v_{tt,k}^i.
\]
From Lemma~\ref{lm:basic} and~\ref{lm:useful}, we obtain the estimate
$\|\mathcal{R}\|_0^2\lesssim \epsilon \E+\epsilon \mathcal{X}$.
Thus, returning to~(\ref{eq:divest}) and using that $\|q_{ttt}+\Psi_{ttt}\cdot v\|_0^2\leq \E$ by~(\ref{eq:ALEenergy}), we get
$
\|\div v_{tt}\|_0^2\lesssim \E+\epsilon \mathcal{X} 
$
and consequently
\be\label{eq:curlest}
|v_{tt}\cdot N|_{0.5}\lesssim \E+\epsilon \mathcal{X}.
\ee
As for the last term on the right-hand side of~(\ref{eq:clever}), we use
the tangential trace theorem (see, for example, equation (6.2) in \cite{CoSh12}) to infer that
\[
|v_{tt}\cdot \tau|\lesssim \|\t v_{tt}\|_0^2 + \|\curl\, v_{tt}\|_0^2.
\]
Since $\curl_{\Psi}v=0$ (recall $v=-\nabla p\circ \Psi$), we have $\curl v_{tt}=[(\curl-\curl_{\Psi})v]_{tt}$.
By a similar inequality as above, using Lemmas~\ref{lm:basic} and~\ref{lm:useful}, we obtain the bound
$\|[(\curl-\curl_{\Psi})v]_{tt}\|_0^2\lesssim \epsilon \E+\epsilon\mathcal{X}$. Together with~(\ref{eq:curlest})
and $\|\g v_{tt}\|_0^2\leq \E$, this leads to
\[
|v_{tt}\cdot \tau|_{0.5}\lesssim \E+\epsilon \mathcal{X}.
\]
Together with the smallness of $h_{\theta}$ and $h_{\theta t}$ from Lemma~\ref{lm:basic}, the bound $|\sqrt{\chi}\t h_{tt}|_1^2\leq \E$ and Lemma~\ref{lm:useful},
we finally infer from~(\ref{eq:clever}) that
\be\label{eq:clever1}
|h_{ttt}|_{0.5}\lesssim\E +\epsilon \mathcal{X}.
\ee
{\em Step 2: $L^2$ estimates for $\partial_t^l q$.}
By the triangle inequality and the definition~(\ref{eq:ALEenergy}) of $\E(t)$, we have that for $ l=1,2,3$,
\begin{align*} 
\|\partial_t^l q\|_0^2   &  \leq\|\partial_t^lq+\partial_t^l\Psi \cdot v\|_0^2  +   \|\partial_t^l\Psi\cdot v\|_0^2  \\
& \leq \mathcal{E} (t)  +\|\partial_t^l\Psi\cdot v\|_0^2 \\
& \lesssim \mathcal{E} (t) 
+ \|v\|_{3}^2  \| \partial_t^l\Psi\|_0^2 \lesssim \E (t)+\epsilon ^2|h_{ttt}|_{0.5}^2\\
& \lesssim  \mathcal{E} (t) +\epsilon \mathcal{X} \,,
\end{align*} 
where we used the Sobolev embedding theorem and~(\ref{eq:clever1}). 

\vspace{.1 in}
\noindent
{\em Step 3: $H^2$ estimate for $q_{tt}$.}    We consider the elliptic equation $- \Delta_\Psi q = f_0-q_t$.   We note that
Lemma \ref{lm:basic2} ensures that  $A^k_iA^j_i \xi _k \xi _j \ge {\frac{1}{2}} | \xi |^2$ for all $ \xi \in \mathbb{R}  ^2$.
Given that $\|f_0 - q_t\|^2_0 \lesssim \E$, elliptic estimates show that
 $\|q\|^2_2 \lesssim \E$.   This, in turn, implies that $\|f_1 - q_{tt}\|_0^2 \lesssim \E$, and elliptic estimates then show that
 $\|q_t\|^2_2 \lesssim \E$.  Hence, we have that $\|f_2 - q_{ttt}\|_0^2 \lesssim \E+\epsilon\mathcal{X}$, and once again use elliptic estimates to conclude
 that  $\|q_{tt}\|^2_2 \lesssim \E+\epsilon\mathcal{X}$.

\vspace{.1 in}
\noindent
{\em Step 4: $H^4$ estimate for $q_{t}$.}   Since $\|f_0 -q_t\|_2^2 \lesssim \E$, Lemma \ref{sobolevA} shows that
$ \|q\|_4^2 \lesssim \E$; thus,  $\|f_1 - q_{tt}\|_2^2 \lesssim \E+\epsilon \mathcal{X}$.
Another application of Lemma \ref{sobolevA} together with Lemma \ref{lm:basic2} then shows that
$\|q_{t}\|_4^2 \lesssim \E + \epsilon \mathcal{X}$.

\vspace{.1 in}
\noindent
{\em Step 5: $H^6$ estimate for $q$.} The elliptic estimates in Steps 3 and 4 made use of Lemma \ref{lm:basic2}.   To obtain the
$H^6$ estimate for $q$ requires us to improve the elliptic estimate in Lemma \ref{sobolevA} to be linear in $\| \sqrt{\chi} \Psi\|_6$.
To this end, we write $ \mathcal{A} ^{jk} = A^j_i A^k_i$ and rewrite (\ref{heatss}a) as
\begin{equation}\label{ss77}
-( \mathcal{A} ^{jk} q,_k),_j = -q_t + f_0 - A^j_i,_j A^k_i q,_k \,.
\end{equation} 
Letting $\bp^\alpha$ act on (\ref{ss77}), we find that $\bp^\alpha q$ satisfies
\begin{align*} 
-\Big[\mathcal{A} ^{ij} (\bp^\alpha q),_j\Big],_i & = - \bp^\alpha (\Psi_t \cdot v +q_t) + \sum_{0<\beta\le \alpha} C_{\alpha\beta} \Big[(\bp^\beta \mathcal{A} ^{ij})(\bp^{\alpha-\beta} q),_j\Big],_i 
\\ & \qquad\qquad
- \sum_{0\le \beta < \alpha} C_{\alpha\beta} \bp^\beta \Big(A^j_i,_j A^k_i\Big) \bp^{\alpha-\beta} q,_j  \,,
\end{align*} 
where $C_{\alpha \beta }$ are constants from the product rule.
Multiplying this equation with $\bp^ \alpha q$ and integrating-by-parts, using the fact that $\bp^ \alpha q=0 $ on $\partial \Omega$ and that $ \mathcal{A} \ge 1/2$, we find that
\begin{align} 
{\frac{1}{2}} \| \bar \p^ \alpha q \|_1^2 &   \le \| \bp^ {\alpha-1} (\Psi_t\cdot v +q_t)\|_0 \| \bp^ {\alpha +1} q\|_0
+ \sum_{0<\beta\le \alpha} C_{\alpha\beta} \Big\| (\bp^\beta \mathcal{A} ^{ij})(\bp^{\alpha-\beta} q),_j\Big\|_0 \| \bp^ \alpha q,_i\|_0 \nonumber \\
& \ \ +  \sum_{0\le \beta < \alpha} C_{\alpha\beta} \Big\| \bp^\beta \Big(A^j_i,_j A^k_i\Big) \bp^{\alpha-\beta} q,_j \Big\|_0 \| \bp^ \alpha q\|_0 
+  \Big\| \bp^{ \alpha -1} \Big(A^j_i,_j A^k_i\Big)  \bp q,_j \Big\|_0 \| \bp^ \alpha q\|_1
\,. \label{ss709}
\end{align} 
Let us examine the second term on the right-hand side of (\ref{ss709}).   By Young's inequality, for $ \delta >0$, 
\begin{align*}
 \sum_{0<\beta\le \alpha} C_{\alpha\beta} \Big\| (\bp^\beta \mathcal{A} ^{ij})(\bp^{\alpha-\beta} q),_j\Big\|_0 \| \bp^ \alpha q,_i\|_0
& \le  \delta  \| \bp^ \alpha q\|_1^2 + C_ \delta  \sum_{0<\beta\le \alpha} C_{\alpha\beta} \Big\| \bp^\beta \mathcal{A}  \ \bp^{\alpha-\beta} Dq\Big\|_0^2
\end{align*} 
where $C_\delta = C/ \delta $.     Since $\bp^5 \mathcal{A} \sim \bp^5 D\Psi\, P(A) + \bp^4 D\Psi\, P(\bp D\Psi,A) + \bp^3 D\Psi\, P(\bp^2D\Psi, \bp D\Psi, A)$, 
it thus follows that for $\alpha = 4 \text{ or } 5$,
\begin{equation}\label{ss701}
\| \bp^\alpha  \mathcal{A}\|_0 \le C \|\bp^ {\alpha -2} (\Psi-e)\|^2_3 \le C\frac{ | \sqrt{ \chi} h|^2_{5.5} }{ {\chi}} \lessim  \frac{ \E }{ {\chi}}\,.
\end{equation} 
The linear inequality (\ref{ss701}) shows that our bootstrap assumptions~(\ref{eq:assumption1}) and~(\ref{eq:assumption2}) imply that the map $h \mapsto \mathcal{A} $ is
linear with respect to these high norms.  

We first consider the case that $ \alpha =4$.   From (\ref{ss701})
when $ \alpha= \beta =4$
\begin{align}
\| \bp^ \alpha   \mathcal{A} \ D q\|_0^2 \lessim \frac{E_\beta e^{-\beta t}}{\chi} \mathcal{E}  \lessim \epsilon e^{-\gamma t} \mathcal{E}  \,. \label{ss721}
 \end{align} 
The Cauchy-Schwarz inequality, together with the Sobolev embedding theorem, shows that  $\| \bp^3 \mathcal{A} \ \bp  Dq\|_0^2$ has the same bound.
Next, $\| \bp^2 \mathcal{A} \ \bp^2  Dq\|_0^2 
+ \| \bp \mathcal{A} \ \bp^3  Dq\|_0^2  \lessim  \|\Psi\|_4^2 \|q\|_4^2 \lessim \epsilon e^{- \beta t} \mathcal{E}  \lessim \epsilon e^{- \gamma t} \mathcal{E} $.

The first, third, and fourth terms on the right-hand side of (\ref{ss709}) are estimated in a similar fashion, so we do not provide the details.  Hence,   by choosing $ \delta >0$ sufficiently
small and employing Young's inequality, we  find that
$$
\|q\|_4^2 + \sum_{ \alpha \le 4} \| \bar \p^ \alpha q \|_1^2  \lessim \mathcal{E} +\epsilon\mathcal{X} \,.
$$

To estimate radial derivatives, we use
 polar coordinates for the disc (with the usual basis ${\bf e_r}$ and ${\bf e_\theta}$).   Expressing  the components of the matrix $ \mathcal{A} $ as
$$
\mathcal{A} = \left[
\begin{array}{lr}
\mathcal{A} ^{rr} & \mathcal{A} ^{r \theta } \\ 
\mathcal{A} ^{\theta r} &\mathcal{A} ^{ \theta \theta }
\end{array}\right]\,,
$$
we may write 
$$
\operatorname{div} ( \mathcal{A}  \ \nabla q) = r^{-1}  (r \mathcal{A} ^{rr} q_r)_r +  r^{-1}  ( \mathcal{A} ^{r\theta } q_ \theta )_r +  r^{-1}  ( \mathcal{A} ^{r\theta } q_r)_\theta 
+  r ^{-1} ( r^{-1} \mathcal{A} ^{ \theta \theta } q_\theta )_\theta  \,.
$$

It follows that 
\begin{align} 
- \mathcal{A} ^{rr} \bp^\alpha q_{rr} & =  r ^{-1} ( r\mathcal{A} ^{rr})_r \bp^ \alpha q_r  +  r^{-1}  ( \mathcal{A} ^{r\theta } \bp^ \alpha q_r)_r +  r^{-1}  ( \mathcal{A} ^{r\theta } \bp^ \alpha q_r)_\theta +  r ^{-1} ( r^{-1} \mathcal{A} ^{ \theta \theta } \bp^ \alpha q_\theta )_\theta  \nonumber
 \\
& \qquad\qquad 
- \bp^\alpha (\Psi_t \cdot v +q_t) + \sum_{0<\beta\le \alpha} C_{\alpha\beta} \Big[(\bp^\beta \mathcal{A} ^{ij})(\bp^{\alpha-\beta} q),_j\Big],_i  \nonumber
\\ & \qquad\qquad
- \sum_{0\le \beta \le \alpha} C_{\alpha\beta} \bp^\beta \Big(A^j_i,_j A^k_i\Big) \bp^{\alpha-\beta} q,_j  \,, \label{ss710}
\end{align} 

Let $ \omega = \{ x \in \Omega\,\,\,: \,\,\, {\frac{1}{2}} < |x| < 1\}$.
For $ \alpha \le 3$, every term on the right-hand side has $L^2(\omega)$-norm bounded by a constant multiple of $ \mathcal{E} $.  Hence, it follows that
$$
\sum_{ \alpha \le 3} \| \bp ^ \alpha q\|_{2, \omega }^2 \lessim \E +\epsilon \mathcal{X} \,.
$$
Allowing $\frac{\p}{\p r}$ to act on (\ref{ss710}), as many as three times, we conclude that
\begin{equation}\label{ss720}
\| q\|_{5, \omega }^2 \lessim \E +\epsilon\mathcal{X}\,.
\end{equation}

We return to the inequality (\ref{ss709}) and consider the case that $ \alpha =5$.  Once again, we focus on the second term on the right-hand side,
the first and third terms being similar (and easier).
 From (\ref{ss721}) $\| \bp^ 5  \mathcal{A} \ D q\|_0^2  \lessim \epsilon e^{-\gamma t} \mathcal{E}$.
The Cauchy-Schwarz inequality, together with the Sobolev embedding theorem, shows that  $\| \bp^4 \mathcal{A} \ \bp  Dq\|_0^2
+\| \bp^3 \mathcal{A} \ \bp^2  Dq\|_0^2 + \| \bp^2 \mathcal{A} \ \bp^3  Dq\|_0^2  \lessim \epsilon e^{-\gamma t} \mathcal{E}$.  Finally,  using (\ref{ss720}), we conclude
$ \| \bp \mathcal{A} \ \bp^4  Dq\|_0^2 \lessim \epsilon \|v\|_4^2 \lessim \epsilon e^{-\gamma t} \mathcal{E}$.   
We conclude that
$$
\|q\|_4^2 + \sum_{ \alpha \le 5} \| \bar \p^ \alpha q \|_1^2  \lessim \mathcal{E} +\epsilon \mathcal{X}\,.
$$
Then setting $ \alpha =0$ and letting $\frac{\p^4}{\p r^4}$ act on  (\ref{ss710}) shows that indeed
$$
\| q\|_{6,\omega }^2 \lessim \E+\epsilon \mathcal{X}\,.
$$
By using a smooth cut-off function whose support contains $\overline{\Omega\setminus \omega }$, we easily obtain the interior estimates, and find that
$ \|q\|_6^2 \lessim \E=\epsilon \mathcal{X}$.
Recalling the definition of $\mathcal{X}$ and the estimates from Steps 2, 3, and 4, we finally infer
$\mathcal{X}\lesssim \E$, which concludes the proof of the lemma.
\end{proof}
\begin{lemma}[Bounding $\partial_t^lq$, $l=0,1,2,3$ by $\D(t)$]\label{lm:2}
With the bootstrap assumptions~(\ref{eq:assumption1}) and~(\ref{eq:assumption2}),and for $\epsilon>0$ sufficiently small,
there exists a $\gamma>0$ such that
\be\label{eq:este}
\sum_{l=0}^2\|\g_t^lq_t\|_{5-2l}^2 + \| q\|_{6.5}^2\lesssim 
\epsilon e^{-\gamma t}\E+\D.
\ee
\end{lemma}
\begin{corollary}\label{co:eq}
 With the  bootstrap assumptions~(\ref{eq:assumption1}),~(\ref{eq:assumption2}) and a sufficiently small $\epsilon>0$,
  \[
 \|v\|_{5.5}^2 + |h_t|^2_5\lesssim  \epsilon e^{-\gamma t}\E+\D
 \]
 with $\gamma=\beta / 2 - \eta$ as defined in Lemma~\ref{lm:useful}.
\end{corollary}
\begin{proof}[Proof of Corollary \ref{co:eq}]
We write (\ref{eq:ALEv}) as
$$
v = Dq \cdot ( \operatorname{Id} -A) - Dq \,.
$$
Using the basic estimate from Lemma \ref{lm:basic2}, we see that
\begin{align*} 
\|v \|_5 & \lesssim (1+ \sqrt{ \epsilon }) \| q\|_6  + E_\beta^ {\frac{1}{2}} e^{-\beta t / 2}  \|\Psi-e\|_6 \,, \\
\|v \|_6 & \lesssim (1+ \sqrt{ \epsilon }) \| q\|_7  + E_\beta^ {\frac{1}{2}} e^{-\beta t / 2} \|\Psi-e\|_7 \,,
\end{align*} 
so that an application of linear interpolation (see, for example, Theorem 7.17 in Adams \cite{Ad}) provides the
inequality
$$
\|v \|^2_{5.5}  \lesssim (1+ \sqrt{ \epsilon }) \| q\|_{6.5}^2  + E_\beta  \|\Psi-e\|^2_{6.5} \,.
$$
Using Lemmas \ref{lm:useful} and \ref{lm:2}, it follows that
\begin{align*} 
\|v \|^2_{5.5} &  \lesssim (1+ \sqrt{ \epsilon }) \| q\|_{6.5}^2  + \frac{E_\beta(t) e^{-\beta t / 2} }{\chi(t)} \, \chi(t) \|\Psi-e\|^2_{6.5} \\
&  \lesssim (1+ \sqrt{ \epsilon }) \| q\|_{6.5}^2  + \frac{E_\beta(t) e^{-\beta t / 2} }{\chi(t)} \, \mathcal{E} (t) \\
& \lesssim  \epsilon ^2 e^{-\gamma t} \mathcal{E} + \mathcal{D}  \,.
\end{align*} 

Next, using the formula (\ref{eq:velformula}), we see that
$$
|h_t|^2_5 \lesssim \chi(t) |h|^2_6 \, \frac{|v|^2_{2.5}}{\chi(t)} + \epsilon |v|_5^2
$$
which once again, thanks to Lemmas \ref{lm:useful} and \ref{lm:2}, is bounded by a constant multiple of $\epsilon ^2 e^{-\gamma t} \mathcal{E} + \mathcal{D}$.
\end{proof}

%
\begin{proof}[Proof of Lemma \ref{lm:2}]
{\em Step 1: $H^1$ estimates for $\partial_t^l q$.}  We make use of the identity $ \nabla q = v \cdot \nabla \Psi$.   It follows that
\begin{align*}
  \nabla q_t  & = v_t \cdot  \nabla \Psi + v\cdot\nabla \Psi_t \\
  \nabla q_{tt}  & = v_{tt} \cdot  \nabla \Psi + 2 v_t \cdot\nabla \Psi_t  +  v \cdot\nabla \Psi_{tt} \\
  \nabla q_{ttt}  & = v_{ttt} \cdot  \nabla \Psi + 3 v_{tt} \cdot\nabla \Psi_t  + 3 v_{t} \cdot\nabla \Psi_{tt}  +  v \cdot\nabla \Psi_{ttt}     \,.
\end{align*} 
Employing H\"{o}lder's inequality and the Sobolev embedding theorem, 
\begin{align*} 
\|  \nabla q_{ttt}\|^2_0   &  \lessim \|v_{ttt}\|^2_0\  |h|^2_2 +  \|v_{tt}\|^2_0\  |h_t|^2_2 + \frac{ \|v_{t}\|^2_2}{{\chi}}\  |\sqrt{\chi} h_{tt}|^2_0 +\frac{ \|v\|^2_2}{{\chi}} \ |\sqrt{\chi}h_{ttt}|^2_0 
\lessim \D
\end{align*} 
where we have used Lemma \ref{eq:useful} for the last inequality.  We have similar estimates for $ q_{tt}$, $q_t$, and $q$ so that
\begin{equation}\label{ss711}
\sum_{l=0}^3 \| \p_t^l q\|_1^2 \lessim \D.
\end{equation} 

\vspace{.1 in}
\noindent
{\em Step 2. $H^3$ estimate for $q_{tt}$.} Just as in the proof of Corollary \ref{co:eq}, we see that as a consequence of Lemma \ref{lm:1},
\begin{equation}\label{ss725}
\sum_{l=0}^2 \| \p_t^l v\|_{5-2l} \lessim \E \,.
\end{equation} 
Returning to the equation (\ref{heatss}a), we estimate $-\Psi_t \cdot v - q_t$ in $ H^1(\Omega) $.
By the Sobolev embedding theorem together with Lemmas \ref{lm:useful0} and \ref{lm:useful},  
\begin{equation}\label{ss722} \| \Psi_t  \|_{W^{1, \infty }} \lessim \|\Psi_t \|_3 \lessim \sqrt{ \epsilon } e^{-\gamma t/2} \,,
\end{equation} 
so that together with  (\ref{ss725}),  $\| \Psi_t \cdot v \|_1^2 \lessim \epsilon e^{\gamma t} \E$. Then, with
with (\ref{ss711}),
\begin{equation}\label{ss712} \| q \|_3^2 \lessim \epsilon e^{-\gamma t} \E + \D \,. \end{equation} 

Next, we return to (\ref{heatss}b) and estimate $f_1 - q_{tt}$ in $ H^1(\Omega) $.   By Lemma \ref{lm:useful0}, $\| \Psi_t \cdot v_t \|_1^2 \lessim \epsilon e^{-\gamma t} \E$,
while $\| \Psi_{tt} \cdot v\|_1 ^2 \lessim \frac{\E e^{-\beta t}}{\chi} E_\beta \lessim  \epsilon e^{-\gamma t} \E $.
The estimates (\ref{ss711}) and  (\ref{ss712}) then show that $\|  f_1 - q_{tt} \|_1^2 \lessim  \epsilon e^{-\gamma t} \E + \D$ so that
\[
 \| q_t \|_3^2 \lessim \epsilon e^{-\gamma t} \E + \D \,. 
\]
A similar estimate then shows that $\|  f_2 - q_{ttt} \|_1^2 \lessim  \epsilon e^{-\gamma t} \E + \D$ so that from (\ref{heatss}c), 
\[
 \| q_{tt} \|_3^2 \lessim \epsilon e^{-\gamma t} \E + \D \,. 
\]

\vspace{.1 in}
\noindent
{\em Step 3. $H^5$ estimate for $q_{t}$.}  From (\ref{ss725}) and (\ref{ss722}), we see that $\| \Psi_t \cdot v\|_3^2  \lessim \epsilon e^{-\gamma t} \E + \D$, so that
with Lemmas \ref{lm:basic2} and \ref{sobolevA}, we have that
\[
 \| q \|_5^2 \lessim \epsilon e^{-\gamma t} \E + \D \,. 
\]
This, in turn, ensures that $\|  f_1 - q_{tt} \|_3^2 \lessim  \epsilon e^{-\gamma t} \E + \D$ 
so that
\[
 \| q_t \|_5^2 \lessim \epsilon e^{-\gamma t} \E + \D \,. 
\]

\vspace{.1 in}
\noindent
{\em Step 4. $H^{6.5}$ estimate for $q$.}  We first look at the estimate (\ref{ss709}) with $ \alpha =5$.   We find that
\begin{align} 
 \| \bar \p^ 5 q \|_1 &   \lessim \| \Psi_t\cdot v\| _4 + \| q_t\|_4
+ \sum_{0<\beta\le 5} \Big\| (\bp^\beta \mathcal{A} ^{ij})(\bp^{5-\beta} q),_j\Big\|_0 \nonumber \\
& \ \ +  \sum_{0\le \beta < 5}  \Big\| \bp^\beta \Big(A^j_i,_j A^k_i\Big) \bp^{\alpha-\beta} q,_j \Big\|_0 
+  \Big\| \bp^{ 4} \Big(A^j_i,_j A^k_i\Big)  \bp q,_j \Big\|_0 \label{ss740}
\end{align} 
For the first term on the right-hand side, we note that with the Sobolev embedding theorem and Lemma~\ref{lm:useful}, 
\begin{align*} 
\| \Psi_t \cdot v \|_k & \lessim \|\sqrt{\chi} \Psi_t\|_k\frac{ \| v\|_3}{\sqrt{\chi}} +  \|\Psi_t\|_3 \|v\|_k \\
& \lessim  \sqrt{ \epsilon } e^{ - \gamma t/2}( \|\sqrt{\chi} \Psi_t\|_k+   \|v\|_k)
 \qquad k=4,5 \,.
\end{align*} 
Using the estimate (\ref{ss701}), we see that
\begin{align*}
\sum_{0<\beta\le 5} \Big\| (\bp^\beta \mathcal{A} ^{ij})(\bp^{5-\beta} q),_j\Big\|_0 \lessim \sqrt{ \epsilon } 
e^{ -\gamma t/2}( \|\sqrt{\chi}(\Psi-e)\|_6+   \|q\|_5)
 \end{align*} 
 The last two term on the right-hand side of (\ref{ss740}) are estimated in the same way so that
 \begin{align*} 
  \| \bar \p^ 5 q \|_1 \lessim  \sqrt{ \epsilon } e^{ -\gamma t/2}( \|\sqrt{\chi} \Psi_t\|_4 + \|\sqrt{\chi} (\Psi-e)\|_6+   
  \|v\|_4 + \| q_t\|_4) \,.
\end{align*} 
Using the formula (\ref{ss710}), we find that 
 \begin{align} 
  \| q \|_6 \lessim  \sqrt{ \epsilon } e^{- \gamma t/2}( \|\sqrt{\chi} \Psi_t\|_4 + \|\sqrt{\chi} (\Psi-e)\|_6+   
  \|v\|_4 + \| q_t\|_4)+\|q_t\|_4 \,. \label{ss741}
\end{align} 
The identical procedure with $ \alpha =6$ then yields
 \begin{align} 
  \| q \|_7 \lessim  \sqrt{ \epsilon } e^{ -\gamma t/2}( \|\sqrt{\chi} \Psi_t\|_5 + \|\sqrt{\chi} (\Psi-e)\|_7+   \|v\|_5 + \| q_t\|_5)
  +\|q_t\|_5\,. \label{ss742}
\end{align} 
Linear interpolation between (\ref{ss741}) and (\ref{ss742}), we have that
\begin{align*} 
  \|  q \|_{6.5}&  \lessim  \sqrt{ \epsilon } e^{ -\gamma t/2}( \|\sqrt{\chi} \Psi_t\|_{4.5} + \|\sqrt{\chi} (\Psi-e)\|_{6.5}+   
  \|v\|_{4.5} + \| q_t\|_{4.5})  +\|q_t\|_{4.5} \\
& \lessim \epsilon e^{-\gamma t/2} \E^{1/2} + \D^{1/2} \,.
\end{align*} 
\end{proof}

\subsection{Lower bound on $\chi(t)$}
The heat equation (\ref{eq:ALEheat})  for $q$ can be rewritten as
\begin{subequations}
\label{eq:qparabolic}
 \begin{alignat}{2}
q_t-a_{kj}q,_{kj}-b_kq,_k&=0&&\,\text{ in }\,\Omega,\label{eq:parabolic}\\
q&=0&&\,\text{ on }\,\Gamma,\label{eq:dirichletc}\\
q(0,\cdot)&=&&q_0>0\,\text{ in }\,\Omega
\end{alignat}
\end{subequations}
where the coefficient matrix $a=(a_{kj})_{k,j=1,2}$, and the vector $b=(b_1,b_2)$ are explicitly given by:
\be\label{eq:coeff}
a_{kj}:=A^k_iA^j_i;\quad b_k:=A^k_{i,j}A^j_i+A^k_i\Psi^i_t.
\ee
We first quote a theorem from~\cite{Od}, that will play an important role in producing quantitative bounds
from below for $\chi(t)$.
\begin{lemma}[Oddson's Theorem 2 in \cite{Od}]\label{lm:oddson}
Let $q\in C^{1,2}(\Omega)$ be a supersolution to~(\ref{eq:qparabolic}) in the unit disc $\Omega=B_1({\bf 0})$, and
let $0<\alpha\leq\frac{1}{2}$ be the normalized ellipticity constant satisfying
\[
a_{jk}\xi_j\xi_k\geq\alpha\big(a_{11}+a_{22}\big)|\xi|^2
\]
for any real vector $\xi=(\xi_1,\xi_2)$.
Moreover, let us introduce the quantities
\[
k_0(T):=\inf_{\Omega\times[0,T]}\frac{1}{a_{11}+a_{22}},
\quad
\beta(T):=\sup_{\Omega\times[0,T]}b\cdot x.
\]
Let $J_{\mu}$ denote the Bessel function of the first kind of order $\mu$ and $\xi_0$ its first positive zero.
If we define 
\[
\mu=\frac{\beta+1}{2\alpha}-1,\quad \lambda=\frac{\alpha\xi_0^2}{k_0},
\] 
then there exists a positive constant $m$ satisfying 
\[
q(t,x)\geq m\rho e^{-\lambda t},
\]
in $B_1(0)\times[\sigma ,\infty[$,
where $\rho$ stands for the distance from $x$ to the boundary $\Gamma$
and $ \sigma $ is an arbitrary small time.
\end{lemma}
\begin{remark}[Optimal decay rate for solutions of the heat equation]\label{re:optimal}
If we set $A=Id$, then problem~(\ref{eq:qparabolic}) turns into 
the initial-boundary value problem for the linear heat equation. 
In this case $k_0=\frac{1}{2}$, $\alpha=1/2$, $\beta\equiv0$,
$\mu=\frac{0+1}{1}-1=0$, and $\lambda=\xi_0^2$, where $\xi_0$ stands for the first positive zero
of $J_{0}(\xi)$. 
In particular, if $q^{\text{heat}}$ denotes the associated solution, then the above lemma implies that
\[
\chi_{\text{heat}}(t):=\inf_{x\in\Gamma}(-\g_Nq^{\text{heat}}(t,x))\gtrsim e^{-\xi_0^2t},
\]
which is the {\em optimal} decay rate in the case of the linear heat equation, as 
the lowest positive eigenvalue of the Dirichlet-Laplacian on the two-dimensional disk corresponds exactly 
to 
\[
\lambda_1=\xi_0^2.
\]
\end{remark}

\begin{corollary}[Lower-bound for $\chi(t)$] \label{co:oddson}
Under the bootstrap assumptions~(\ref{eq:assumption1}) and~(\ref{eq:assumption2}) with
$\epsilon$ small enough,
there exists a universal constant $C>0$ such that
\[
\chi(t)\gtrsim c_1e^{-(\lambda_1+\tilde{\lambda}(t))t},
\]
where $c_1=\int_{\Omega}q_0\varphi_1\,dx$
is the first
coefficient in the eigenfunction expansion of the initial datum $q_0$ with respect to
the $L^2$ ortho-normal basis $\{\varphi_1,\varphi_2,\dots\}$ of the eigenvectors of the Dirichlet-Laplacian
on $B_1(0)$,
i.e $q_0=c_1\varphi_1+c_2\varphi_2+\dots$.
Moreover, $\tilde{\lambda}(t) \ge 0$  satisfying
$
\tilde{\lambda}(t) \le C\epsilon
$
for some positive constant $C$. In particular, with $\epsilon>0$ sufficiently small so that $C\epsilon<\eta/4$, we obtain the improvement
of the  
bootstrap bound~(\ref{eq:assumption2}) given by $\chi(t)\gtrsim c_1e^{-(\lambda_1+\eta/4)t}$.
\end{corollary}
\begin{proof}
The proof of Oddson's Theorem 2 in \cite{Od} (Lemma~\ref{lm:oddson}) relies on the construction of a comparison function of the form
$v(t,r)=r^{-\mu}J_{\mu}(\xi_0r)e^{-\lambda t}$, where $\lambda,\mu,\xi_0$ are given in the statement of Lemma~\ref{lm:oddson},
$J_{\mu}$ is a Bessel function of the first kind and $r=|x|$ is the radial coordinate.
The first property of $v$ which is important for the proof is that $v$ vanishes at 
the spatial boundary $\Gamma$ and approaches it 
like $c(1-r)e^{-\lambda t}$
as $r\to1$. This is a consequence of the fact that 
$\lim_{r\to1}\frac{J_{\mu}(\xi_0r)}{r^{\mu}(1-r)}=c$ for some constant $c>0$, a well known property of Bessel functions.
The second important property is that $v$ is a subsolution for~(\ref{eq:qparabolic}) (and it is constructed with the help of 
maximal Pucci operators as explained in detail in~\cite{Od}).

The goal is to prove that for any arbitrarily small time $\sigma>0$ there exists a strictly positive constant $\delta(\sigma)>0$ such that
$q-\delta v $ is a {\em positive} supersolution to the parabolic problem~(\ref{eq:qparabolic}) on the time interval $[\sigma,\infty[$.
The desired lower bound for $q$ then follows from 
 the weak maximum principle.

Since $v$ is a subsolution, it follows that for any $\delta>0$, $q-\delta v$ is a supersolution.  The positivity of  $q-\delta v$ at $t= \sigma $ 
follows from the parabolic Hopf lemma, from which we infer the existence of a
 constant $\delta( \sigma )$ such that $\frac{q}{v}>\delta( \sigma )$ uniformly over
$\bar{\Omega}$. Note that we  have used the fact that $v(\sigma,r)$ behaves like $C(1-r)$ near the boundary $\Gamma$  for some positive constant $C$.
Therefore it follows that the constant $m$ in the statement of Lemma~\ref{lm:oddson} {\it a priori} depends on the time $\sigma>0$,
and moreover, $m$ is proportional to the lower bound for $-\partial q/ \partial N|_{t=\sigma}$ on $\Gamma$.

From the proof of the parabolic Hopf lemma (see for instance Theorem 3.14 in~\cite{Fr64}), the 
value $-\partial q/ \partial N|_{t=\sigma}$ 
is proportional to the minimal value of the temperature $q$ on a space-time region
of the form
$K_{\sigma}:=B_{1-C\sigma}\times [\sigma/2,3\sigma/2]$,
divided by $\sigma$ (which is roughly the distance of $K_{\sigma}$ from the parabolic boundary of $\Omega\times[0,2\sigma]$).
Note that, unlike the elliptic case, we are forced to take into account the time-dependence of the solution
and in particular the region $K_{\sigma}$ cannot be chosen uniformly for all times, but only for times greater or equal some arbitrarily
small $\sigma>0$.
However, our solution is continuous all the way to $t=0$ and we do nevertheless obtain a lower bound for all times due to
the Taylor sign condition; namely, due to (\ref{eq:wlog}),
\[
-\g_Nq_0=\frac{-\g_Nq_0}{c_1}c_1\gtrsim c_1.
\]
Note however that if we define the dimensionless quantity $L=(-\g_Nq_0)/c_1>0$ and assume no universal bound on $L$ from below, 
the only modification in the statement of the main theorem will be that the smallness condition on initial data~(\ref{eq:smallassum}) 
will additionally depend on $L$.

As to the bound on $\tilde{\lambda}$, note that the exponent 
$\lambda=\lambda((a_{ij}),(b_i))$ depends on the coefficients $(a_{ij})_{i,j=1,2}$ and $(b_i)_{i=1,2}$
through the relationship $\lambda=\alpha\xi_0^2/k_0$.
Since $k_0$ and $\xi_0$ vary continuously as the coefficients are varied,
it proves that $\lambda$ depends continuously on the coefficients $a_{ij}, b_i$ of the parabolic operator.
On the other hand, by Remark~\ref{re:optimal}
it follows $\lambda|_{a_{ij}=\delta_{ij},b_i=0}=\lambda_1$.
As a consequence
\[
|\tilde{\lambda}(t)|=|\lambda(t)-\lambda_1|\leq C(\|A-\text{Id}\|_{L^{\infty}},\|b\|_{L^{\infty}})
=O(\|D^2(\Psi-e)\|_{L^{\infty}},\|\Psi_t\|_{L^{\infty}}).
\]
\end{proof}

\section{Energy identity and the higher-order energy estimate}\label{se:long}

\subsection{The energy identity}

Much of our analysis is founded on basic  higher-order energy identities for
the classical Stefan problem. These identities provide the geometrical control
of the evolving phase boundary, which in turn controls the decay of the temperature function;
moreover, these identities  
explain our definition of  the higher-order energy function $\E$ and the  dissipation function $\D$.

\begin{proposition}[Energy identity]\label{lm:lemma1} With $R=1+h$ and $R_J=RJ^{-1}$, sufficiently smooth solutions to the classical Stefan problem satisfy
\begin{align} 
\frac{d}{dt} \mathcal{E} (t) + \mathcal{D} (t) &  =\sum_{j=0}^3\left(
\int_{\Gamma}(-\g_Nq_t)R_J^2|\t^{6-2j}\g_t^jh|^2
+\int_{\Omega}(\mathcal{R}_j
+\tilde{\mathcal{R}}_j)
+\int_{\Gamma}\mathcal{G}_j \right) 
\nonumber \\
& \qquad\qquad + \sum_{j=1}^3 \left(\int_{\Omega}(\mathcal{S}_j+\tilde{\mathcal{S}}_j)
+\int_{\Gamma}\mathcal{H}_j\right)
\,, \label{eq:enidentity}
\end{align} 
where the error terms $\mathcal{R}_j$, 
$\tilde{\mathcal{R}}_j$, $\mathcal{S} _j$, $\tilde {\mathcal{S} }_j$, 
$\mathcal{G}_j$, and $ \mathcal{H} _j$ are given by~(\ref{eq:remainderj}), ~(\ref{eq:remaindergammaj}), (\ref{eq:sremainder}),  ~(\ref{eq:sjtilde}),
~(\ref{eq:rjtilde}), and  ~(\ref{eq:sremaindergamma}), respectively.
\end{proposition}
The proof is provided in Appendix \ref{se:derivation}.

\begin{remark}
 On the right-hand side  of (\ref{eq:enidentity}), we have isolated the 
 error term 
\begin{equation}\label{ssss1}
\mathcal{G} _{\operatorname{Hopf} } = \int_{\Gamma}(-\g_Nq_t)R_J^2|\t^{6-2j}\g_t^{j}h|^2\,dx',
\end{equation} 
from  the other boundary-integral error terms $\mathcal{G}_{j}$ and $ \mathcal{H} _j$; indeed, 
$\mathcal{G} _{\operatorname{Hopf} } $ can only be thought of as an 
 ``error term'' on a transient time-interval, for  after a sufficiently large time, we will
 no longer be able to control $\mathcal{G} _{\operatorname{Hopf} } $ via
  energy methods, and instead, we have to rely upon  a Hopf-type argument to
 prove  that $\mathcal{G} _{\operatorname{Hopf} } < 0$.
 \end{remark}

\subsection{Energy estimates}
To control some of the highest-order error terms in our energy estimates, we shall make use of the following technical lemma, whose
proof is given in \cite{CoSh07} and \cite{CoSh10}.
\begin{lemma} \label{technical}
Let $H^ {\frac{1}{2}} (\Omega)'$ denote the dual space of $H^ {\frac{1}{2}} (\Omega)$.  There exists a positive constant $C$ such that
$$
\| \bar \partial F\|_{^ {\frac{1}{2}} (\Omega)'} \le C \|F\|_{ H^ {\frac{1}{2}} (\Omega)} \text{ for } F \in H^ {\frac{1}{2}} (\Omega) \,.
$$
\end{lemma} 
As a consequence of the energy identity  (\ref{eq:enidentity}), we can establish our fundamental energy inequality.
\begin{proposition}[The energy estimate]\label{lm:apriori1}
 Suppose that the bootstrap 
 assumptions~(\ref{eq:assumption1}) and~(\ref{eq:assumption2}) hold with $\epsilon>0$ and $\eta>0$ sufficiently small.  Letting
  $K=\frac{\|q_0\|_4}{\|q_0\|_0}$, 
 \be\label{eq:basicest}
 \sup_{0\leq s\leq t}\E(s)+\frac{1}{2}\int_0^t\D(s)\,ds\leq\E(0)+CK^2\int_0^te^{\eta s}\E(s)\,ds+
 O(\sqrt{\epsilon} )\sup_{0\leq s\leq t}\E(s) \ \text{ for } \ t \in [0,T] \,.
 \ee
\end{proposition}
\begin{proof}
Throughout the proof, we will rely on the {\it a priori} bounds of Section~\ref{se:basic}; in particular, we will often
make use of
Lemmas~\ref{lm:useful},~\ref{lm:basic},~\ref{lm:1}, and~\ref{lm:2}. 

\vspace{.1 in}
\noindent
{\em Step 1. The estimate for $\mathcal{G} _{\operatorname{Hopf} }$ in (\ref{ssss1})}
 We claim that
\be\label{eq:error1}
|\mathcal{G} _{\operatorname{Hopf} }|
\leq CK^2\int_0^te^{\eta s}\E(s)\,ds.
\ee
Note that
$$
\Big|\int_0^t\int_{\Gamma}(-\g_Nq_t) R_J^2 \big|\t^6h\big|^2\Big|
\le C
\Big|\int_0^t\int_{\Gamma}\frac{(-\g_Nq_t)}{-\g_Nq}(-\g_Nq)\big|\t^6h\big|^2\Big|
\le C
\int_0^t\Big|\frac{\g_Nq_t}{-\g_Nq}\Big|_{L^{\infty}}\E(s)\,ds.
$$
In order to bound the term $\big|\frac{\g_Nq_t}{\g_Nq}\big|$, we need a decay estimate
for the numerator $|\g_Nq_t|$. 
The Sobolev embedding theory would yield the bound $|\g_Nq_t|_{L^{\infty}}\lesssim\|q_t\|_{2 + \delta }$ for 
$ \delta >0$, but by definition of our decay norm $E_\beta$, it is only the $H^2(\Omega)$-norm of $q_t$ for which
we have the desired decay.   Thus, we arrive at the decay estimate for $q_t$ by using a
comparison principle together with Theorem 1 in Oddson \cite{Od};  indeed,  in Appendix  \ref{se:barrier}, we prove that
\be\label{eq:barrier2}
|\g_Nq_t|_{L^{\infty}}\lesssim K^2c_1e^{-\beta t/2}.
\ee
It then follows from the bootstrap assumption~(\ref{eq:assumption2}) that
\beas
\Big|\frac{\g_Nq_t(s)}{-\g_Nq(s)}\Big|_{L^{\infty}}
\leq \frac{CK^2c_1e^{-(\lambda_1-\eta/2)s}}{c_1e^{-(\lambda_1+\eta/2)s}}
\leq CK^2e^{\eta s} \,,
\eeas
which, in turn,  establishes (\ref{eq:error1}).

\vspace{.1 in}
\noindent
{\em Step 2. Estimates for  $ \mathcal{R} _j$, $\tilde{ \mathcal{R} _j}$, and $\mathcal{G} _j$ in (\ref{eq:enidentity}).}
Our objective will be to show that
\be\label{eq:error2}
\Big|\int_0^t\int_{\Omega}(\mathcal{R}_j+\tilde{\mathcal{R}}_j)
+\int_0^t\int_{\Gamma}\mathcal{G}_j\Big|
\leq O(\sqrt{\epsilon} )\sup_{0\leq s\leq t}\E(s)+\delta\int_0^t\D(s), \ \text{ for }   j=0,\dots 3.
\ee

We establish (\ref{eq:error2}) for the most difficult case,   $j=0$.   The case when $j=1$, $2$, or $3$ can then be proven in a similar fashion.
The proof for $j=0$ is divided into three parts, and we shall begin with the term
$\mathcal{R}_0$.
\subsubsection*{Estimates for the integral $\int_{\Omega}\mathcal{R}_0$}
 As derived in (\ref{eq:remainder}),  the term $\mathcal{R}_0$ can be written as
\begin{align}
\mathcal{R}_0 & :=
\underbrace{\mu\sum_{l=1}^5c_l\t^{l}A^k_i\t^{6-l}q_{,k}\t^6v^i}_{=:I_1}
+
\underbrace{(\mu q_{,k}A^s_iA^k_r),_s\t^6\Psi_{\kappa}^r\t^6v^i}_{=:I_2}
+
\underbrace{\mu\{\t^6,A^k_i\}q_{,k}\t^6v^i}_{=:I_3}
+
\underbrace{\mu A^s_i\t^6\Psi^rA^k_rq,_k\{\t^6,\g_s\}v^i}_{=:I_4}
 \nonumber\\
& \ \  \qquad\qquad
-
\underbrace{(\mu A^k_i),_k\t^6q\t^6v^i}_{=:I_5} 
-
\underbrace{\mu A^k_i\{\t^6,\g_k\}q\t^6v^i}_{=:I_6} 
 -\mu\sum_{l=1}^6\left(
\underbrace{c_l\t^lA^k_i\t^{6-l}v^i_{,k} \ \big(\t^6q+\t^6\Psi \cdot v\big)}_{=:I_7} \right. \nonumber \\
& \ \  \qquad\qquad \left.
+
\underbrace{d_l\t^{6-l}w\cdot\t^lv \  \big(\t^6q+\t^6\Psi \cdot v\big)}_{=:I_8}
-
\underbrace{\t^6\Psi \cdot v_t \ \big(\t^6q+\t^6\Psi \cdot v\big)}_{=:I_9}\right)
\tag{\ref{eq:remainder}}.
\end{align}

\vspace{.1 in}
\noindent
{\em Estimate of $\int_\Omega I_1 $.}
For the extremal case $l=5$,
\begin{align*} 
\Big|\int_{\Omega}\t^5A^k_i\t q,_{k}\t^6v^i\Big|
&  \le 
\|\t^5A^k_i\|_{L^4}\|\t q,_{k}\|_{L^4}\|\t^6v^i\|_0 \\
&
 \lesssim  \|\Psi - \operatorname{Id} \|_{6.5}\|\t q,_{k}\|_{0.5}\|\t^6v^i\|_0 \\
&\lesssim |h|_6\|q\|_{4}\|\t^6v^i\|_0 \\
&
 \lesssim \frac{\|q\|_4}{\chi(t)^{1/2}}\E^{1/2}\D^{1/2} \\
&
\leq\frac{C}{\delta}e^{-\gamma t}\epsilon\E+\delta\D,
\end{align*} 
where we have used H\"{o}lder's inequality and  the Sobolev embedding theorem, as well as  Young's inequality together with
Lemma \ref{lm:useful} for the last inequality.

If $l=4$, then Lemmas \ref{lm:basic} and \ref{lm:2} and  Corollary~\ref{co:eq} show that
\beas
\Big|\int_{\Omega}\t^4A^k_i\t^2 q,_{k}\t^6v^i\Big|
&\leq&\|\t^4A^k_i\|_{0}\|\t^2 q,_{k}\|_{L^{\infty}}\|\t^6v^i\|_0
\lesssim   |h|_{4.5}\|q\|_{4.5}\D^{1/2} \\
&\lesssim&\epsilon(\|q\|_{4.5}^2+\D)
\lesssim\epsilon(\epsilon e^{-\gamma t}\E+\D)
\lesssim \epsilon^2 e^{-\gamma t}\E+\epsilon\D.
\eeas
The case when $l=1,2$ or $3$ are estimated in the same way and yield the same bound.

\vspace{.1 in}
\noindent
{\em Estimates of $\int_\Omega I_k$ for $k=2,3,4,5$.}
The following estimate holds:
$$
\Big|\int_{\Omega} I_2+I_3+I_4+I_5\Big|
\lesssim\frac{C}{\delta}\epsilon e^{-\gamma t}\E + \delta \D.
$$
For  the integral of $I_2$,  an application of an $L^{\infty}$-$L^2$-$L^2$ H\"older's inequality together with
 Lemmas  \ref{lm:useful} and \ref{lm:basic}  leads to
\begin{align}
\Big|\int_{\Omega}I_2\Big|
&\lesssim\|(\mu q_{,k}A^s_iA^k_r),_s\|_{L^{\infty}}\|\t^6\Psi^r\|_{0}\|\t^6v^i\|_0 \nonumber\\
&\lesssim\|\mu A^s_iA^k_r\|_{W^{1,\infty}}\|q\|_3\frac{\E^{1/2}}{\chi^{1/2}}\D^{1/2}
\lesssim \frac{C}{\delta}e^{-\gamma t}\epsilon\E+\delta\D. \nonumber
\end{align}
The estimates for terms $I_3$, $I_4$, $I_5$, and $I_6$ are established in the same manner.
Note that the commutator  $\{\t^6,A^k_i\}q,_k$ in $I_3$ is defined in (\ref{eq:comm}b) and has at most five derivatives acting on $q,_k$; moreover,
 the expression $\{\t^6,\g_k\}f=\t^6\g_kf-\g_k\t^6f$ is of the form
$\sum_{1\leq|\alpha|\leq6}a_{\alpha}\g_{\alpha}f$, where the $a_{\alpha}$ are smooth uniformly bounded
functions on the set $\omega=\{x\in\Omega\big|\,\,\frac{1}{2}\leq|x|\leq1\}$.

\vspace{.1 in}
\noindent
{\em Estimate $\int_\Omega I_7$.}
We first consider the case that $l=6$, and write
$$
\int_{\Omega}\t^6A^k_iv,_k^i\big(\t^6q+\t^6\Psi\cdot v\big)
=\underbrace{\int_{\Omega}\t^6A^k_iv,_k^i\t^6q}_{J_1}
+\underbrace{\int_{\Omega}\t^6A^k_iv,_k^i\t^6\Psi\cdot v}_{J_2} \,.
$$

Thanks to Lemma \ref{technical}, we see that
$J_1  \leq \|\t^5A\|_{0.5} \| Dv \t^6q\|_{0.5}$.    By linear interpolation and the Sobolev embedding theorem,
$ \| Dv \t^6q\|_{0.5} \lessim \|v\|_3 \|q\|_6 + \|v\|_{2.5} \|q\|_{6.5} \lessim \|v\|_3\|q\|_{6.5}$.   It thus follows that
\begin{align*} 
J_1 &\lesssim | h |_6 \|_{5.5}\|v\|_3\|q\|_{6.5} 
\lesssim \frac{C}{\delta} |h|_6^2 \|v\|_{3}^2+\delta \|q\|_{6.5}^2
\lesssim
\frac{C\E E_{\beta}e^{-\beta t}}{\delta\chi(t)}
\lesssim\delta\D+\epsilon e^{-\gamma t}\E
+\delta(\epsilon e^{-\gamma t}\E+\D),
\end{align*} 
for some positive constant $\gamma>0$, where we have employed
Lemmas~\ref{lm:useful} and ~\ref{lm:2} with Corollary~\ref{co:eq}.

As for the integral of $J_2$, we again use Lemma~\ref{technical} to deduce that
\beas
&&\Big|\int_{\Omega}\t^6A^k_iv,_k^i\t^6\Psi\cdot v\Big|
\leq \|\t^5A^k_i\|_{0.5}\|v,_k^i\t^6\Psi\cdot v\|_{0.5}
\lesssim\|v\|_{2.5}^2\|\Psi-\operatorname{Id}\|_{6.5}^2
\lesssim e^{-\beta t}E_{\beta}\frac{\E}{\chi(t)}
\lesssim\epsilon e^{-\gamma t}\E,
\eeas
where $\gamma>0$ is given by Lemma~\ref{lm:useful}.
Now for the case that  $l=5$  in the integral of the term $I_7$, 
it follows that
\beas
&&\Big|\int_{\Omega}\t^5A^k_i\t v,_k^i(\t^6q+\t^6\Psi\cdot v)\Big|
\leq
\|\t^5A^k_i\|_{L^4}\|\t v,_k^i\|_{L^4}\|\t^6q+\t^6\Psi\cdot v\|_{0}\\
&&\lesssim \|\t^5A^k_i\|_{0.5}\|\t v,_k^i\|_{0.5}\|\t^6q+\t^6\Psi\cdot v\|_{0}
\lesssim \frac{\E^{1/2}}{\chi(t)^{1/2}}\|v\|_{2.5}\E^{1/2}\lesssim
\frac{E_{\beta}^{1/2}e^{-\beta t/2}}{\chi(t)^{1/2}}\E
\lesssim \sqrt{\epsilon}  e^{-\gamma t}\E,
\eeas
where we used Lemma~\ref{lm:useful} again and the fact that (by definition of $\E$), 
 $\|\t^6q+\t^6\Psi\cdot v\|_0^2\lesssim\E$.
Hereby we used the 
estimate~(\ref{eq:useful}). The remaining cases $l=1,2,3,4$ follow analogously and the estimates rely on a systematic use of 
Lemmas~\ref{lm:useful},~\ref{lm:basic},~\ref{lm:2}, and Corollary~\ref{co:eq}.

\vspace{.1 in}
\noindent
{\em Estimate of $\int_\Omega I_8$.}
For the case that  $l=1$ or $2$, we have that
\begin{align*}
\Big|\int_{\Omega}
\t^{6-l}w\cdot\t^lv\big(\t^6q+\t^6\Psi\big)\Big|
& \lesssim\|\t^{6-l}w\|_0\|\t^lv\|_{L^{\infty}}\|\t^6q+\t^6\Psi\cdot v\|_0
\lesssim\frac{\D^{1/2}}{\chi(t)^{1/2}}E_{\beta}^{1/2}e^{-\beta t/2}\E^{1/2}\\
& \lesssim\frac{\epsilon e^{-\gamma t}}{\delta}\E+\delta\D \,,
\end{align*}
while for the case that $l=3,4,5$ or $6$,
\[
\Big|\int_{\Omega}
\t^{6-l}w\cdot\t^lv\big(\t^6q+\t^6\Psi\cdot v\big)\Big|
\lesssim\|\t^{6-l}w\|_{L^{\infty}}\|\t^lv\|_0\|\t^6q+\t^6\Psi\cdot v\|_0
\lesssim\epsilon \D,
\]
where we used the Sobolev embedding $H^{1+\delta}\hookrightarrow L^{\infty}$ and 
Lemma~\ref{lm:basic}.

\vspace{.1 in}
\noindent
{\em Estimate of $\int_\Omega I_9$.}
We see that 
$$
\Big|\int_{\Omega}\t^6\Psi\cdot v_t(\t^6q+\t^6\Psi\cdot v)\Big|
\lesssim\|\t^6\Psi\|_0\|v_t\|_{L^{\infty}}\|(\t^6q+\t^6\Psi\cdot v)\|_0
\lesssim\frac{\E^{1/2}}{\chi(t)^{1/2}}E_{\beta}^{1/2}e^{-\beta t/2}
\E^{1/2}
\lesssim
\sqrt{\epsilon} e^{-\gamma t/2}\E \,,
$$
with the decay rate $\gamma >0$ given in  Lemma~\ref{lm:useful}.
\subsubsection*{Estimate of $\int_{\Omega}\tilde{\mathcal{R}}_1$}
In the same manner, we find that 
$\Big|\int_{\Omega}\tilde{\mathcal{R}}_1\Big|
\lesssim \epsilon e^{-\gamma t}\E+\delta\D$.

\subsubsection*{Estimate of the boundary integral $\int_{\Gamma}\mathcal{G}_0$}
We begin with the formula~(\ref{eq:remaindergamma}) (whereby we recall~(\ref{eq:ntilde}) $\tilde{n}=AN=\sqrt{R^2+R_{\theta}^2}n$).
\begin{align}
\mathcal{G}_0 & =
\underbrace{-\g_Nq\t^6\Psi\cdot \tilde{n}\t^6\Psi\cdot \tilde{n}_t}_{K_1}
-
\underbrace{\g_Nq\frac{d}{dt}\Big[R\t^6h\big(
-R_J+\sum_{a=0}^5c_a^J\t^ah\t^{6-a}\xi\cdot(h\xi-h_{\theta}T)\big)\Big]}_{K_2} \nonumber\\
& \ 
+
\underbrace{\g_Nq\frac{d}{dt}\Big[\big(-R_J+\sum_{a=0}^5c_a^J\t^ah\t^{6-a}\xi\cdot(h\xi-h_{\theta}T)\big)^2\Big]}_{K_3}
+
\underbrace{\sum_{l=1}^6a_l(-\g_Nq)\t^6\Psi\cdot \tilde{n}\t^{6-l}(v-w)\cdot\t^l\tilde{n}}_{K_4}. \tag{\ref{eq:remaindergamma}}
\end{align}
\noindent
{\em Estimate of $\int_\Omega K_1$.}
Note that
\beas
&&\Big|\int_{\Gamma}\g_Nq_t\t^6\Psi\cdot \tilde{n}\t^6\Psi\cdot \tilde{n}_t\Big|
\lesssim
|\g_Nq_t|_{L^{\infty}}|\tilde{n}_t|_{L^{\infty}}|\t^6\Psi|_0^2
\lesssim
|\g_Nq_t|_{1}|h_t|_2\frac{\E}{\chi(t)}
\lesssim E_{\beta}\frac{e^{-\beta t}}{\chi(t)}\E \lesssim \epsilon e^{-\gamma t}\E,
\eeas
where we used the trace theorem and Lemma~\ref{lm:useful}.

\vspace{.1 in }
\noindent 
{\em Estimates of $\int_\Omega K_2$ and $\int_\Omega K_3$.}
These two integrals are lower-order and thanks to
 Lemmas~\ref{lm:basic} and~\ref{lm:useful} are bounded by $ \epsilon e^{-\gamma t}\E+\delta\D$.
Note that $|J|=1+O(\epsilon)$ remains close to $1$ due to the a priori smallness bounds from Lemma~\ref{lm:basic}.

\vspace{.1 in }
\noindent
{\em Estimate of $\int_\Omega K_4$.} The estimate of $\int_\Omega K_4$ requires some explanation, as it
 has the largest derivative count in $\mathcal{G}_0$.  In Appendix \ref{se:derivation}, we derive the
identity
\begin{equation}\label{ssss200}
\t^6\Psi\cdot \tilde{n}
=R_J\t^6h  - R_J +\sum_{a=0}^5c_a^J\t^ah\t^{6-a}\xi\cdot(h\xi-h_{\theta}\tau),
\end{equation} 
where we recall that $\tau$ is the unit tangent defined by~(\ref{eq:normaltangent}) and $R_J=RJ^{-1}$.
Substitution of (\ref{ssss200})  in the integral
$\int_{\Gamma}(-\g_Nq)\t^6\Psi\cdot \tilde{n}\t^{6-l}(v-w)\cdot\t^l\tilde{n}$ then yields
\begin{align} 
&\Big|\int_{\Gamma}(-\g_Nq)\t^6\Psi\cdot \tilde{n}\t^{6-l}(v-w)\cdot\t^l\tilde{n}\Big| \lesssim   \Big|\int_{\Gamma}(-\g_Nq)\t^{6-l}(v-w)\cdot\t^l\tilde{n}\Big|
 \nonumber \\
& \quad \qquad \qquad
+\Big|\int_{\Gamma}(-\g_Nq)O(\t^5h)\cdot \tilde{n}\t^{6-l}(v-w)\cdot\t^l\tilde{n}\Big|
+\Big|\int_{\Gamma}(-\g_Nq)R\t^6h\t^{6-l}(v-w)\cdot\t^l\tilde{n}\Big|.
\label{eq:interm}
\end{align} 
The first and the second integrals on the right-hand side of~(\ref{eq:interm}) are easily estimated using H\"{o}lder's inequality and the
Sobolev embedding theorem, while the third integral on the right-hand side of~(\ref{eq:interm}) requires some care due to 
the presence of $\t^6h$.
If $l=1$ or $l=2$, then
\beas
&&\Big|\int_{\Gamma}(-\g_Nq)R\t^6h\t^{6-l}(v-w)\cdot\t^l \tilde{n}\Big|
\leq|\sqrt{-\g_Nq}\t^6h|_0|\sqrt{-\g_Nq}R|_{L^{\infty}}
(|\t^4v|_1+|\t^4h_t|_1)|\t^l \tilde{n}|_{L^{\infty}}\\
&&\lesssim\E^{1/2}\|q\|_2^{1/2}(\epsilon e^{-\gamma t/2}\E^{1/2}+\D^{1/2})\epsilon 
\lesssim\epsilon^2 P(\E,E_{\beta})e^{-\gamma t}\E+\delta\D,
\eeas
where we have used Corollary~\ref{co:eq}, Lemma~\ref{lm:basic}, and then Young's inequality for the last estimate.
The case that $l=3$, $4$, or $5$ follows similarly from Lemmas~\ref{lm:useful},~\ref{lm:basic}, and~\ref{lm:2}.
The case $l=6$  appears problematic because of the term $\t^6\tilde{n} \cdot \tau$ which, modulo coefficients, is essentially
$\t^7 h$, one derivative more than appears in $\E$.  The integral is, however, easily estimated thanks to the presence of
an exact derivative, formed from the integrand $\t^7h \ \t^6h$. 

We set $J_h = \sqrt{ R^2+ h_\theta^2}$ and write the unit tangent to $\Gamma(t)$ as 
$ \mathfrak{t}  =  J_h ^{-1} (R \tau + h_\theta N)$.   A simple computation shows that
$$
n_\theta = J_h ^{-2} (R^2 + 2h_\theta^2 + R h_{\theta\theta}) \mathfrak{t}   \,.
$$

  Since $v-w = \mathfrak{t}   \cdot (v-w) \, \mathfrak{t}  $ on $\Gamma$, we
see that $\t^6 n \cdot (v-w) = \mathfrak{t}  \cdot (v-w) \ \t^6n\cdot \mathfrak{t}  $. 
We then write
$$
\t^6\tilde{n}\cdot (v-w)  = g_1 \, \t^7h+g_2,
$$
where $g_1 = \mathfrak{t}\cdot (v-w) J_h ^{-2} R$, and where $g_2$  is a lower-order term in $v -w$ and has at most six tangential derivatives on $h$.
We then write
\begin{align*} 
\int_{\Gamma}(-\g_Nq)R \, \t^6h \, (v-w)\cdot\t^6\tilde{n} & = \int_{\Gamma}(-\g_Nq)R \, g_1\, \t^6h \,  \t^7 h +  \int_{\Gamma}(-\g_Nq)R \, \t^6h \,g_2  \\
& =- {\frac{1}{2}}  \int_{\Gamma} \bar \partial [(-\g_Nq)R \, g_1]\, |\t^6h|^2 +  \int_{\Gamma}(-\g_Nq)R \, \t^6h \,g_2
\end{align*} 
Arguing in a similar fashion as for the case that $l=1$ or $2$, we see that
\[
\Big|\int_{\Gamma}(-\g_Nq)R\t^6h(v-w)\cdot\t^6\tilde{n}\Big|
\lesssim \sqrt{\epsilon} e^{-\gamma t}\E.
\]

 \vspace{.1 in}
 \noindent
{\em Step 3. Estimates for $\mathcal{S}_j$, $\tilde{\mathcal{S}}_j$, $\mathcal{H}_j$ in~(\ref{eq:enidentity}).}
We next prove that
\be\label{eq:error3}
\Big|\int_0^t\int_{\Omega}(\mathcal{S}_j+\tilde{\mathcal{S}}_j)
+\int_0^t\int_{\Gamma}\mathcal{H}_j\Big|
\leq O(\sqrt{\epsilon} )\sup_{0\leq s\leq t}\E(s)+\delta\int_0^t\D(s),\quad j=1,2,3.
\ee

We will analyze the case that $j=1$, as the estimates for the case that $j=2$ or $3$ follow in the same manner.
We begin with the definition of $\mathcal{S}_1$ given in  (\ref{eq:sremainder}) as
\begin{align}
\mathcal{S}_1 & :=
\sum_{0<a+b<6\atop a\leq5,\,b\leq1}c_{ab}\mu\t^a\g_t^bA^k_i\t^{5-a}\g_t^{1-b}q_{,k}\t^5v^i
+S'_1 \nonumber\\
& \ \ 
-\sum_{l=1}^5d_l\mu\t^{5-l}\Psi_{t}\cdot\t^lv
\big(\t^5q_t+\t^5\Psi_t\cdot v\big)
-\sum_{l=1}^5c_l\mu\t^lA^k_i\t^{5-l}v^i_{,k}\big(\t^5q_t+\t^5\Psi_t\cdot v\big) \tag{\ref{eq:sremainder}},
\end{align}
where $S'_1$ is a lower-order term given by
\beas
S'_1&=&(\mu q_{,k}A^s_iA^k_r),_s\t^5\Psi^r_t\t^5v^i
+\{\t^5\g_t,A^k_i\}q_{,k}\t^5v^i
+\{\t^5A_{t,i}^k\}q_{,k}\t^5v^i+\mu A^s_i\t^5\Psi_t^rA^k_rq,_k\{\t^5,\g_s\}v^i\\
&&
-(\mu A^k_i),_k\t^5q_t\t^5v^i
+\mu A^k_i\t^5\g_t^jq\{\t^5,\g_k\}v^i
+\mu A^k_i\{\t^5,\g_k\}\g_tq\t^5v^i
\eeas
Most of the estimates are completely standard and we focus on the more 
problematic terms, characterized by the highest number of derivatives applied to two out of the three terms
in our cubic integrands. For illustration, in the first term on the right-hand side 
of~(\ref{eq:sremainder})
we analyze the cases $(b=0,a=1)$ and $(b=0,a=5)$. 
If $(b=0,a=1)$ then we first integrate-by-parts and  an
$L^{\infty}$-$L^2$-$L^2$ H\"{o}lder's  inequality to find that
\beas
\Big|\int_{\Omega}\t A^k_i\t^{4}\g_tq_{,k}\t^5v^i\Big|
&=&\Big|\int_{\Omega}\t^2A^k_i\t^{3}\g_tq_{,k}\t^5v^i
+\int_{\Omega}\t A^k_i\t^{3}\g_tq_{,k}\t^6v^i\Big|\\
&\leq&\|\t A^k_i\|_{W^{1,\infty}}\|\t^{3}\g_tq_{,k}\|_0
(\|\t^5v^i\|+\|\t^6v\|_0)\\
&\leq&|h|_{3.5}\|q_t\|_4(\|\t^5v^i\|+\|\t^6v\|_0)\lesssim
\epsilon \D,
\eeas
where Lemmas~\ref{lm:useful} and~\ref{lm:2} have been used.
If $(b=0,a=5)$ then 
\beas
\Big|\int_{\Omega}\t^5A^k_i\g_tq_{,k}\t^5v^i\Big|
\leq|h|_{6}\|q_t\|_{2}\|\t^5v\|_0\leq\frac{\E^{1/2}}{\chi(t)^{1/2}}E_{\beta}(t)^{1/2}e^{-\beta t/2}\D^{1/2}
\lesssim\frac{\epsilon e^{-\gamma t}}{\delta}\E+\delta\D,
\eeas
where we used Lemmas~\ref{lm:useful},~\ref{lm:basic}, and~\ref{lm:1}.
The remaining estimates in the expressions~(\ref{eq:sremainder}) 
and~(\ref{eq:s1tilde}) for
$\mathcal{S}_1$ and $\tilde{\mathcal{S}}_1$ 
follow in the identical manner.   As to the boundary integral of
 $\mathcal{H}_1$, we state  the formula for the integrand derived in (\ref{eq:sremaindergamma}) as
\be\tag{\ref{eq:sremaindergamma}}
\begin{array}{l}
\displaystyle
\mathcal{H}_1:=
2\g_Nq\t^5h_tR_J\sum_{a=0}^4\t^ah_t\t^{5-a}\xi\cdot \tilde{n}
+\sum_{l=1}^4a_l(-\g_Nq)\t^5\Psi\cdot \tilde{n}\t^5v\cdot \tilde{n}\t^{5-l}(v-w)\cdot\t^l\tilde{n}.
\end{array}
\ee
We consider the boundary integral of the  first term on the right-hand side.  We begin with 
the interpolation bound
\begin{equation}\label{ssss202}
|h_t|_4\leq|h_t|_3^{1/2}| h_t|_5^{1/2}
\lesssim\sqrt{\epsilon} \frac{\D^{1/4}}{\chi(t)^{1/4}},
\end{equation} 
where we have used (\ref{ssss201}) to bound $|\t^3h_t|$
and the definition of $\D$ given in (\ref{eq:ALEdissipation}).
If $a=4$ in the first term of the right-hand side of (\ref{eq:sremaindergamma}), then
\beas
\Big|\int_{\Gamma}\g_Nq\t^5h_tR_J\t^4h_t\t\xi\cdot \tilde{n}\Big|
&=&\frac{1}{2}\Big|\int_{\Gamma}\g_Nq\t(|\t^4h_t|^2)R_J\t\xi\cdot \tilde{n}\Big|\\
&\lesssim&\Big|\int_{\Gamma}\t(\g_Nq\t\xi\cdot \tilde{n})|\t^4h_t|^2\Big|
\lesssim\|q\|_4\epsilon\frac{\D^{1/2}}{\chi(t)^{1/2}}
\lesssim \epsilon^{3/2}e^{-\gamma t}\D^{1/2}\\
&\lesssim&\epsilon^{3/2}e^{-2\gamma t}+\epsilon^{3/2}\D,
\eeas
where we have once again used Lemma~\ref{lm:useful} in  second inequality,
the estimate (\ref{ssss202}), and Young's inequality.
If $a\in\{0,1,2,3\}$ then
\beas
\Big|\int_{\Gamma}\g_Nq\t^5h_tR_J\t^ah_t\t^{5-a}\xi\cdot \tilde{n}\Big|
&\lesssim&|\g_Nq|_{L^{\infty}}|\t^5h_t|_{0}|R_J|_{L^{\infty}}|\t^ah_t|_0
|\t^{5-a}\xi\cdot \tilde{n}|_{L^{\infty}}\\
&\lesssim&\|q\|_3\frac{\D^{1/2}}{\chi(t)^{1/2}}\epsilon 
\lesssim\epsilon^{3/2}e^{-\gamma t}\D^{1/2}
\lesssim\epsilon^{3/2}e^{-2\gamma t}+\epsilon^{3/2}\D,
\eeas
where we used Lemmas~\ref{lm:basic} and~\ref{lm:useful} and the same idea as above.
The estimates for the second term on the right-hand side 
of~(\ref{eq:sremaindergamma}) follow in an analogous vein, relying 
crucially on Lemmas~\ref{lm:basic} and~\ref{lm:useful}.
This finishes the proof of~(\ref{eq:error3}).

 \vspace{.1 in}
 \noindent
{\em Step 4.} 
The proof of the lemma is a direct consequence of the bounds (\ref{eq:error1}), (\ref{eq:error2}), and (\ref{eq:error3}).
\end{proof}
\section{Existence for all time $t\ge 0$ and nonlinear stability}\label{se:main}
\subsection{Structure of the proof}
The basic goal in our strategy for global-in-time existence and decay of the temperature function is to prove
that on any time-interval on which the bootstrap assumptions (\ref{eq:assumption1}) and (\ref{eq:assumption2})  are valid, we have that
$$
\sup_{0\leq s\leq t}\E(s)+\int_{0}^t\D(s)\,ds
\leq C_K\E(0),
$$
where $C_K>0$ is some explicit constant depending on $K$. 
Upon choosing the initial data $(q_0,h_0)$ sufficiently small, we can obtain an improvement of the first 
bootstrap bound in~(\ref{eq:assumption1}). In Section~\ref{se:Ebetabound} we show the improvement of the bootstrap assumption
on $E_{\beta}$ in~(\ref{eq:assumption1}) and in Corollary~\ref{co:oddson} we have already shown the improvement of
the bootstrap assumption~(\ref{eq:assumption2}).
By a continuity argument this leads to a global existence result.

In order to implement the above strategy, we start with the basic energy inequality given by~(\ref{eq:basicest}). 
Note however the presence of an exponentially growing term 
$CK^2\int_0^te^{\eta s}\E(s)\,ds$ on the right-hand side of~(\ref{eq:basicest}).
That term appears by treating the terms
$\int_{\Gamma}(-\g_Nq_t)R_J^2|\t^{6-2j}\g_t^jh|^2\,d\theta$, $j=0,1,2,3$ as error terms.
By applying a straightforward Gronwall-type argument,
this will be enough to guarantee that
solutions to the classical Stefan problem (\ref{eq:ALE})  exist on a {\it sufficiently
long} time-interval $[0,T_K]$, where the time $T_K$ may be larger than the time of existence guaranteed by
our local well-posedness theorem  in~\cite{HaSh}.
As we  explained in the introduction,  by a {\it sufficiently long}  time-interval, we mean a time $T_K$ after which
the dynamics of the Stefan problem~(\ref{eq:ALEheat}) are, in fact,  dominated by
the projection of the solution  onto the first eigenfunction $\varphi_1$ of the Dirichlet-Laplacian.

To prove global existence we need, however,  more refined estimates that will show that
the  $\int_{\Gamma}(-\g_Nq_t)R_J^2|\t^{6-2j}\g_t^jh|^2\,d\theta$ are in fact sign-definite for $t\geq T_K$, leading to the elimination of
the exponentially-in-time growing bounds.
First, in Section~\ref{se:positive} we prove strict positivity 
of the term $\g_Nq_t$ at time $T_K$.
Finally in Section~\ref{se:mainproof}, we use  a comparison principle  to show that $\g_Nq_t$
remains positive {\em after} time $T_K$. This allows us, in turn, to prove the 
uniform-in-time energy bound  and extend the solution  for all time $t\ge 0$.

\subsection{Boundedness of $E_{\beta}$}\label{se:Ebetabound}
The following lemma shows that under the bootstrap assumptions, the bound on $E_{\beta}+\int_0^tD(s)\,ds$ 
from~(\ref{eq:assumption2}) is improved. 
\begin{lemma}\label{lm:heat}
There exists a constant $\tilde{C}$ and $\epsilon>0$ sufficiently small, such that if the bootstrap 
assumptions~(\ref{eq:assumption1}) and~(\ref{eq:assumption2}) hold with such $\epsilon$ and $\tilde{C}$, then
\[
 E_{\beta}(t)+\int_0^tD(s)\,ds<\frac{\tilde{C}}{2} E_{\beta}(0).
 \]
 \end{lemma}
\begin{proof}
We set 
$$ x(t) = \| q(t)\|_4^2 + \| q_t(t)\|_2^2+  \| q_{tt}(t)\|_0^2 \ \ \text{ and recall  that} \ \ 
D(t) = \| q(t)\|_5^2 + \| q_t(t)\|_3^2+  \| q_{tt}(t)\|_1^2 \,.$$

\vspace{.1 in}
\noindent
{\em Step 1. Energy inequality for $q_{tt}$.}   From equation (\ref{heatss}c), we see that
\begin{align*}
{\frac{1}{2}} \frac{d}{dt} \| q_{tt}\|_0^2 + \| \nabla_\Psi q_{tt}\|_0^2 =   
\int_\Omega f_2  \, q_{tt},
\end{align*} 
 where the forcing term $f_2$ is defined just below equation (\ref{heatss}).
We next show that the right-hand side can be bounded by $ \epsilon D$.   
We first focus on the term $(\Psi_t \cdot v)_{tt}$ in the forcing function
$f_2$. Using the product rule we obtain
\begin{align*}
\int_\Omega  (\Psi_{t} \cdot v)_{tt} \, q_{tt} & =  
\underbrace{ \int_\Omega  \Psi_{ttt} \cdot v \, q_{tt}}_{\mathcal{A}_1} + \underbrace{ \int_\Omega  2\Psi_{tt} \cdot v_t \, q_{tt}}_{ \mathcal{A} _2}
+\underbrace{ \int_\Omega  \Psi_{t} \cdot v_{tt} \, q_{tt} }_{ \mathcal{A} _3} \,.
\end{align*} 
For the integral $\mathcal{A} _1$, we see that
\begin{align*}
\Big|\int_\Omega  \Psi_{ttt} \cdot v \, q_{tt}\Big|
& \le \|\Psi_{ttt}\|_0\|v\|_{L^{\infty}}\|q_{tt}\|_0
\lesssim |h_{ttt}|_{0.5}\|v\|_{2}\|q_{tt}\|_0
 \lesssim \epsilon D,
\end{align*}
where we used the bound~(\ref{eq:clever1}) to estimate $|h_{ttt}|_{0.5}$ by $\E^{1/2}$.
The estimate $|\mathcal{A}_2|\lesssim\epsilon D$ follows analogously to the estimate for term $\mathcal{A}_1$ and
the bound on $\mathcal{A}_3$ follows from
\[
\Big|\mathcal{A}_3\Big|
\lesssim \|\Psi_t\|_{L^{\infty}}\|v_{tt}\|_{0}\|q_{tt}\|_0
\lesssim\|\Psi_{t}\|_{2}
\lesssim \epsilon D,
\]
where we have used Lemma~\ref{lm:basic} to infer that $\|\Psi_{t}\|_{2}\lesssim \epsilon $.
All  of the remaining terms in the forcing function $f_2$ can be estimated by a straightforward application of the
Sobolev embedding theorem together with Lemma~\ref{lm:basic} (to guarantee the smallness of various Sobolev norms
applied to the coefficient matrix $(A^k_i)_{k,i=1,2}$).  Thus, in summary,
\be\label{eq:diffin}
\frac{1}{2}\frac{d}{dt}\|q_{tt}\|_0^2+\|\nabla_{\Psi}q_{tt}\|_0^2
\leq C \epsilon \, D.
\ee

\vspace{.1 in}
\noindent
{\em Step 2. Elliptic estimates.}
We next prove that the quantities $x$ and $y$ are respectively controlled
by $\|q_{tt}\|_0^2$ and $\|\nabla_{\Psi}q_{tt}\|_0^2$.
Using the elliptic regularity estimate of Lemma~\ref{sobolevA},
the elliptic equations~(\ref{heatss}), and Lemma~\ref{lm:basic},  it follows that
\be\label{eq:mh1}
\|q_{t}\|_{2}\lesssim\|q_{tt}\|_{0}+\|f_1\|_0 \,,
\ee
and
\begin{equation}\label{ss901}
 \|q\|_{4}\lesssim\|q_t\|_2+\|f_0\|_2 \lessim
 \|q_{tt}\|_{0}+\|f_1\|_0+\|f_0\|_2 \,.
\end{equation} 
A straightforward application of the Sobolev embedding theorem together with Lemma~\ref{lm:basic}
implies that
\begin{equation}\label{ss900}
\|f_1\|_0^2+\|f_0\|_2^2\lesssim\epsilon x(t).
\end{equation} 
Hence, with (\ref{eq:mh1})--(\ref{ss900}),
$$
x(t) \lessim  \|q_{tt}\|_{0}  + \epsilon x(t) \,,
$$
so that for $ \epsilon >0$ taken sufficiently small,
\[
x(t)\lesssim\|q_{tt}(t)\|_0^2 \,.
\]
Since $\ \|f_1\|_1^2+ \|f_0\|_3^2\lesssim\epsilon D(t)$, the same argument provides 
\be\label{eq:Dbound}
D\lesssim \|q_{tt}\|_1^2\,
\lesssim \|\nabla q_{tt}\|_0^2 \lesssim \|\nabla_{\Psi}q_{tt}\|_0^2 \,,
\ee
the last inequality following from the uniform lower-bound of  the  matrix $A\, A^T$.

\vspace{.1 in}

\noindent
{\em Step 3. Poincar\'e inequality.}
The following bound holds:
\be\label{poincare2}
(\lambda _1 - O( \epsilon )) \|f\|_0^2 \le \| \nabla _\Psi f \|_0^2 \,,
\ee
where $ \nabla_\Psi = A^T \nabla $ and $f\in H_0^1(\Omega)$.
To see~(\ref{poincare2}), note that the inequalities (\ref{A-1_est}) and (\ref{eq:apriori3der}) show that
$$
\| A- \operatorname{Id} \|_{ L^ \infty } \lessim \epsilon  \,,
$$
from which it follows that $A^k_i A^j_i \xi _k \xi _j \ge (1- O( \epsilon )) | \xi |^2$ for all $ \xi \in \mathbb{R}^2  $.  The
Poincar\'{e} inequality  $ \lambda _1 \| f\|_0^2 \le \| \nabla f \|_0^2$ for all $q \in H^1_0(\Omega) $ then concludes the proof.

\vspace{.1 in}
\noindent
{\em Step 4. The differential inequality and decay.} 
From~(\ref{eq:diffin}) and~(\ref{eq:Dbound}) we obtain that
\[
\frac{1}{2}\frac{d}{dt}\|q_{tt}\|_0^2+(1-O(\epsilon))\|\nabla_{\Psi}q_{tt}\|_0^2
\leq 0.
\]
Using the Poincar\'e inequality~(\ref{poincare2}),  it follows that
\[
\frac{d}{dt}\|q_{tt}\|_0^2+(2\lambda_1-O(\epsilon))\|q_{tt}\|_0^2\leq0.
\]
From this differential inequality,  we immediately infer the bound
\[
\|q_{tt}(t)\|_0^2\leq \|q_{tt}(0)\| e^{-(2\lambda_1-O(\epsilon))t}.
\]
From the elliptic estimate in Step 2, it finally follows that
\[
x(t)\leq C \|q_{tt}(0)\|_0^2e^{-(2\lambda_1-O(\epsilon))t} \leq C' E_{\beta}(0)e^{-(2\lambda_1-O(\epsilon))t}.
\]
Since $E_{\beta}(t)= x(t) e^{\beta t}$ and $\beta =2\lambda_1-\eta < 2\lambda_1 -O(\epsilon)$ for $\epsilon$ sufficiently small, it is now clear that we can choose $\tilde{C}$ so that
on the time interval of validity of bootstrap assumptions~(\ref{eq:assumption1}) and~(\ref{eq:assumption2}) we actually
have the improved bound
$
E_{\beta}(t)\leq \frac{\tilde{C}}{2}e^{-\beta t}.
$
\end{proof}
\subsection{Pointwise positivity of $\g_Nq_t$ at time $T_K=\bar{C}\ln K$}\label{se:positive}
\begin{lemma}\label{lm:positive}
Assume that the solution $(q,h)$ to the Stefan problem~(\ref{eq:ALE}) exists on a given time interval $[0,T]$.
Let the bootstrap assumptions~(\ref{eq:assumption1}) and~(\ref{eq:assumption2}) hold on that time interval
with $\epsilon>0$ sufficiently small, and assume the smallness assumption~(\ref{eq:smallassum}) for the initial data.
There exists a universal constant $\bar{C}$ such
that if $T\geq  T_K:=\bar{C}\ln K$,  then
\[
-q_t(T_K,x)>Cc_1e^{-\lambda_1T_K}\varphi_1(x),\quad x\in B_1(0),
\]
where $\varphi_1$ is the first eigenfunction of the Dirichlet-Laplacian on $\Omega$ and $c_1=\int_{\Omega}q_0\varphi_1\,dx$.
As a consequence,
\[
 \inf_{x\in\Gamma}\g_Nq_t(T_K,x)>0.
\]
\end{lemma}
\begin{proof}
{\em Step 1.  Hardy-type estimate.} As a consequence of the higher-order Hardy inequality (see Lemma 1 in  \cite{CoSh12}) and 
the Sobolev embedding theorem, 
for any $f\in H^{2.25}(B_1(0))\cap H^1_0(B_1(0))$,
\be\label{eq:veryuseful}
\sup_{x\in B_1(0)}\left|\frac{f(x)}{\varphi_1(x)}\right|
\leq C\|f\|_{2.25},
\ee
where $\varphi_1$ is the first eigenfunction of the Dirichlet-Laplacian on the unit ball.

\vspace{.1 in}
\noindent
{\em Step 2. The Duhamel formula.}
Let 
\begin{equation}\nonumber
q_0=\sum_{j=1}^{\infty}c_j\varphi_j
\end{equation} 
be the eigenvector decomposition of the initial datum $q_0$ with respect to the $L^2$ orthonormal basis
$\{\varphi_1,\varphi_2,\dots\}$ associated with the Dirichlet-Laplacian on the unit disk $B_1(0)$.
Writing the time-differentiated Stefan problem as a perturbation of the linear heat equation, we see that in $\Omega$,  $q_t$ satisfies
\begin{equation}\label{sss1}
q_{tt}-\Delta q_t=N(q,h),
\end{equation} 
where 
\begin{equation}\label{eq:inh}
N(q,h):=(a_{ij}-\delta_{ij})q_{t,ij}+b_iq_{t,i}+a_{ij,t}q,_{ij}+b_{i,t}q,_i+A_{,it}^kq,_kw^i+A^k_iq,_kw^i_t\,,
\end{equation} 
and the coefficients $a_{ij}$, $b_i$ are defined in (\ref{eq:coeff}).
Note that at time $t=0$,
$q_t(0)=\Delta q_0+\nabla q_0\cdot w_0$;
moreover, since $\Delta \varphi_j=-\lambda_j\varphi_j$ and $e^{t \Delta}$ is a linear semi-group , 
 the Duhamel principle implies that the solution $q_t$ to (\ref{sss1}) can be written as
\[
-q_t=\underbrace{c_1\lambda_1e^{-\lambda_1t}\varphi_1+\sum_{j=2}^{\infty}c_j\lambda_je^{-\lambda_jt}\varphi_j}_{=:X}
-\underbrace{e^{t\Delta}(\nabla q_0\cdot w_0)}_{=:Y}-\underbrace{\int_0^te^{(t-s)\Delta}N(q,h)}_{=:Z} \,.
\]

We first prove that $X(t) >0$ for times $t=\bar{C}\ln K$, where $\bar C$ denotes a universal constant.
We shall then show that at time $t=\bar{C}\ln K$, $|Y(t)|+|Z(t)|$ is bounded by a small fraction of $X(t)$.

\vspace{.1 in}
\noindent
{\em Step 3. Estimate of $X$.}
We begin by writing  $X$ as
\be\label{eq:X}
\begin{array}{l}
\displaystyle
X(t,x)= c_1\lambda_1e^{-\lambda_1 t}\varphi_1(x)
+ c_1\lambda_1e^{-\lambda_1 t}\varphi_1(x)
\Big(\sum_{j=2}^{\infty}\frac{c_j\lambda_j}{c_1\lambda_1}e^{(\lambda_1-\lambda_j)t}
\frac{\varphi_j(x)}{\varphi_1(x)}\Big).
\end{array}
\ee
Our goal is to prove that the term 
\be\label{eq:sigma}
\sigma:=\sum_{j=2}^{\infty}\frac{c_j\lambda_j}{c_1\lambda_1}e^{(\lambda_1-\lambda_j)t}
\frac{\varphi_j(x)}{\varphi_1(x)}
\ee
is small. By Corollary \ref{cor:hardy}, 
\begin{equation}\label{sss2}
\frac{|c_j|}{c_1}\leq K \text{ for all integers } j \ge 2\,.
\end{equation} 
Furthermore, using the normalization $\|\varphi_j\|_0=1$,  and the eigenvalue problem,
$\Delta\varphi_j=-\lambda_j\varphi_j$, elliptic regularity shows that $\| \varphi\|_2 \le \lambda_j$ and
that $\| \varphi\|_4 \le \lambda_j^2$; hence, linear interpolation provides us with the inequality
\begin{equation}\label{sss3}
\|\varphi_j\|_{2.25}\lesssim\lambda_j^{1.25} \,.
\end{equation}

Using (\ref{sss2}) and (\ref{sss3}),  together with the bound (\ref{eq:veryuseful}), we see that
$$
|\sigma|\leq CK\sum_{j=2}^{\infty}\lambda_j^{2.25}e^{(\lambda_1-\lambda_j)t} \,.
$$
Since $\lambda_1<\lambda_2\leq\lambda_3\leq\dots$, 
 there exists a constant $c^*$, uniform in $j \geq 2$, 
such that $\lambda_1/\lambda_j<(1-2c^*)$.  This implies that
$$(\lambda_1-\lambda_j)<-2c^*\lambda_j  \text{ for integers } j\ge 2 \,. $$
In particular, for $t\geq\bar{C}\ln K$
\[
CK\lambda_j^{2.25}e^{-c^*\lambda_jt}\leq CK\lambda_j^{2.25}K^{-\bar{C}c^*\lambda_j}
=C\frac{\lambda_j^{2.25}}{K^{\bar{C}c^*\lambda_j-1}}<\frac{1}{2}
\]
for $\bar{C}$ chosen sufficiently large, but independent of $K$.
(Recall  that $K>1$ since $K\geq\frac{\|q_0\|_1}{\|q_0\|_0}\geq1+\lambda_1>6$.)
Hence, from~(\ref{eq:sigma}) and the previous inequality it follows
\[
|\sigma|\leq\frac{1}{2}\sum_{j=2}^{\infty}e^{-c^*\lambda_jt}
\leq\frac{1}{2}\sum_{j=2}^{\infty}K^{-\bar{C}c^*\lambda_j}<\frac{1}{2}.
\]
Plugging this into~(\ref{eq:X}), we obtain for any $x\in B_1(0)$
\be\label{eq:X1}
X(t,x)\geq\frac{1}{2} c_1\lambda_1e^{-\lambda_1 t}\varphi_1(x)>0,\,\,\,\, t\geq\bar{C}\ln K.
\ee

\noindent
{\em Step 4. Estimates of  $Y$ and $Z$.}
The term $Y$ satisfies the estimate
\[
\|Y(t,x)\|_{L^{\infty}}\lesssim\|e^{t\Delta}(\nabla q_0\cdot w_0)\|_{2}
\lesssim e^{-\lambda_1 t}\|\nabla q_0\cdot w_0\|_{2}
\lesssim e^{-\lambda_1 t}\|q_0\|_3\|w_0\|_{2}
\lesssim \epsilon c_1 e^{-\lambda_1 t},
\]
where we used the Sobolev embedding theorem together with the bound $\|q_0\|_3\lesssim K c_1$,
which follows from $\|q_0\|_4/\|q_0\|_0\leq K$.
Thus $|Y(t,x)|<\frac{1}{4}|X(t,x)|$ with $\epsilon$ sufficiently small.   Next, to estimate $Z$ which vanishes at the boundary, 
we have that
\beas
\frac{|Z|}{\varphi_1(x)}&\leq& \int_{0}^t\big|\frac{e^{\Delta(t-s)}N(q,h)(s)}{\varphi_1(x)}\big|\,ds
\lesssim \int_{0}^t\|e^{\Delta(t-s)}N(q,h)(s)\|_{2.25}\,ds\\
&\lesssim& \int_{0}^t\|N(q,h)(s)\|_{2.25}\,ds \lesssim \sqrt{t}\big(\int_{0}^t\|N(q,h)(s)\|_{2.25}^2\,ds\big)^{1/2}.
\eeas
In the above chain of inequalities, we have used the bound~(\ref{eq:veryuseful}) for the second inequality, and the fact that 
$\|e^{t\Delta}\|_{H^s\to H^s}\leq1$. 

We shall conclude our estimate by showing that 
\be\label{eq:Nbound}
\int_{0}^t\|N(q,h)(s)\|_{2.25}^2\,ds\lesssim c_1^{11/5}\epsilon^{9/5}.
\ee
We recall that
\be\tag{\ref{eq:inh}}
N(q,h):=\underbrace{(a_{ij}-\delta_{ij})q_{t,ij}}_{=:Z_1}
+\underbrace{b_iq_{t,i}+a_{ij,t}q,_{ij}+b_{i,t}q,_i+A_{,it}^kq,_kw^i+A^k_iq,_kw^i_t}_{=:Z_2}\,,
\ee
and note that $Z_1$ is the highest-order term with respect to the number of derivatives applied to $q$.  Writing $Z_1=(a- \operatorname{Id} )D^2q_t$, where $\operatorname{Id} $ 
denotes the identity 
matrix, we see that
\[
\int_0^t\|Z_1\|_{2.25}^2\,ds \lesssim \|a-\operatorname{Id} \|_{2.25}^2\|D^2q_t\|_{2.25}^2
\lesssim \sup_{0\leq s\leq t}\|a-\operatorname{Id} \|_{2.25}^2\int_0^t\|q_t\|_{4.25}^2.
\]
From the sharp estimate~(\ref{re:height1}), we infer that $\sup_{0\leq s\leq t}\|a-\operatorname{Id} \|^2_{2.25}\lesssim  c_1^2$;
furthermore, for the term $\|q_t\|_{4.25}$ we apply the interpolation estimate
$\|q_t\|^2_{4.25}\lesssim\|q_t\|_2^{1/5}\|q_t\|_{4.5}^{9/5}\lesssim c_1^{1/5}e^{-\beta t/10}\|q_t\|_{4.5}^{9/5}$.

Using Lemma \ref{lm:2}, we then infer that
\[
\int_0^t\|Z_1\|_{2.25}^2\,ds\lesssim c_1^{11/5}\int_0^t e^{-\beta t/10}\|q_t\|_{4.5}^{9/5}\lesssim
c_1^{11/5}\epsilon^{9/5} \,,
\]
the last inequality following from H\"{o}lder's inequality and the fact that 
 $\int_0^t\|q_t\|_{4.5}^2\lesssim \epsilon^2$ by Lemma \ref{lm:2} and the bootstrap assumption~(\ref{eq:assumption1}).

Analogous estimates are applied to the term $Z_2$ to finally deduce~(\ref{eq:Nbound}).
By~(\ref{eq:Nbound}) and the above chain of estimates,  it follows that
\[
\frac{|Z|}{\varphi_1(x)}\lesssim \sqrt{t} c_1^{11/10}\epsilon^{9/10}.
\]
Hence, at time $T=\bar{C}\ln K$
\begin{align*}
|Z(T,x)| & \lesssim \bar{C}^{1/2}\ln K^{1/2} c_1 c_1^{1/10}\epsilon^{9/10}\varphi_1(x)\lesssim c_1\frac{\epsilon_0^{1/10}}{F(K)^{1/20}}
\bar{C}^{1/2}\ln K^{1/2}\epsilon^{9/10}\varphi_1(x)\\
& \le c_1\frac{\epsilon}{F(K)^{1/20}}\bar{C}^{1/2}\ln K^{1/2}\varphi_1(x)<\frac{1}{4} c_1\lambda_1e^{-\lambda_1 T}\varphi_1(x)
\leq \frac{1}{2}X(t,x).
\end{align*}
Note that we have used the estimate $c_1^{1/10}\lesssim \epsilon_0^{1/10}/F(K)^{1/20}$ (which follows from
$\|q_0\|\lesssim\E(0)^{1/2}$ and the smallness assumption~(\ref{eq:smallassum})) as well as $\epsilon_0\leq\epsilon$ which is going to hold
by our choice of $\epsilon_0$.
Observe that the very last inequality follows from~(\ref{eq:X1}). The next-to-last bound
is equivalent to 
\[
\frac{\epsilon}{F(K)^{1/20}}<\frac{\lambda_1}{4\bar{C}^{1/2}\ln K^{1/2}K^{\bar{C}\lambda_1}},
\]
which then follows from the choice~(\ref{eq:F}) of the function $F(K)$ in Remark~\ref{re:fk}. The second inequality above follows
from the estimate
$c_1\lesssim\|q_0\|\lesssim\E(0)^{1/2}\lesssim\epsilon/F(K)^{1/2}$.

\vspace{.1 in}
\noindent
{\em Step 5. Finishing the proof.}
From the above estimates on $X$, $Y$, and $Z$
it finally follows that for any $x\in B_1(0)$, $T=\bar{C}\ln K$,
\beas
-q_t(T,x)&\geq& |X(T,x)|-|Y(T,x)|-|Z(t,x)|
\geq X(T,x)-\frac{1}{2}X(T,x)-\frac{1}{4}X(T,x)=\frac{1}{4}X(t,x)\\
&\geq& 
c^*c_1e^{-\lambda_1 T}\varphi_1(x).
\eeas
Finally, since $\g_N\varphi_1\geq c$ for some $c>0$ uniformly over $\Gamma$ and since $\varphi_1>0$ in $\Omega$,
it follows $\inf_{x\in\Gamma}\g_Nq_t(T,x)>0$.
\end{proof}
\subsection{Proof of Theorem~\ref{th:main}}\label{se:mainproof}
{\em Step 1.}
By Proposition~\ref{lm:apriori1} with $\epsilon$ sufficiently small,  we conclude that
\be\label{eq:gron1}
\sup_{0\leq t\leq T}\E(t)+\int_0^T\D(t)\,dt\leq2\E(0)+CK^2\int_0^Te^{\eta t}\E(t)\,dt,\quad t\in[0,\mathcal{T}],
\ee
where $\mathcal{T}$ is the maximal interval of existence on which the
bootstrap assumptions~(\ref{eq:assumption1}) and~(\ref{eq:assumption2}) hold (with $\epsilon$ sufficiently small).
Our goal is to prove that on $[0,\mathcal{T}]$, the quantity $\E(t)$ is bounded from above by $2\E(0)e^{2CK^2t}$. 
We shall accomplish that by bounding $\E(t)$ from above by the function
$g(t):\R^+\to\R$, which is defined as the solution of the differential equation
\[
g'(t)=CK^2e^{\eta t}g(t),\quad g(0)=2\E(0).
\]
Solving this differential equation, we obtain
\beas
g(t)&=&2\E(0)e^{\frac{CK^2}{\eta}(e^{\eta t}-1)}=2\E(0)e^{(1+\sum_{k=2}^{\infty}\eta^{k-1}t^k/k!)CK^2t}\\
&=&2\E(0)e^{(1+O(\eta))CK^2t}\leq 2\E(0)e^{2CK^2t},
\eeas
where the convergence of the sum $\sum_{k=2}^{\infty}\eta^{k-1}t^k/k!$ is guaranteed for times $t\leq\frac{1}{\sqrt{\eta}}$.
Applying the integral Gronwall inequality to the difference $\E(t)-g(t)$, it follows from~(\ref{eq:gron1}) and the previous inequality that 
\[
\E(t)\leq g(t)\leq 2\E(0)e^{2CK^2t}
\]
for any $t\leq \mathcal{T}$. Our goal is to prove that $\mathcal{T}\geq\bar{C}\ln K$.
Using~(\ref{eq:gron1}) once again, we obtain the same smallness bound on $\int_0^t\D(s)\,ds$ to finally conclude that
\be\label{eq:intermed}
\sup_{0\leq s\leq t}\E(s)+\int_0^t\D(s)\,ds\leq 2\E(0)e^{2CK^2t}.
\ee
For $t\leq \bar{C}\ln K$, (\ref{eq:intermed}) and smallness assumption~(\ref{eq:smallassum}) on $\E(0)$ implies that 
\[
\sup_{0\leq s\leq t}\E(s)+\int_0^t\D(s)\,ds\leq \epsilon/2.
\]
Moreover, by Lemma~\ref{lm:heat} and since the bootstrap assumptions~(\ref{eq:assumption1}) 
and~(\ref{eq:assumption2}) are valid on $[0,\mathcal{T}]$ it follows that
\[
E_{\beta}(t)+\int_0^tD(s)\,ds<\frac{\tilde{C}}{2}E_{\beta}(0).
\]
Thus, by the continuity of $\E+E_{\beta}$ and the maximality of $\mathcal{T}$, we conclude 
$\min\{\mathcal{T},\bar{C}\ln K\}=\bar{C}\ln K=T_K$ since the bootstrap assumptions are still satisfied at time $t=\bar{C}\ln K$
(the argument is true as long as $\eta$ above is chosen in such a way that $\frac{1}{\sqrt{\eta}}>\bar{C}\ln K$).
By the local well-posedness theorem from~\cite{HaSh} and the continuity 
of $\E$ and $E_{\beta}$ in time, we actually have the strict
inequality $\mathcal{T}> T_K$
as we can extend the solution locally in time. 
We will argue by contradiction that $\mathcal{T}=\infty$. Assume
$\mathcal{T}<\infty$. 

\vspace{.1 in}
\noindent
{\em Step 2. Preserving the positivity of $\g_Nq_t$.}
We next show that $\g_Nq_t>0$ on the time interval $[T_K,\mathcal{T}[$.
This will be done with help of Lemma~\ref{lm:positive} and the maximum principle.
We start by constructing a suitable comparison function,
\be\label{eq:barrier}
P(t,r)=\kappa_1 e^{-\frac{3}{2}\lambda_1t}(\varphi_1(r)-\kappa_2(1-r^2)),
\ee
with positive constants $\kappa_1,\kappa_2$ to be specified later.
A straightforward calculation shows that
\begin{align}
(\g_t-a_{ij}\g_{ij}-b_i\g_i)P
 & =\kappa_1e^{-\frac{3}{2}\lambda_1t}\big[-\frac{1}{2}\lambda_1\varphi_1-2\kappa_2\text{Tr}(a) \nonumber \\
& \ \ +\frac{3}{2}\lambda_1\kappa_2(1-r^2)-(a_{ij}-\delta_{ij})\varphi_1-b\cdot(\nabla\varphi_1+2\kappa_2x)\big]. \label{eq:comp}
\end{align}
Observe that both $\varphi_1$ and $(1-r^2)$ vanish for $r=1$, the trace of the matrix $a$ is very close to $2$, i.e.,
$a_{11}+a_{22}=2+O(\epsilon )$ and the coefficients $b_i$ are very small, i.e. $|b|=O(\epsilon )$. Note that
the first and the second term in the parenthesis on the right-hand side of~(\ref{eq:comp}) are negative, while the fourth and the fifth term are 
small of order $\epsilon$.
If $r=|x|$ is close to $1$, then the second term dominates the third term and if $r$ is away from the boundary $r=1$, then one can choose $\kappa_2>0$
so that the first term dominates the third term. 
It follows easily that there exists a $\kappa_2>0$ and some constant $C_1$ such that
\be\label{eq:neg}
(\g_t-a_{ij}\g_{ij}-b_i\g_i)P<-C_1\kappa_1e^{-\frac{3}{2}\lambda_1t}.
\ee
It then follows from~(\ref{eq:neg}) and~(\ref{eq:coeff}) that
\be\label{eq:neg1}
(\g_t-a_{ij}\g_{ij}-b_i\g_i)(-q_t-P)>-(\g_ta_{ij}\,q,_{ij}+\g_tb_{i}\,q_i+\g_tA_{,i}^k\,q,_kw^i+A^k_iq,_kw^i_t)+C_1\kappa_1e^{-\frac{3}{2}\lambda_1t}.
\ee
Note, however, that the term in parenthesis on the right-hand side above
is a quadratic non-linearity and as such decays at least
as fast as $e^{-2\beta t}$:
\begin{align}
& \|\g_ta_{ij}\,q,_{ij}+\g_tb_{i}\,q_i+\g_tA_{,i}^k\,q,_kw^i+A^k_iq,_kw^i_t\|_{L^{\infty}} \nonumber \\
& \leq\|\g_ta_{ij}\,q,_{ij}\|_{1+\delta}+\|\g_tb_{i}\,q_i\|_{1+\delta}+\|\g_tA_{,i}^k\,q,_kw^i\|_{1+\delta}+\|A^k_iq,_kw^i_t\|_{1+\delta} \nonumber \\
& \leq C_2E_{\beta}(0)^{1/2}\epsilon e^{-2\beta t}\leq C_2c_1\epsilon e^{-2\beta t}. \label{eq:track}
\end{align}
Now, using~(\ref{eq:neg1}) and the above bound, we note that by choosing the constant $\kappa_1:=\frac{C_2}{C_1}c_1\epsilon $, 
we have that
\[
(\g_t-a_{ij}\g_{ij}-b_i\g_i)(-q_t-P)>C_2c_1\epsilon e^{-\frac{3}{2}\lambda_1t}-C_2c_1\epsilon e^{-2\beta t}>0,
\]
since $2\beta=2\lambda_1-\eta>\frac{3}{2}\lambda_1$.
The previous bound implies that $-q_t-P$ is a supersolution for the operator $\g_t-a_{ij}\g_{ij}-b_i\g_i$. Moreover,
by the construction of $P$, we have $-q_t-P=0$ on $\Gamma=\partial B_1(0)$.
Furthermore, at time $T_K=\bar{C}\ln K$, we have by Lemma~\ref{lm:positive} and~(\ref{eq:barrier}), that
\[
(-q_t-P)|_{T=\bar{C}\ln K}>Cc_1e^{-\lambda_1T}\varphi_1(x)-Cc_1\epsilon e^{-\frac{3}{2}\lambda_1T}\varphi_1(x)
+Cc_1\epsilon \kappa_2e^{-\frac{3}{2}\lambda_1T}(1-r^2)>0
\]
for $\epsilon$ sufficiently small. Thus, by Lemma~\ref{lm:oddson}, there exists a constant $m>0$ such that
\[
-q_t-P\geq m(1-r)e^{-(\lambda_1-O(\epsilon ))t},\,\,\,\,t>T_K,
\]
or in other words
\beas
-q_t&\geq& m(1-r)e^{-(\lambda_1-O(\epsilon ))t}+Cc_1\epsilon (1-r)e^{-\frac{3}{2}\lambda_1t}(\frac{\varphi_1(r)}{1-r}-\kappa_2(1+r))\\
&=& (1-r)e^{-(\lambda_1-O(\epsilon ))t}
\Big(m+Cc_1\epsilon e^{(-\frac{1}{2}\lambda_1t-O(\epsilon ))t}(\frac{\varphi_1(r)}{1-r}-\kappa_2(1+r))\Big),
\eeas
which readily gives the positivity of $\g_Nq_t$ on the time-interval $[T_K,\mathcal{T}[$ since $\frac{\varphi_1(r)}{1-r}-\kappa_2(1+r)>0$
by our choice of $\kappa_2$ above.
We conclude that the positivity of $-q_t$ at time $T_K=\bar{C}\ln K$ is a property preserved by our bootstrap regime and
moreover we get a quantitative lower bound on $\g_Nq_t$ on the time interval $[T_K,\mathcal{T}[$.

\vspace{.1 in}
\noindent
{\em Step 3. Conclusion.}
Thus for any $t\in[T_K,\mathcal{T}[$, the energy identity takes the form
\beas
&&\E(t)+\int_{T_K}^t\D(t)
+\frac{1}{2}\sum_{j=0}^3\int_{T_K}
^t\int_{\Gamma}\g_Nq_tR_J^2|\t^{6-2j}\g_t^jh|^2\,dx\\
&&
=\E(T_K)+\sum_{i=1}^4\int_{T_K}^t\int_{\Omega}\{\mathcal{R}_i+\mathcal{S}_i\}
+\sum_{i=0}^4\int_{T_K}
^t\int_{\Omega}\{\tilde{\mathcal{R}}_i+\tilde{\mathcal{S}}_i\}
+\sum_{i=0}^4\int_{T_K}
^t\int_{\Gamma}\{\mathcal{G}_i+\mathcal{H}_i\},
\eeas
where we formally define $\mathcal{S}_4=\tilde{\mathcal{S}}_4=\mathcal{G}_4=0$.
In particular, by the energy estimates stated in~(\ref{eq:error2}) and~(\ref{eq:error3}) the 
right-hand side of the above identity is bounded by
\[
\E(T_K)+O(\sqrt{\epsilon} )\sup_{{T_K}\leq s\leq t}\E(s)+
(O(\epsilon )+\delta)\int_{T_K}
^t\D(s)\,ds.
\]
Note here the absence of the exponentially growing term in the above bound 
as opposed to their presence in Proposition~\ref{lm:apriori1}. This is due to the fact that
terms $\int_{T_K}
^t\int_{\Gamma}\g_Nq_tR_J^2|\t^{6-2j}\g_t^jh|^2\,dx$, $j=0,1,2,3$ are positive and no longer treated as error terms.
By absorbing the small multiples of $\sup_{T_K\leq s\leq t}\E(s)$ and $\int_{T_K}
^t\D(s)\,ds$ into the left-hand
side and using the positivity of $\g_Nq_t$ from Step 2, we obtain
\be\label{eq:smallness2}
\sup_{T_K\leq s\leq t}\E(s)+\int_{T_K}
^t\D(s)\,ds
\leq 2\E(T_K)\leq 8\E(0)e^{2CK^2T_K}
\ee
by~(\ref{eq:intermed}).
Finally, we choose $\epsilon_0$ in the statement of Theorem~\ref{th:main} so that $\epsilon_0^2<\epsilon^2/2$.
Bound~(\ref{eq:smallness2}) and the condition $\E(0)\lesssim\epsilon_0^2/F(K)$ (with $F(K)$ given as in~(\ref{eq:F}))
imply
\[
\sup_{T_K\leq s\leq t}\E(s)+\int_{T_K}
^t\D(s)\,ds
\leq \frac{\epsilon^2}{2}.
\]
Together with Lemma~\ref{lm:heat} and Corollary~\ref{co:oddson}, we infer that the bootstrap assumptions~(\ref{eq:assumption1}) and~(\ref{eq:assumption2}) are improved.
Since $\E$ is continuous in time, we can
extend the solution by the local well-posedness theory to an interval $[0,\mathcal{T}+T^*]$
for some small positive time $T^*$. This however contradicts the maximality of $\mathcal{T}$ and
hence $\mathcal{T}=\infty$.
This concludes the proof of the main theorem.
\section{The $d$-dimensional case on general near-spherical domains}\label{se:general}
In this section we briefly sketch the set-up of the problem in general dimensions and explain how to adapt the arguments
from the 2-D case to the 3-D case.
Let $\Omega(t)\subset\R^d$ be an open simply connected subset of $\R^d$, $d\geq2$.
The moving boundary $\Gamma(t)=\partial\Omega(t)$ is parametrized as a graph over the unit sphere $\K^{d-1}$
\[
\Gamma(t)=\{x\,|\,\,\,x=R(t,\xi)\xi=(1+h(t,\xi))\xi,\,\,\,\,\xi\in\K^{d-1}\}.
\]
Initially $R_0(\xi)$ is assumed to be close to $1$, i.e.
$
R_0(\xi)-1=h_0(\xi)=O(\epsilon).
$
We shall assume that $\Omega_0$ is diffeomorphic to $B_1(0)$, where $\Phi:\Omega\to\Omega_0$ is the diffeomorphism mapping of
the unit ball onto the initial domain. Moreover, let $\tilde{\Psi}$ denote the family of diffeomorphisms  from the initial domain $\Omega_0$ to the moving domain $\Omega(t)$,
satisfying the harmonic equation $\Delta\tilde{\Psi}=0$ and the boundary condition $\tilde{\Psi}(\Gamma_0)=\Gamma(t)$.
We shall pull back the Stefan problem onto the unit ball $B_1(0)$ via the map $\Psi:B_1(0)\to\Omega(t)$
given as a composition of $\tilde{\Psi}$ and $\Phi$:
\[
\Psi=\tilde{\Psi}\circ\Phi.
\]
Upon defining $q$, $v$, $w$, and $A$ just as in Section~(\ref{eq:ALE}), 
the Stefan problem~(\ref{eq:stefan}) takes exactly the same form as~(\ref{eq:ALE}).
Abusing the notation, the normal velocity $\mathcal{V}(\Gamma(t))$ is now given
by
\[
\mathcal{V}(\Gamma(t))=\frac{R_tR}{\sqrt{R^2+|\nabla R|_{\K^{d-1}}^2}}.
\]
Here $|\nabla R|_{\K^{d-1}}^2$ stands for the squared norm of the Riemannian gradient of $R(t,\cdot)$ on the unit sphere $\K^{d-1}$, 
which is a coordinate invariant expression. The gauge equation for $\Psi$ transforms
into
\[
\Delta_{\Phi^{-1}}\Psi=0,\qquad \Psi(t,\K^{d-1})=\Gamma(t)
\]
due to the assumption $\Delta\tilde{\Psi}=0$ and the definition of $\Psi$.
This easily implies the optimal trace bound $\|\Psi\|_{H^s(\Omega)}\lesssim|\Psi|_{H^{s-0.5}}(\Omega)$ due to the the smoothness of
$\Phi$ and the {\em closeness} assumption $\|D\Phi-\text{Id}\|_{H^s}\lesssim \epsilon$, with $s$ sufficiently large.  When $d=3$,
the Sobolev embedding theorem  requires us   to raise the degree of spatial regularity in the definition our energy spaces by one derivative.

The second key observation is that the lower bound for the quantity $\chi(t)$ is obtained in the same 
way as in the case that $d=2$,  from Lemma~\ref{lm:oddson}.
We  $\nu_1$ denote the first eigenvalue of the operator $-\Delta_{\Phi^{-1}}$, 
which is the pull-back of the negative Laplacian from the initial domain $\Omega_0$ to $B_1(0)$.
By Lemma~\ref{lm:oddson} we obtain that
\[
\chi(t)\gtrsim c_1e^{-\lambda t},
\]
where $|\lambda-\nu_1|\leq O(|h-h_0|_{W^{2,\infty}}+|h_t|_{L^{\infty}})=O(\epsilon)$. 
Since $\|D\Phi-\text{Id}\|_{H^s}\lesssim\epsilon$ for $s$ sufficiently large, we have 
$|\nu_1-\lambda_1|\lesssim\epsilon$, where we recall that $\lambda_1$ stands for the
first eigenvalue of the Dirichlet-Laplacian. Together, the two previous estimates imply the
analogous conclusion of Corollary~\ref{co:oddson}, namely
\[
\chi(t)\gtrsim c_1e^{-\lambda_1 -\tilde{\lambda}_1(t) t},\quad \tilde{\lambda}_1=O(\epsilon).
\]
Let $\t^i$ denote the tangential component of $\g^i$ restricted to $\K^2$. To each multi-index $\vec{\alpha}=(\alpha_1,\alpha_2,\alpha_3)$ we associate
the tangential operator $\t^{\vec{\alpha}}=\t^{\alpha_1}\t^{\alpha_2}\t^{\alpha_3}$.
With $d=3$, we define
\begin{align}
&\E_{3D}(t)=\E_{3D}(q,h)(t):=  \nonumber \\
& \frac{1}{2}\sum_{|\alpha|+2b\leq6}\|\mu^{1/2}\t^{\vec{\alpha}}\partial_t^bv\|_{L^2_x}^2
+\frac{1}{2}\sum_{|\vec{\alpha}|+2b\leq7}|(-\g_Nq)^{1/2}R_J\t^{\vec{\alpha}}\g_t^bh|_{L^2_x}^2
+\frac{1}{2}\sum_{|\vec{\alpha}|+2b\leq7}\|\mu^{1/2}(\t^{\vec{\alpha}}\partial_t^bq+\t^{\vec{\alpha}}\partial_t^b\Psi\cdot v)\|_{L^2_x}^2 \nonumber \\
& \sum_{|\vec{\alpha}|+2b\leq6}\|(1-\mu)^{1/2}\g_{\vec{\alpha}}\partial_t^bv\|_{L^2_x}^2
+\frac{1}{2}\sum_{|\vec{\alpha}|+2b\leq7}\|(1-\mu)^{1/2}(\g_{\vec{\alpha}}\partial_t^bq+\g_{\vec{\alpha}}\partial_t^b\Psi\cdot v)\|_{L^2_x}^2 \nonumber
\end{align}
and
\begin{align}
& \D_{3D}(t)=\D_{3D}(q,h)(t):=  \nonumber \\
& \sum_{|\vec{\alpha}|+2b\leq7}\|\mu^{1/2}\t^{\vec{\alpha}}\partial_t^bv\|_{L^2_x}^2
+\sum_{|\vec{\alpha}|+2b\leq 6}|(-\g_Nq)^{1/2}R_J\t^{\vec{\alpha}}\g_t^bh_t|_{L^2_x}^2
+\sum_{|\vec{\alpha}|+2b\leq 6}\|\mu^{1/2}(\t^{\vec{\alpha}}\partial_t^bq_t+\t^{\vec{\alpha}}\partial_t^b\Psi_t\cdot v)\|_{L^2_x}^2 \nonumber \\ 
& +\sum_{|\vec{\alpha}|+2b\leq7}\|(1-\mu)^{1/2}\g_{\vec{\alpha}}\partial_t^bv\|_{L^2_x}^2
+\sum_{|\vec{\alpha}|+2b\leq6}\|(1-\mu)^{1/2}(\g_{\vec{\alpha}}\partial_t^bq_t+
\g_{\vec{\alpha}}\partial_t^b\Psi_t\cdot v)\|_{L^2_x}^2\,.  \nonumber
\end{align}
The lemmas of Section~\ref{se:basic} carry through analogously, as do the energy estimates of Section~\ref{se:long}.
By the continuity argument of Section~\ref{se:main}, we arrive at the 3-D version of our main theorem:
\begin{theorem}[The 3-D case]
Let $(q_0,h_0)$ satisfy the 
Taylor sign condition~(\ref{eq:wlog}), the strict positivity 
assumption~(\ref{eq:positive}), and the corresponding compatibility conditions. 
Let $\|q_0\|_4/\|q_0\|_0\leq K$.
Then there exists an $\epsilon_0=\epsilon_0(K)>0$ and $\delta_0>0$ such that if
$\E(q_0,h_0)<\epsilon_0^2$,
then there exists a unique global solution to problem~(\ref{eq:ALE}), satisfying
\[
\sup_{0\leq t\leq\infty}\E_{3D}(q(t),h(t))<C\delta_0^2,
\]
for some universal constant $C>0$, and
$
\|q\|_{H^5(B_1(0))}^2\leq C e^{-\beta t},
$
where $\beta=2\lambda_1- C \epsilon _0$ and $\lambda_1$ is the smallest eigenvalue of the 
Dirichlet-Laplacian on the unit ball $B(0,1) \subset \mathbb{R}^3  $.
The moving boundary $\Gamma(t)$ settles asymptotically to some nearby steady surface 
$\bar{\Gamma}$ and we have uniform-in-time estimate
\[
\sup_{0\leq t<\infty}|h-h_0|_{5.5} \lesssim \sqrt{ \delta _0}
\]
\end{theorem}

%
%

        

\appendix
 %
 \numberwithin{equation}{section}
\section{Proof of Proposition~\ref{lm:lemma1}}\label{se:derivation}
To prove the energy identity of Proposition~\ref{lm:lemma1}, we start by applying 
the differential operator of the form $\t^{6-2j}\g_t^j$ to the equation~(\ref{eq:ALEv}).
For $j=0,1,2,3$ we multiply it then by $\t^{6-2j}\g_t^j$ and integrate-by-parts.
Additionally, if $j=1,2,3$ we apply the operator $\t^{7-2j}\g_t^j$ to~(\ref{eq:ALEv}), 
multiply by $\t^{7-2j}\g_t^{j-1}v^i$, and
again integrate-by-parts.

\noindent
Based on these two cases we distinguish between the two different types of identities.  
\subsection{Identities of the first type}\label{se:Aone}
Recall that
$\mu:\bar{\Omega}\to\R$ is a $C^{\infty}$ cut-off function with the
property
\[
\mu(x)\equiv0\,\,\,\,\text{ if }\,\,|x|\leq1/2;\qquad\mu(x)\equiv1\,\,\,\,\text{ if }\,\,3/4\leq|x|\leq1.
\]
Applying the tangential differential operator $\mu\t^6$ 
to the equation~(\ref{eq:ALEv}), multiplying it by $\t^6v^i$ and
integrating over $\Omega$, we obtain
\[
\big(\mu\t^6v^i+\mu\t^6A^k_iq_{,k}+\mu A^k_i\t^6q_{,k},\,\t^6v^i\big)_{L^2}=
\sum_{l=1}^5c_l\big(\mu\t^lA^k_i\t^{6-l}q_{,k},\,\t^6v^i\big)_{L^2},
\]
where $c_{l}={6\choose l}$.
Recalling~(\ref{eq:comm}), we write
\[
\t^6A^k_i=-A^s_i\t^6\Psi^r_{,s}A^k_r+\{\t^6,A^k_i\},
\]
where $\{\t^6,A^k_i\}$ denotes the lower-order commutator defined in~(\ref{eq:comm}b).
With this identity, we obtain
\begin{align}
 \big(\mu\t^6A^k_iq_{,k},\t^6v^i\big)_{L^2(\Omega)}
& =-\big(\mu A^s_i\t^6\Psi^r_{,s}A^k_rq_{,k},\t^6v^i\big)_{L^2(\Omega)}
+\big(\mu\{\t^6,A^k_i\}q_{,k},\t^6v^i\big)_{L^2(\Omega)} \nonumber \\
& =-\int_{\Gamma}q_{,k}A^s_i\t^6\Psi^rA^k_r\t^6v^iN^s
+\int_{\Omega}\mu A^s_i\t^6\Psi^rA^k_rq,_k\t^6v^i_{,s}
+\int_{\Omega}\mathcal{T}_1 \nonumber\\
& =-\int_{\Gamma}q_{,k}A^s_i\t^6\Psi^rA^k_r\t^6v^iN^s
-\int_{\Omega}\mu A^s_i\t^6\Psi^rv^r\t^6v^i_{,s}
+\int_{\Omega}\mathcal{T}_1, \label{eq:t1}
\end{align}
where we have integrated-by-parts with respect to $x^s$ for the second equality, and have
used the identity $v^r=-A^k_rq_{,k}$  for the last equality;
the error term $\mathcal{T}_1$ is given by
\[
\mathcal{T}_1=
(\mu q_{,k}A^s_iA^k_r),_s\t^6\Psi_{\kappa}^r\t^6v^i
+\mu\{\t^6,A^k_i\}q_{,k},\t^6v^i+\mu A^s_i\t^6\Psi^rA^k_rq,_k\{\t^6,\g_s\}v^i.
\]
Furthermore, integration-by-parts with respect to $x^k$ yields 
\be\label{eq:est2}
\begin{array}{l}
 \displaystyle
\big(\mu A^k_i\t^6q_{,k},\,\t^6v^i\big)_{L^2}
=\int_{\Omega}\mu A^k_i\g_k\t^6q\t^6v^i+\int_{\Omega}\mu A^k_i\{\t^6,\g_k\}q\t^6v^i\\
\displaystyle
=-\int_{\Omega}\mu A^k_i\t^6q\t^6v^i_{,k}-\int_{\Omega}(\mu A^k_i),_k\t^6q\t^6v^i+\int_{\Omega}\mu A^k_i\{\t^6,\g_k\}q\t^6v^i,
\end{array}
\ee
where we have used $\t^6q=0$ on $\Gamma$,  and where $\{\t^6,\g_k\}$ denotes the lower-order commutator.
Summing~(\ref{eq:t1}) and~(\ref{eq:est2}), we find that
\be\label{eq:new1}
\begin{array}{l}
\displaystyle
\big(\mu\t^6A^k_iq_{,k}+\mu A^k_i\t^6q_{,k},\t^6v^i\big)_{L^2(\Omega)}
=-\int_{\Gamma}q_{,k}A^s_i\t^6\Psi^rA^k_r\t^6v^iN^s\\
\displaystyle 
-\int_{\Omega}\mu A^k_i\t^6v^i_{,k}\big(\t^6q+\t^6\Psi\cdot v\big)
+\int_{\Omega}(\mathcal{T}_1-(\mu A^k_i),_k\t^6q\t^6v^i+\mu A^k_i\{\t^6,\g_k\}q\t^6v^i).
\end{array}
\ee
The first two terms on the right-hand side of~(\ref{eq:new1}) will
be the source of positive definite quadratic contributions to the energy. 
To extract the quadratic coercive contribution from the first integral on the right-hand side 
of~(\ref{eq:new1}), note that $q,_k=N^k\g_Nq$ on $\Gamma$, and also
recall from (\ref{eq:ntilde}) the normal vector
$\tilde{n}=A^TN$.
Thus
\be\label{eq:aid1}
-\int_{\Gamma}q_{,k}A^s_i\t^6\Psi^rA^k_r\t^6v^iN^s
=\int_{\Gamma}(-\g_Nq)\t^6\Psi^r\tilde{n}^r\t^6v^i\tilde{n}^i
=\int_{\Gamma}(-\g_Nq)\t^6\Psi\cdot \tilde{n}\t^6v\cdot \tilde{n}.
\ee
Using  the boundary condition (\ref{eq:ALEneumann}), we reexpress
 $\t^6v\cdot \tilde{n}$ as
\beas
\t^6v\cdot \tilde{n}&=&\t^6w\cdot \tilde{n}+\t^6(v-w)\cdot \tilde{n}\\
&=&\t^6w\cdot \tilde{n}+\t^6\big(\underbrace{(v-w)\cdot \tilde{n}}_{=0}\big)-\sum_{l=1}^6a_l\t^{6-l}(v-w)\cdot\t^l\tilde{n}\\
&=&\t^6\Psi_t\cdot \tilde{n}-\sum_{l=1}^6a_l\t^{6-l}(v-w)\cdot\t^l\tilde{n}.
\eeas
Due to the above identity and~(\ref{eq:aid1}), we obtain that
\begin{align}
& \int_{\Gamma}(-\g_Nq)\t^6\Psi\cdot \tilde{n}\t^6v\cdot \tilde{n}
= \int_{\Gamma}(-\g_Nq)\t^6\Psi\cdot \tilde{n} \ \t^6\Psi_t\cdot \tilde{n}
-\sum_{l=1}^6a_l\int_{\Gamma}(-\g_Nq)\t^6\Psi\cdot \tilde{n}\t^{6-l}(v-w)\cdot\t^l\tilde{n} \nonumber \\
& \ \ =\int_{\Gamma}(-\g_Nq)\frac{1}{2}\frac{d}{dt}|\t^6\Psi\cdot \tilde{n}|^2\,dx'
+\int_{\Gamma}\g_Nq\t^6\Psi\cdot \tilde{n} \ \t^6\Psi\cdot \tilde{n}_t
-\sum_{l=1}^6a_l\int_{\Gamma}(-\g_Nq)\t^6\Psi\cdot \tilde{n}\t^{6-l}(v-w)\cdot\t^l\tilde{n} \label{eq:Rh}.
\end{align}
Recall that $\tilde{n}=J^{-1}(RN-R_{\theta}\tau)=J^{-1}(N+hN-h_{\theta}\tau)$.
Thus, using $\Psi(t,\xi)=N+h(t,\xi)N$, we obtain via the Leibniz rule
\begin{align*} 
\t^6\Psi\cdot \tilde{n} &= \left[  \t^6N+\t^6(hN) \right] \cdot \left[ (1+h)N-h_{\theta}\tau \right] J^{-1}\\
& =(-R+R\t^6h  +\sum_{a=0}^5c_a\t^ah\t^{6-a}N\cdot(hN-h_{\theta}\tau))J^{-1}\\
& = (-R_J+R_J\t^6h +\sum_{a=0}^5c_a^J\t^ah\t^{6-a}N\cdot(hN-h_{\theta}\tau)),
\end{align*} 
where we have used the relations $\t^2N=-N$ and $N\cdot \tau=0$ and also denoted $c_a^J=c_aJ^{-1}$ (recall $R_J=RJ^{-1}$).
From here we obtain
\be\label{eq:psin}
\begin{array}{l}
\displaystyle
\frac{1}{2}\frac{d}{dt}\big(|\t^6\Psi\cdot \tilde{n}|^2\big)
=\frac{1}{2}\frac{d}{dt}\big(R_J^2|\t^6h|^2\big)\\
\displaystyle
+\frac{d}{dt}\big[\big(-R_J+2R_J\t^6h 
+\sum_{a=0}^5c_a^J\t^ah\t^{6-a}N\cdot(hN-h_{\theta}\tau)\big)
\big(-R_J+\sum_{a=0}^5c_a^J\t^ah\t^{6-a}N\cdot(hN-h_{\theta}\tau)\big)\big].
\end{array}
\ee
Thus, going back to~(\ref{eq:Rh}), we obtain
\beas
\int_{\Gamma}(-\g_Nq)\t^6\Psi\cdot \tilde{n}\t^6\Psi_t\cdot \tilde{n}
&=&\frac{1}{2}\frac{d}{dt}\int_{\Gamma}(-\g_Nq)R_J^2|\t^6h|^2
+\frac{1}{2}\int_{\Gamma}\g_Nq_tR_J^2|\t^6h|^2
+\int_{\Gamma}\mathcal{T}_2,
\eeas
where the error term $\mathcal{T}_2$ is given by
\[
\mathcal{T}_2=
(-\g_Nq)\frac{d}{dt}\big[\big(-R_J+2R_J\t^6h 
+\sum_{a=0}^5c_a^J\t^ah\t^{6-a}N\cdot(hN-h_{\theta}\tau)\big)
\big(-R_J+\sum_{a=0}^5c_a^J\t^ah\t^{6-a}N\cdot(hN-h_{\theta}\tau)\big)\big].
\]
As to the second term on the right-hand side of~(\ref{eq:new1}), note that
\beas
A^k_i\t^6v^i_{,k}&=&\t^6(A^k_iv^i_{,k})-\sum_{l=1}^6c_l\t^lA^k_i\t^{6-l}v^i_{,k}
=-\t^6(q_t+v\cdot w)-\sum_{l=1}^6c_l\t^lA^k_i\t^{6-l}v^i_{,k},
\eeas
where $A^k_iv_{,k}^i=-\div_{\Psi}v=-(q_t+v\cdot w)$ by the parabolic equation~(\ref{eq:ALEheat}).
Thus
\begin{align}
&-\int_{\Omega}\mu A^k_i\t^6v_{,k}^i\big(\t^6q+\t^6\Psi\cdot v\big) =
\int_{\Omega}\mu\t^6(q_t+\Psi_{t}\cdot v)\big(\t^6q+\t^6\Psi\cdot v\big)
+\sum_{l=1}^6c_l\int_{\Omega}\mu\t^lA^k_i\t^{6-l}v_{,k}^i\big(\t^6q+\t^6\Psi\cdot v\big) \nonumber \\
& \ \ \ =\frac{1}{2}\frac{d}{dt}\int_{\Omega}\mu\big(\t^6q+\t^6\Psi\cdot v\big)^2
+\int_{\Omega}\mu(\sum_{l=1}^6d_l\t^{6-l}\Psi_{t}\cdot\t^lv-\t^6\Psi\cdot v_t)
\big(\t^6q+\t^6\Psi\cdot v\big) \nonumber \\
& \ \ \ +\sum_{l=1}^6c_l\int_{\Omega}\mu\t^lA^k_i\t^{6-l}v^i_{,k}\big(\t^6q+\t^6\Psi\cdot v\big) \label{eq:new2}
\end{align}
Combining~(\ref{eq:new1}) -~(\ref{eq:new2}) we obtain
\be\label{eq:en12}
\begin{array}{l}
\displaystyle
\int_{\Omega}\mu|\t^6v|^2\,dx+
\frac{1}{2}\frac{d}{dt}\int_{\Gamma}(-\g_Nq)R_J^2|\t^6h|^2\,dx'
+\frac{1}{2}\frac{d}{dt}\int_{\Omega}\mu(\t^6q+\t^6\Psi\cdot v)^2\,dx\\
\displaystyle
=-\frac{1}{2}\frac{d}{dt}\int_{\Gamma}(-\g_Nq)
|\sum_{a=0}^5c_a\t^aR\t^{6-a}N\cdot \tilde{n}|^2\,dx'+
\int_{\Omega}\mathcal{R}_0+\int_{\Gamma}\mathcal{G}_0
\end{array} 
\ee
with the error terms $\mathcal{R}_0$ and $\mathcal{G}_0$ given by 
\begin{align}
\mathcal{R}_0=&
\mu\sum_{l=1}^5c_l\t^{l}A^k_i\t^{6-l}q_{,k}\t^6v^i
+(\mu q_{,k}A^s_iA^k_r),_s\t^6\Psi^r\t^6v^i+\mu\{\t^6,A^k_i\}q_{,k}\t^6v^i \nonumber\\
&+\mu A^s_i\t^6\Psi^rA^k_rq,_k\{\t^6,\g_s\}v^i
-(\mu A^k_i),_k\t^6q\t^6v^i\nonumber\\
&-\mu A^k_i\{\t^6,\g_k\}q\t^6v^i-\mu\sum_{l=1}^6\big(c_l\t^lA^k_i\t^{6-l}v^i_{,k}
+d_l\t^{6-l}w\cdot\t^lv-\t^6\Psi\cdot v_t\big)\big(\t^6q+\t^6\Psi\cdot v\big);\label{eq:remainder}
\end{align}
\be
\label{eq:remaindergamma}
\begin{array}{l}
\displaystyle
\mathcal{G}_0=
-\g_Nq\t^6\Psi\cdot \tilde{n}\t^6\Psi\cdot \tilde{n}_t
+(-\g_Nq)\frac{d}{dt}\Big[R_J\t^6h\big(
-R_J+\sum_{a=0}^5c_a^J\t^ah\t^{6-a}N\cdot(hN-h_{\theta}\tau)\big)\Big]\\
\displaystyle
-\g_Nq\frac{d}{dt}\Big[\big(-R_J+\sum_{a=0}^5c_a^J\t^ah\t^{6-a}N\cdot(hN-h_{\theta}\tau)\big)^2\Big]
+\sum_{l=1}^6a_l(-\g_Nq)\t^6\Psi\cdot \tilde{n}\t^{6-l}(v-w)\cdot\t^l\tilde{n}.
\end{array}
\ee
Let now $\alpha=(\alpha_1,\alpha_2)$ be an arbitrary multi-index of order $6$.
Applying the operator $(1-\mu)\g^{\alpha}$ to~(\ref{eq:ALEv}) and multiplying 
by $\g^{\alpha}v^i$, we obtain
\beas
&&\big((1-\mu)\g^{\alpha}v^i
+(1-\mu)\g^{\alpha}A^k_iq,_k
+(1-\mu)A^k_i\g^{\alpha}q,_k,\,\g^{\alpha}v^i\big)_{L^2(\Omega)}\\
&&=
-\sum_{0<\beta\leq\alpha}c_{\beta}\big((1-\mu)\g^{\beta}A^k_i\g^{\alpha-\beta}q_{,k},\,
\g^{\alpha}v^i\big)_{L^2}
\eeas
In the same way as above we arrive at the following energy identity
\be\label{eq:en12tilde}
\int_{\Omega}(1-\mu)|\g^{\alpha}v|^2\,dx
+\frac{1}{2}\frac{d}{dt}\int_{\Omega}(1-\mu)(\g^{\alpha}q_t+\g^{\alpha}\Psi\cdot v)^2\,dx
=\int_{\Omega}\tilde{\mathcal{R}}_0\,dx,
\ee
where 
\be\label{eq:r1tilde}
\begin{array}{l}
\displaystyle 
\tilde{\mathcal{R}}_0=
(1-\mu)\sum_{0<\beta<\alpha}c_{\beta}\g^{\beta}A^k_i\g^{\alpha-\beta}q_{,k}\g^{\alpha}v^i
+(1-\mu)\Big((q_{,k}A^s_iA^k_r),_s\g^{\alpha}\Psi^r\g^{\alpha}v^i
+\{\g^{\alpha},A^k_i\}q_{,k}\g^{\alpha}v^i\Big)\\
\displaystyle
-(1-\mu)\sum_{0<\beta\leq\alpha}\big(c_{\beta}\g^{\beta}A^k_i\g^{\alpha-\beta}v_{,k}^i
+d_{\beta}\g^{\alpha-\beta}w\cdot\g^{\beta}v-\g^{\alpha}\Psi\cdot v_t\big)\big(\g^{\alpha}q+\g^{\alpha}
\Psi\cdot v\big),
\end{array}
\ee
Summing the identities~(\ref{eq:en12}) and~(\ref{eq:en12tilde}), with $j=0$ we arrive at 
\begin{align} 
&\int_{\Omega}\mu|\t^{6-2j}\g_t^jv|^2+\sum_{|\alpha|=6-2j}\int_{\Omega}(1-\mu)|\g^{\alpha}\g_t^jv|^2
+\frac{1}{2}\frac{d}{dt}\int_{\Gamma}(-\g_Nq)R_J^2\big|\t^{6-2j}\g_t^jh\big|^2 \nonumber  \\
&  
+\frac{1}{2}\frac{d}{dt}\int_{\Omega}\mu\big(\t^{6-2j}\g_t^jq+\t^{6-2j}\g_t^j\Psi\cdot v\big)^2
+\frac{1}{2}\frac{d}{dt}\sum_{|\alpha|=6-2j}\int_{\Omega}(1-\mu)(\g^{\alpha}\g_t^jq+\g^{\alpha}\g_t^j\Psi\cdot v)^2 \nonumber \\
&  =
-\int_{\Gamma}(-\g_Nq_t)R_J^2|\t^{6-2j}\g_t^jh|^2
+\int_{\Omega}(\mathcal{R}_j
+\tilde{\mathcal{R}}_j)
+\int_{\Gamma}\mathcal{G}_j \,, \label{eq:uno}
\end{align} 
By imitating the same calculation as above 
we obtain the remaining error terms. 
With $j=1,2,3$ the formulas for $\mathcal{R}_j$, $\tilde{\mathcal{R}}_j$, and $\mathcal{G}_j$ in~(\ref{eq:uno}) read
\begin{align}
\mathcal{R}_j=&
\sum_{0<a+b<6-j}\mu d_{ab}\t^a\g_t^bA^k_i\t^{6-2j-a}\g_t^{j-b}q_{,k}\t^{6-2j}\g_t^jv^i
+(\mu q_{,k}A^s_iA^k_r),_s\t^{6-2j}\g_t^j\Psi^r\t^{6-2j}\g_t^jv^i\nonumber\\
&+\mu A^s_i\t^{6-2j}\g_t^j\Psi^rA^k_rq,_k\{\t^{6-2j},\g_s\}\g_t^jv^i
+\mu\{\t^{6-2j}\g_t^j,A^k_i\}q_{,k}\t^{6-2j}\g_t^jv^i\nonumber\\
&-(\mu A^k_i),_k\t^{6-2j}\g_t^jq\t^{6-2j}\g_t^jv^i
-\mu A^k_i\{\t^{6-2j},\g_k\}\g_t^jq\t^{6-2j}\g_t^jv^i\nonumber\\
&-\mu\sum_{0\leq a+b<6-j}\big(d_{ab}\t^{6-2j-a}\g_t^{j-b}A^k_i\t^a\g_t^bv^i_{,k}
+d_{ab}\t^a\g_t^b\Psi_t\cdot\t^{6-2j-a}\g_t^{j-b}v-\t^{6-2j}\g_t^j\Psi\cdot v_t\big)\nonumber\\
&\quad\times\big(\t^{6-2j}\g_t^jq+\t^{6-2j}\g_t^j\Psi\cdot v\big);\label{eq:remainderj}
\end{align}
\begin{align}
\mathcal{G}_j & =
-\g_Nq\t^{6-2j}\g_t^j\Psi\cdot \tilde{n}\t^{6-2j}\g_t^j\Psi\cdot \tilde{n}_t 
+\g_Nq\g_t\Big[\t^{6-2j}\g_t^jhR_J(-R_J+\sum_{a=0}^{5-2j}d_a^J\t^a\g_t^jh\t^{6-2j-a}N\cdot \tilde{n})\Big] \nonumber \\
& \ \ +\frac{d}{dt}\Big[\big(\sum_{a=0}^{5-2j}d_a^J\t^a\g_t^jh\t^{6-2j-a}N\cdot \tilde{n}\big)^2\Big] \nonumber \\
& \ \ +\sum_{0\leq a+b<6-j}d_{ab}(-\g_Nq)\t^{6-2j}\g_t^j\Psi\cdot \tilde{n}\t^{6-2j-a}\g_t^{j-b}(v-w)\cdot\t^a\g_t^b\tilde{n} \label{eq:remaindergammaj}.
\end{align}
\begin{align}
& \tilde{\mathcal{R}}_j=
(1-\mu)\sum_{0<\beta<\alpha}c_{\beta}\g^{\beta}A^k_i\g^{\alpha-\beta}q_{,k}\g^{\alpha}v^i
+(1-\mu)\tilde{\mathcal{T}_1} \nonumber \\
& \ \ \ -(1-\mu)\sum_{0<\beta\leq\alpha}\big(c_{\beta}\g^{\beta}A^k_i\g^{\alpha-\beta}v^i_{,k}
+(1-\mu)d_{\beta}\g^{\alpha-\beta}w\cdot\g^{\beta}v-\g^{\alpha}\Psi\cdot v_t\big)\big(\g^{\alpha}q+\g^{\alpha} \Psi\cdot v\big). \label{eq:rjtilde}
\end{align}

\subsection{Identities of the second type}
Applying $\t^5\g_t$ to~(\ref{eq:ALEv}) and
computing the $L^2(\Omega)$-product with
$\mu\t^5v^i$ we obtain
\[
\big(\mu\t^5v^i_t+\mu\t^5A^k_{i,t}q_{,k}+\mu A^k_i\t^5q_{,kt},\,\t^5v^i\big)_{L^2}=
\sum_{0<a+b<6\atop a\leq5,\,b\leq1}c_{ab}\big(\mu\t^a\g_t^bA^k_i\t^{5-a}\g_t^{1-b}q_{,k},\,\t^5v^i\big)_{L^2},
\]
where $c_{ab}$ are constants appearing due to the usage of Leibniz product rule above.
Recalling~(\ref{eq:comm}), we write
\[
\t^5A^k_{i,t}=-A^s_i\t^5\Psi^r_{,st}A^k_r+\{\t^5\g_t,A^k_i\},
\]
where $\{\t^5\g_t,A^k_i\}$ stands for the lower order commutator defined in~(\ref{eq:comm}).
With this identity, we obtain
\be\label{eq:s2}
\begin{array}{l}
\displaystyle
\big(\mu\t^5A^k_{i,t}q_{,k},\t^5v^i\big)_{L^2(\Omega)}
=-\big(\mu A^s_i\t^5\Psi^r_{,st}A^k_rq_{,k},\t^5v^i\big)_{L^2(\Omega)}
+\big(\mu\{\t^5\g_t,A^k_i\}q_{,k},\t^5v^i\big)_{L^2(\Omega)}\\
\displaystyle
=-\int_{\Gamma}q_{,k}A^s_i\t^5\Psi^r_tA^k_r\t^5v^iN^s
+\int_{\Omega}\mu A^s_i\t^5\Psi^r_tA^k_rq,_k\t^5v^i_{,s}
+\int_{\Omega}\mathcal{U}_1\\
\displaystyle
=-\int_{\Gamma}q_{,k}A^s_i\t^5\Psi^r_tA^k_r\t^5v^iN^s
-\int_{\Omega}\mu A^s_i\t^5\Psi^r_tv^r\t^5v^i_{,s}
+\int_{\Omega}\mathcal{U}_1,
\end{array}
\ee
where we have integrated-by-parts with respect to $x^s$ in the second equality and
we have also used the identity $v^r=-A^k_rq_{,k}$ to write the last
line more concisely. The error term $\mathcal{U}_1$ is given by
\[
\mathcal{U}_1=
(\mu q_{,k}A^s_iA^k_r),_s\t^5\Psi^r_t\t^5v^i
+\mu\{\t^5\g_t,A^k_i\}q_{,k},\t^5v^i+\mu A^s_i\t^5\Psi^r_tA^k_rq,_k\{\t^5,\g_s\}v^i.
\]
Furthermore, integrating by parts with respect to $x^k$
\be\label{eq:s3}
\begin{array}{l}
 \displaystyle
\big(\mu A^k_i\t^5\g_tq_{,k},\,\t^5v^i\big)_{L^2}
=\int_{\Omega}\mu A^k_i\g_k\t^5\g_tq\t^5v^i+\int_{\Omega}\mu A^k_i\{\t^5,\g_k\}q_t\t^5v^i\\
\displaystyle
=-\int_{\Omega}\mu A^k_i\t^5q_t\t^5v^i_{,k}
-\int_{\Omega}(\mu A^k_i),_k\t^5q_t\t^5v^i+\int_{\Omega}\mu A^k_i\t^5q_t\{\t^5,\g_k\}v^i
+\int_{\Omega}\mu A^k_i\{\t^5,\g_k\}q_t\t^5v^i,
\end{array}
\ee
where we have used $\t^5q_t=0$ on $\Gamma$.
Summing~(\ref{eq:s2}) and~(\ref{eq:s3}), we obtain
\be\label{eq:snew1}
\begin{array}{l}
\displaystyle
\big(\mu\t^5A^k_{i,t}q_{,k}+\mu A^k_i\t^5q_{,kt},\t^5v^i\big)_{L^2(\Omega)}
=-\int_{\Gamma}q_{,k}A^s_i\t^5\Psi_t^rA^k_r\t^5v^iN^s
-\int_{\Omega}\mu A^k_i\t^5v_{,k}^i\big(\t^5q_t+\t^5\Psi_t\cdot v\big)\\
\displaystyle
+\int_{\Omega}\big(\mathcal{U}_1-(\mu A^k_i),_k\t^5q_t\t^5v^i+\mu A^k_i\t^5q_t\{\t^5,\g_k\}v^i
+\mu A^k_i\{\t^5,\g_k\}q_t\t^5v^i\big).
\end{array}
\ee
The first two terms on the right-hand side of~(\ref{eq:snew1}) will
be the source of positive definite quadratic contributions to the energy. 
To extract the quadratic coercive contribution from the first integral on the right-hand side 
of~(\ref{eq:snew1}), note that $q,_k=N^k\g_Nq$ on $\Gamma$. Thus
\[
-\int_{\Gamma}q_{,k}A^s_i\t^5\Psi_t^rA^k_r\t^5v^iN^s
=\int_{\Gamma}(-\g_Nq)\t^5\Psi_t^r\tilde{n}^r\t^5v^i\tilde{n}^i
=\int_{\Gamma}(-\g_Nq)\t^5\Psi_t\cdot \tilde{n}\t^5v\cdot \tilde{n}.
\]
Just like in Section~\ref{se:Aone} - as in the identities leading up to~(\ref{eq:psin}) - we obtain
\begin{align}
& (-\g_Nq)\t^5\Psi_t\cdot \tilde{n}\t^5v\cdot \tilde{n}
=|\t^5h_t|^2R_J^2+|\sum_{a=0}^4c_a^J\t^ah_t\t^{5-a}N\cdot \tilde{n}|^2
+2\t^5h_tR_J\sum_{a=0}^4c_a^J\t^ah_t\t^{5-a}N\cdot \tilde{n} \nonumber \\
& \ \ +\sum_{l=1}^4a_l(-\g_Nq)\t^5\Psi\cdot \tilde{n}\t^5v\cdot \tilde{n}\t^{5-l}(v-w)\cdot\t^l\tilde{n},
\end{align}
where $c_a^J=c_aJ^{-1}$ and $c_a$ are some universal constants.
As to the second term on the right-hand side of~(\ref{eq:snew1}), note that
\[
A^k_i\t^5v^i_{,k}=\t^5(A^k_iv^i_{,k})-\sum_{l=1}^5c_l\t^lA^k_i\t^{5-l}v_{,k}^i
=-\t^5(q_t+v\cdot w)-\sum_{l=1}^5c_l\t^lA^k_i\t^{5-l}v_{,k}^i,
\]
where $A^k_iv_{,k}^i=-\div_{\Psi}v=-(q_t+v\cdot w)$ by the parabolic equation~(\ref{eq:ALEheat}).
Thus
\be\label{eq:snew2}
\begin{array}{l}
\displaystyle
-\int_{\Omega}\mu A^k_i\t^5v_{,k}^i\big(\t^5q_t+\t^5\Psi_t\cdot v\big)\\
\displaystyle
=\int_{\Omega}\mu\t^5(q_t+\Psi_{t}\cdot v)\big(\t^5q_t+\t^5\Psi_t\cdot v\big)
+\sum_{l=1}^5c_l\int_{\Omega}\mu\t^lA^k_i\t^{5-l}v_{,k}^i\big(\t^5q_t+\t^5\Psi_t\cdot v\big)\\
\displaystyle
=\int_{\Omega}\mu\big(\t^5q_t+\t^5\Psi_t\cdot v\big)^2
+\sum_{l=1}^5d_l\int_{\Omega}\mu\t^{5-l}\Psi_{t}\cdot\t^lv
\big(\t^5q_t+\t^5\Psi_t\cdot v\big)\\
\displaystyle
\quad+\sum_{l=1}^5c_l\int_{\Omega}\mu\t^lA^k_i\t^{5-l}v_{,k}^i\big(\t^5q_t+\t^5\Psi_t\cdot v\big)
\end{array}
\ee
Combining~(\ref{eq:snew1}) -~(\ref{eq:snew2}) we obtain
\be\label{eq:sen12}
\begin{array}{l}
\displaystyle
\frac{1}{2}\frac{d}{dt}\int_{\Omega}\mu|\t^5v|^2\,dx+\int_{\Gamma}(-\g_Nq)|\t^5\Psi_t\cdot \tilde{n}|^2\,dx'
+\int_{\Omega}\mu(\t^5q_t+\t^5\Psi_t\cdot v)^2\,dx\\
\displaystyle 
=\int_{\Gamma}\g_Nq|\sum_{a=0}^4c_a^J\t^ah_t\t^{5-a}\xi\cdot \tilde{n}|^2
+\int_{\Omega}\mathcal{S}_1+\int_{\Gamma}\mathcal{H}_1
\end{array} 
\ee
with the error terms $\mathcal{S}_j$ and $\mathcal{H}_j$ given by: 
\begin{align}
\mathcal{S}_j=&
\sum_{0<a+b<6-j\atop a\leq7-2j,\,b\leq j}d_{ab}\mu\t^a\g_t^bA^k_i\t^{7-2j-a}
\g_t^{j-b}q_{,k}\t^{7-2j}\g_t^{j-1}v^i
+(\mu q_{,k}A^s_iA^k_r),_s\t^{7-2j}\g_t^j\Psi^r\t^{7-2j}\g_t^{j-1}v^i\nonumber\\
&+\mu\{\t^{7-2j}\g_t^j,A^k_i\}q_{,k},\t^{7-2j}\g_t^{j-1}v^i+\mu A^s_i\t^{7-2j}\g_t^j\Psi^rA^k_rq,_k\{\t^{7-2j},\g_s\}\g_t^{j-1}v^i\nonumber\\
&-(\mu A^k_i),_k\t^{7-2j}\g_t^jq\t^{7-2j}\g_t^{j-1}v^i
+\mu A^k_i\t^{7-2j}\g_t^jq\{\t^{7-2j},\g_k\}\g_t^{j-1}v^i\nonumber\\
&+\mu A^k_i\{\t^{7-2j},\g_k\}\g_t^jq\t^{7-2j}\g_t^{j-1}v^i
-\big(\t^{7-2j}\g_t^jq+\t^{7-2j}\g_t^j\Psi\cdot v\big)\nonumber\\
&\times
\sum_{0\leq a+b<6-j}d_{ab}\big(\mu\t^a\g_t^b\Psi_{t}\cdot\t^{7-2j-a}
\g_t^{j-1-b}v
+\mu\t^{7-2j-a}\g_t^{j-1-b}A^k_i
\t^{a}\g_t^bv^i_{,k}\big).\label{eq:sremainder}
\end{align}
\begin{align}
\mathcal{H}_j=&
2\g_Nq\t^{7-2j}\g_t^jhR_J\sum_{0\leq a+b<7-j}c_{ab}^J\t^a\g_t^{b}h
\t^{7-2j-a}\g_t^{j-b}N\cdot \tilde{n}\nonumber\\
&+(-\g_Nq)\t^{7-2j}\g_t^{j}\Psi
\cdot \tilde{n}\sum_{0\leq a+b<6-j}d_{ab}\t^{7-2j-a}\g_t^{j-1-b}(v-w)
\cdot\t^a\g_t^b\tilde{n}.\label{eq:sremaindergamma}
\end{align}
Note that the first line of~(\ref{eq:sremaindergamma}) appears as an expanded difference
between two positive definite expressions $(-\g_Nq)|\t^{7-2j}\g_t^j\Psi\cdot \tilde{n}|^2$ and
$(-\g_Nq)|\t^{7-2j}\g_t^jh|^2$. We do this just like after~(\ref{eq:Rh}) using the formula
$\tilde{n}=J^{-1}(N+hN-h_{\theta}\tau)$ and the parametrization $\Psi(t,\xi)=(1+h(t,\xi))N$.
Let now $\alpha=(\alpha_1,\alpha_2)$ be an arbitrary multi-index of order $5$.
Applying the operator $(1-\mu)\g^{\alpha}\g_t$ to~(\ref{eq:ALEv}) and multiplying 
by $\g^{\alpha}v^i$, we obtain
\beas
&&\big((1-\mu)\g^{\alpha}v^i_t
+(1-\mu)\g^{\alpha}A^k_{i,t}q,_k
+(1-\mu)A^k_i\g^{\alpha}q,_{kt},\,\g^{\alpha}v^i\big)_{L^2(\Omega)}\\
&&=
-\sum_{0<|\beta|+b<5\atop \beta\leq\alpha;b\leq1}c_{\beta b}\big((1-\mu)\g^{\beta}\g_t^bA^k_i\g^{\alpha-\beta}\g_t^{1-b}q_{,k},\,
\g^{\alpha}v^i\big)_{L^2}.
\eeas
In the same way as above we arrive at the following energy identity
\be\label{eq:sen12tilde}
\frac{1}{2}\frac{d}{dt}\int_{\Omega}(1-\mu)|\g^{\alpha}v|^2\,dx
+\int_{\Omega}(1-\mu)(\g^{\alpha}q_t+\g^{\alpha}\Psi_t\cdot v)^2\,dx
=\int_{\Omega}\tilde{\mathcal{S}}_1\,dx,
\ee
where 
\be\label{eq:s1tilde}
\begin{array}{l}
\displaystyle 
\tilde{\mathcal{S}}_1=
\sum_{0<|\beta|+b<5\atop \beta\leq\alpha;b\leq1}c_{\beta b}(1-\mu)\g^{\beta}\g_t^bA^k_i\g^{\alpha-\beta}\g_t^{1-b}q_{,k}
\g^{\alpha}v^i
+(1-\mu)\Big((q_{,k}A^s_iA^k_r),_s\g^{\alpha}\Psi^r_t\g^{\alpha}v^i
+\{\g^{\alpha}\g_t,A^k_i\}q_{,k}\g^{\alpha}v^i\Big)\\
\displaystyle
-(1-\mu)\sum_{0<\beta\leq\alpha}c_{\beta}\mu\g^{\alpha-\beta}\Psi_{t}\cdot\t^lv
\big(\t^{\alpha}q_t+\t^{\alpha}\Psi_t\cdot v\big)
-(1-\mu)\sum_{\beta\leq\alpha}c_l\mu\g^{\beta}A^k_i\t^{\alpha-\beta}v^i_{,k}\big(\t^{\alpha}q_t+\t^{\alpha}\Psi_t\cdot v\big),
\end{array}
\ee
For a general $j\in\{1,2,3\}$ we have
\be\label{eq:sjtilde}
\begin{array}{l}
\displaystyle 
\tilde{\mathcal{S}}_j=
\sum_{0<|\beta|+b<7-2j\atop \beta\leq\alpha;b\leq j}
d_{\beta b}(1-\mu)\g^{\beta}\g_t^bA^k_i\g^{\alpha-\beta}\g_t^{j-b}q_{,k}
\g^{\alpha}v^i
+(1-\mu)\tilde{\mathcal{U}}_1\\
\displaystyle
-(1-\mu)\sum_{0\leq|\beta|+b<|\alpha|}d_{\beta b}
\mu\g^{\beta}\g_t^{b}\Psi_{t}\cdot\g^{\alpha-\beta}\g_t^{j-1-b}v
\big(\t^{\alpha}\g_t^{j}q+\t^{\alpha}\g_t^j\Psi\cdot v\big)\\
\displaystyle
-(1-\mu)\sum_{0\leq|\beta|+b<|\alpha|}d_{\beta b}
\mu\g^{\alpha-\beta}\g_t^{j-1-b}A^k_i\t^{\beta}\g_t^bv^i_{,k}
\big(\t^{\alpha}\g_t^jq+\t^{\alpha}\g_t^j\Psi\cdot v\big).
\end{array}
\ee
Summing the identities~(\ref{eq:sen12}) and~(\ref{eq:sen12tilde}) we arrive at 
\be\label{eq:due}
\begin{array}{l}
\displaystyle
\frac{1}{2}\frac{d}{dt}\big\{\int_{\Omega}\mu|\t^{7-2j}\g_t^jv|^2+
\sum_{|\alpha|=7-2j}\int_{\Omega}(1-\mu)|\g^{\alpha}\g_t^jv|^2\big\}
+\int_{\Gamma}(-\g_Nq)R_J^2\big|\t^{7-2j}\g_t^jh\big|^2\\
\displaystyle
+\int_{\Omega}\mu\big(\t^{7-2j}\g_t^jq+\t^{7-2j}\g_t^j\Psi\cdot v\big)^2
+\sum_{|\alpha|=7-2j}\int_{\Omega}(1-\mu)(\g^{\alpha}\g_t^jq+
\g^{\alpha}\g_t^j\Psi\cdot v)^2\\
\displaystyle
=\int_{\Omega}(\mathcal{S}_j+\tilde{\mathcal{S}}_j)
+\int_{\Gamma}\mathcal{H}_j,\,\quad j=1,2,3.
\end{array}
\ee
Summing the identities~(\ref{eq:uno}) for $j=0,1,2,3$ and~(\ref{eq:due}) for $j=1,2,3$, we conclude the proof of Proposition~\ref{lm:lemma1}.


\section{Proof of the inequality (\ref{eq:barrier2})}\label{se:barrier}
We use the comparison function
$P$ defined in~(\ref{eq:barrier}) with the same $\kappa_2$ and $\kappa_1$.
Note that $\kappa_1=C_{*}\epsilon c_1$ is defined as a multiple of $c_1$ for
some constant $C_*>0$.
Using~(\ref{eq:neg}) and upon possibly enlarging $C_*$, we infer
$(\g_t-a_{ij}-b_i)(-q_t+P)\leq0$. Theorem 1 from~\cite{Od} guarantees
\be\label{eq:oddsonagain}
-q_t+P\leq C_0c_1\rho e^{(-\lambda_1+C\epsilon )t}, 
\ee
where $\rho(r)=1-r$ stands for the distance function to the boundary $\Gamma$. 
Note that the constant $\kappa_1$ in the definition~(\ref{eq:barrier}) is chosen right after~(\ref{eq:track}).
It is in particular proportional to $E_{\beta}(0)^{1/2}\leq \|q_0\|_4$. By definition of $K$ we have that $\|q_0\|_4\leq K\|q_0\|_0$.
Since however $\|q_0\|_0\leq Kc_1$, we obtain
$
P/\rho\leq C K^2 c_1e^{-3\lambda_1t/2}.
$
Similarly, the constant $C_0$ is proportional to the $L^{\infty}$-norm of the initial datum for $-q_t+P$, wherefrom
we again obtain $C_0\leq K^2 c_1$ by the same argument as above.
Dividing by $\rho$ in~(\ref{eq:oddsonagain}), from the above inequality, we infer that
\[
|\g_Nq_t|_{\infty}\leq CK^2c_1 e^{(-\lambda_1+C\epsilon )t}.
\]
This proves the inequality~(\ref{eq:barrier2}).
 \numberwithin{equation}{section}





 \section*{Acknowledgements}

MH was supported by the National Science Foundation under grant NSF DMS-1211517 and
SS was supported by the National
Science Foundation under grant DMS-1001850.
The authors would like to thank David Jerison for discussions leading to Lemma~\ref{lm:hardy}.
They also thank the anonymous referee for helpful comments.



\frenchspacing
\bibliographystyle{plain}

\begin{thebibliography}{99}
\bibitem{Lu}
Luckhaus, S.
Solutions for the two-phase Stefan problem with the Gibbs-Thomson law for the
melting temperature.
\textit{Europ. J. Appl. Math.}
\textbf{1} (1990), 101--111.
\end{thebibliography}

\end{document}